\documentclass{amsart}
\usepackage{enumitem}
\usepackage{amssymb}
\usepackage{amsthm}
\usepackage{amsmath}
\usepackage[scr=boondox]{mathalfa}
\usepackage[usenames]{color}
\usepackage{hyperref}

\theoremstyle{plain}
\newtheorem{proposition}{Proposition}[section]
\newtheorem{theorem}[proposition]{Theorem}
\newtheorem{lemma}[proposition]{Lemma}
\newtheorem{corollary}[proposition]{Corollary}
\theoremstyle{definition}
\newtheorem{example}[proposition]{Example}
\newtheorem{definition}[proposition]{Definition}

\theoremstyle{remark}
\newtheorem{remark}[proposition]{Remark}
\newtheorem{notation}[proposition]{Notation}

\newtheorem{question}[proposition]{Question}

\newtheorem*{claim}{Claim}

\DeclareMathOperator{\Id}{Id}
\DeclareMathOperator{\Cay}{Cay}
\DeclareMathOperator{\SL}{\mathsf{SL}}
\DeclareMathOperator{\GL}{\mathsf{GL}}
\DeclareMathOperator{\SO}{\mathsf{SO}}

\DeclareMathOperator{\PSL}{\mathsf{PSL}}

\DeclareMathOperator{\Hom}{Hom}

\DeclareMathOperator{\Hull}{Hull} 
 
\DeclareMathOperator{\rank}{rank} 
 
\DeclareMathOperator{\Gr}{Gr} 
\DeclareMathOperator{\id}{id} 
\DeclareMathOperator{\diag}{diag}

\DeclareMathOperator{\Isom}{Isom}

\DeclareMathOperator{\Span}{Span}

\DeclareMathOperator{\Bc}{\mathcal{B}}

\DeclareMathOperator{\Fc}{\mathcal{F}}
\DeclareMathOperator{\Gc}{\mathcal{G}}
\DeclareMathOperator{\Hc}{\mathcal{H}}

\DeclareMathOperator{\Lc}{\mathcal{L}}
\DeclareMathOperator{\Nc}{\mathcal{N}}

\DeclareMathOperator{\Pc}{\mathcal{P}}

\DeclareMathOperator{\Xc}{\mathcal{X}}

\DeclareMathOperator{\Cb}{\mathbb{C}}

\DeclareMathOperator{\Kb}{\mathbb{K}}
\DeclareMathOperator{\Nb}{\mathbb{N}}
\DeclareMathOperator{\Pb}{\mathbb{P}}
\DeclareMathOperator{\Rb}{\mathbb{R}}

\DeclareMathOperator{\Zb}{\mathbb{Z}}

\newcommand{\abs}[1]{\left|#1\right|}

\newcommand{\norm}[1]{\left\|#1\right\|}

\begin{document}

\title{Notions of Anosov representation of relatively hyperbolic groups}
\author{Tianqi Wang}
\address{National University of Singapore}
\email{twang@u.nus.edu}
\thanks{Wang was partially supported by the NUS-MOE grant A-8000-458-00-00 and by the Merlion PhD program 2021}

\date{\today}

\begin{abstract}
    We prove that divergent, extended geometrically finite (in the sense of \cite{Weisman}) representations can be interpreted as restricted Anosov (in the sense of \cite{TW23}) representations over certain flow spaces. We also show that the representations of this type are stable under small type preserving deformations. As an example, we show that a representation induced from a geometrically finite one through a Galois covering, constructed in \cite{TW23}, is divergent and extended geometrically finite with a non-homeomorphic boundary extension.
\end{abstract}

\maketitle

\setcounter{tocdepth}{1}
\tableofcontents

\section{Introduction}

The notion of Anosov representation of a hyperbolic group into a semisimple Lie group $G$ is introduced by Labourie \cite{Lab} and developed further by many researchers such as Guichard--Wienhard \cite{GW}, Kapovich--Leeb--Porti \cite{KLP16,KLP17,KLP18}, Gu\'eritaud--Guichard--Kassel--Wienhard \cite{GGKW}, and Bochi--Potrie--Sambarino \cite{BPS}. They provide many different ways to characterize Anosov representations such as by the contraction property of the geodesic flow on associated bundles, by the convergence group actions on the limit sets, and by the linear growth property of the Cartan projections.

For convenience, we consider the case when $G = \SL(d,\Kb)$ where $\Kb = \Rb$ or $\Cb$. Let $k\leqslant d/2$ be a positive integer and let \[\Fc_{k,d-k} = \{(V,W)\in\Gr_k(\Kb^d)\times\Gr_{d-k}(\Kb^d): V\subset W\}\] be the flag manifold. The results easily extend to the case of other non-compact semisimple Lie groups with finite centers (see for example \cite[Section 8]{BPS} and \cite[Appendix A]{CZZ} as references). We quickly recall these ways of characterizing Anosov representations. Let $\Gamma$ be a hyperbolic group and let $\rho$ be a representation of $\Gamma$ into $\SL(d,\Kb)$.

Since $\Gamma$ is hyperbolic, one can construct an associated geodesic flow $(U(\Gamma),\phi)$ following \cite{Gromov}, \cite{Matheus}, \cite{Cham} and \cite{Mineyev}. One may roughly understand $U(\Gamma)$ as the collection of bi-infinite geodesics parametrized by length in the Cayley graph of $\Gamma$ and $\phi$ the shift of the parameters.

The flow $\phi$ can be naturally lifted to the trivial bundle $U(\Gamma)\times \Kb^d$ by parallel transformations, which commutes with the $\Gamma$-action on $U(\Gamma)\times \Kb^d$ defined by $\gamma (z,v) = (\gamma z, \rho(\gamma)v)$. We say that $E_\rho = \Gamma\backslash(U(\Gamma)\times \Kb^d) = E^s_\rho \oplus E^u_\rho$, a decomposition of $E_\rho$ into $\phi$-invariant subbundles, is a dominated splitting of rank $k$ if $E^s_\rho$ is of rank $k$ and there exist constants $C,c > 0$, such that \[ \dfrac{\norm{\phi^t_z(v)}}{\norm{\phi^t_z(w)}} \leqslant C e^{-c t}\dfrac{\norm{v}}{\norm{w}}~,\] for any $z\in \Gamma\backslash U(\Gamma)$, $v\in (E^s_\rho)_z$, $w\in (E^u_\rho)_z$ nonzero vectors and $t\in\Rb_{\geqslant 0}$. Here $\norm{\cdot}$ is a metric on $E_\rho$, and since $\Gamma\backslash U(\Gamma)$ is compact, the definition is independent of the choice of such metric.
It is not hard to see that the condition is equivalent to requiring that the induced flow on $\Hom(E^u_\rho,E^s_\rho)$ (by $\phi^t$-conjugation) is contracting exponentially.

We say that $\rho$ is $P_k$-Anosov if the associated flat bundle $E_\rho$ admits a dominated splitting of rank $k$, referring to \cite{BPS}.

The second way to characterize Anosov representations by the convergence group actions on the limit sets (see for example \cite{Can}) is formulated in the following way. Denote the Gromov boundary of $\Gamma$ by $\partial \Gamma$. Then $\rho$ is $P_k$-Anosov if and only if there exists a pair of $\rho$-equivariant continuous maps $\xi=(\xi^k,\xi^{d-k}):\partial\Gamma \to \Gr_k(\Kb^d)\times Gr_{d-k}(\Kb^d)$ that is
\begin{itemize}
    \item compatible, that is, the image of $\xi$ lies in $\Fc_{k,d-k}$,
    \item transverse, that is, $\xi^k(p)\oplus \xi^{d-k}(q) = \Kb^d$ for any $p\ne q\in \partial \Gamma$,
    \item strongly dynamics preserving, that is, if $(\gamma_n)_{n\in \Nb}\subset \Gamma$ is a sequence with $\gamma_n\to p\in \partial \Gamma$ and $\gamma_n^{-1}\to q\in \partial \Gamma$ as $n\to +\infty$, then $\rho(\gamma_n)V \to \xi^k(p)$ for any $V\in \Gr_k(\Kb^d)$ transverse to $\xi^{d-k}(q)$.
\end{itemize} 

Finally, we may also characterize Anosov representations by the dominated property (linear growth of the Cartan projections) following \cite{KLP18} and \cite{BPS}. The representation $\rho$ is $P_k$-Anosov if and only if there exist constants $C,c>0$, such that \[\dfrac{\sigma_{k+1}(\rho(\gamma))}{\sigma_k(\rho(\gamma))}\leqslant Ce^{-c\abs{\gamma}}~,\] for any $\gamma\in \Gamma$, where $\abs{\cdot}$ denotes the word metric on $\Gamma$ with respect to a fixed generating set and $\sigma_k(\cdot)$ denotes the $k$-th singular value with respect to the standard metric on $\mathbb{K}^d$.
\\

One may view the classical notion of Anosov representation of hyperbolic groups as a generalization of convex-cocompact representation into higher rank semisimple Lie groups. A recent trend is to study notions of Anosov representation of relatively hyperbolic groups, which are, in the same sense as the classical one, trying to generalize the notion of geometrically finite representation into  higher rank semisimple Lie groups. This induces for instance the cusped Hitchin representations introduced by Canary--Zhang--Zimmer \cite{CZZ}, the relatively dominated and relatively Anosov representations introduced by Zhu \cite{Zhu} and Zhu--Zimmer \cite{ZZ1}, the relatively Morse and relatively asymptotically embedded representations introduced by Kapovich--Leeb \cite{KL}, and the extended geometrically finite representations introduced by Weisman \cite{Weisman}.

We now assume that $(\Gamma,\Pc)$ is a relatively hyperbolic group and $\rho:\Gamma\to \SL(d,\Kb)$ is a representation. Recall the following definition. 
\begin{definition}
    A representation $\rho:\Gamma\to \SL(d,\Kb)$ is \emph{$P_k$-divergent} if \[\lim\limits_{n\to +\infty} \dfrac{\sigma_{k}(\rho(\gamma_n))}{\sigma_{k+1}(\rho(\gamma_n))}=+\infty~,\] for any sequence of pairwise distinct elements $(\gamma_n)_{n\in \Nb}\subset\Gamma$.
\end{definition}

One may define Anosov representations of $(\Gamma,\Pc)$ by replacing the Gromov boundary of the hyperbolic group in the classical case by the Bowditch boundary $\partial(\Gamma,\Pc)$. More concretely, following Zhu--Zimmer \cite{ZZ1}, $\rho$ is said to be $P_k$-Anosov relative to $\Pc$ if there exists a pair of continuous, $\rho$-equivariant maps $\xi=(\xi^k,\xi^{d-k}):\partial(\Gamma,\Pc)\to \Fc_{k,d-k}$ that is transverse and strongly dynamics preserving (see Definition \ref{DefRelativelyAnosov}). A slightly different notion is that of $P_k$-asymptotically embedded representation introduced by Kapovich--Leeb \cite{KL}. Instead of the dynamics preserving property, it requires the limit map $\xi$ to be a homeomorphism between $\partial(\Gamma,\Pc)$ and its image in the flag manifold and requires $\rho$ to be $P_k$-divergent (see Definition \ref{DefAsymptoticallyEmbedded}). These two definitions are equivalent referring to \cite[Proposition 4.4]{ZZ1}.

Another further generalization of Anosov representation of relatively hyperbolic groups is introduced by Weisman \cite{Weisman}, namely the extended geometrically finite representations. 
Comparing to the relatively Anosov and relatively asymptotically embedded representations, we consider continuous projections (rather than homeomorphisms) from a subset of $\Fc_{k,d-k}$ onto $\partial(\Gamma,\Pc)$. More concretely, a boundary extension of $(\Gamma,\Pc)$ in $\Fc_{k,d-k}$ is a closed $\rho(\Gamma)$-invariant subset $\Lambda$ of $\Fc_{k,d-k}$, together with a continuous, $\rho$-equivariant, surjective map $\zeta:\Lambda\to \partial(\Gamma,\Pc)$. The representation $\rho$ is said to be extended geometrically finite relative to $\Pc$, if there exists a boundary extension $\zeta:\Lambda\to \partial(\Gamma,\Pc)$ in $\Fc_{k,d-k}$ that is
\begin{itemize}
    \item transverse, i.e., $x,y\in \Lambda$ are transverse flags when $\zeta(x)\ne\zeta(y)$,
    \item extending the convergence dynamics, i.e., there exists a family $(C_p)_{p\in \partial (\Gamma, \Pc)}$ of open subsets of $\Fc_{k,d-k}$, such that $\Lambda\setminus \zeta^{-1}(p)\subset C_p$ for each $p\in \partial (\Gamma, \Pc)$, and if $(\gamma_n)_{n\in \Nb}$ is a sequence in $\Gamma$ with $\gamma_n\to p\in\partial(\Gamma,\Pc)$ and $\gamma_n^{-1}\to q\in \partial(\Gamma,\Pc)$ when $n\to +\infty$, then for any $x\in C_q$, $\rho(\gamma_n)x$ has all accumulation points lying in $\zeta^{-1}(p)$ as $n\to +\infty$, which is uniform on any compact subset of $C_q$.
\end{itemize}\ \\

On the other hand, Tholozan--Wang \cite{TW23} provides a purely abstract way to describe the dominated splitting property. As long as the given group $\Gamma$ is discrete, we may consider the setting of a flow space $(Y,\phi)$ with a properly discontinuous $\Gamma$-action on $Y$ that commutes with $\phi$. Then we say a representation $\rho:\Gamma\to \SL(d,\Kb)$ is $P_k$-Anosov in restriction to $Y$ if $E_\rho = \Gamma\backslash (Y \times \Kb^d)$ admits a dominated splitting of rank $k$ with respect to some metric on $E_\rho$. If the $\Gamma$-action on $Y$ is cocompact, the dominated splitting is independent of the choice of the metric, as any two metrics are always bi-Lipschitz to each other. However, the choice of the metric is important when the $\Gamma$-action is not cocompact, especially when we consider flow spaces associated to relatively hyperbolic groups.

One natural question to ask is whether there is a proper flow space such that we could interpret extended geometric finiteness as Anosov in restriction to such a flow space. We give an answer when the representation is divergent. When $(\Gamma, \Pc)$ is a relatively hyperbolic group and $X$ is a Gromov model of $(\Gamma, \Pc)$, we construct an associated flow space $\Fc(X)$ (see Notation~\ref{notation: compact conical limit flow}). Then for a given boundary extension $\zeta:\Lambda\to \partial(\Gamma,\Pc)$, one can extend $\Fc(X)$ to a flow space $\Fc(\Lambda,\zeta,X)$ (see Section \ref{BoundaryExtensionSection}). Roughly speaking, $\Fc(\Lambda,\zeta,X)$ is the collection of bi-infinite flow lines with transverse endpoints in $\Lambda$. Then we have the following equivalence.

\begin{theorem}\label{MainTheorem}
    Let $(\Gamma,\Pc)$ be a relatively hyperbolic group and let $\rho$ be a representation of $\Gamma$ into $\SL(d,\Kb)$, then the following are equivalent.
    \begin{itemize}
        \item[(1)] $\rho$ is $P_k$-divergent and extended geometrically finite relative to $\Pc$;
        \item[(2)] There exists a Gromov model $X$ of $(\Gamma,\Pc)$, and a minimal boundary extension $\zeta:\Lambda\to \partial(\Gamma,\Pc)$, such that $\rho$ is $P_k$-Anosov in restriction to $\Fc(\Lambda,\zeta,X)$.
    \end{itemize}
    Moreover, when (2) holds, $\rho$ is $P_k$-weakly dominated with respect to $X$.
\end{theorem}

The main reason that we study extended geometrically finite representations which are divergent is that they ``extend the convergence dynamics more canonical'', that is, we have a canonical choice of the boundary extension and a canonical choice of the family of open sets $(C_p)_{p\in \partial(\Gamma,\Pc)}$ in the definition. More concretely, as long as a representation $\rho:\Gamma\to \SL(d,\Kb)$ is $P_k$-divergent, it has a well-defined limit set \[\Lc(\rho) = \{ \lim\limits_{n\to +\infty} (U_k(\rho(\gamma_n)),U_{d-k}(\rho(\gamma_n))): (\gamma_n)_{n\in \Nb}\text{ a sequence of}\] \[\text{pairwise distinct elements in }\Gamma\}\subset \Fc_{k,d-k}~,\] where $U_k(g)$ denotes the eigenspace of $gg^t$ for the largest $k$ eigenvalues for a matrix $g\in\SL(d,\Kb)$ that has a singular values gap at $k$, i.e. $\sigma_k(g)>\sigma_{k+1}(g)$. If $\rho$ is moreover extended geometrically finite, then their is a canonical choice of the boundary extension in $\Fc_{k,d-k}$ given by $\Lc(\rho)$. More concretely,

\begin{theorem}\label{MainRefinedtoMinimal}
    If a representation $\rho:\Gamma\to \SL(d,\Kb)$ is $P_k$-divergent and extended geometrically finite relative to $\Pc$, then $\rho$ is extended geometrically finite with boundary extension $\zeta: \Lc(\rho) \to \partial (\Gamma, \Pc)$ defined such that 
    \begin{itemize}
        \item If $(\gamma_n)_{n\in \Nb}\subset \Gamma$ is a sequence with $\gamma_n\to q\in\partial_{con}(\Gamma,\Pc)$ then $\zeta^{-1}(q)$ is the singleton \[\lim\limits_{n\to +\infty} (U_k(\rho(\gamma_n)),U_{d-k}(\rho(\gamma_n)))~;\]
        \item For any $p\in\partial_{par}(\Gamma,\Pc)$, \[\zeta^{-1}(p) = \{ \lim\limits_{n\to +\infty} (U_k(\rho(\gamma_n)),U_{d-k}(\rho(\gamma_n))): (\gamma_n)_{n\in \Nb}\subset P\] \[\text{ a sequence of pairwise distinct elements}\}~.\]
    \end{itemize}
\end{theorem}
With this canonical boundary extension in $\Fc_{k,d-k}$, the family $(C_p)_{p\in \partial (\Gamma, \Pc)}$ of open subsets of $\Fc_{k,d-k}$ in the definition can be given by $C_p = \{x\in \Fc_{k,d-k}: x\text{ is transverse to }\zeta^{-1}(p)\}$ for each $p\in \partial (\Gamma, \Pc)$ (see Remark \ref{OpenSettobeWholeTransverse}).\\

Analogous to the cusped Hitchin representations studied by \cite{CZZ}, we show that this type of representations described in Theorem \ref{MainTheorem} is stable under ``small type preserving deformations''. For a representation $\rho_0$ of a relatively hyperbolic group $(\Gamma,\Pc)$ into $\SL(d,\Kb)$. We denote \[\Hom_{\Pc}(\rho_0) = \{\rho\in \Hom(\Gamma,\SL(d,\Kb)):\forall P\in\Pc, \rho|_P\text{ is conjugate to }\rho_0|_P\}\] the collection of representations that have the same ``parabolic type'' as $\rho_0$. Then we have the following theorem.

\begin{theorem}\label{MainStabilityTheorem}
    Let $(\Gamma,\Pc)$ be a relatively hyperbolic group and let $\rho_0:\Gamma\to\SL(d,\Kb)$ be a representation that is $P_k$-divergent and extended geometrically finite relative to $\Pc$. Then there exists an open neighborhood $O\subset\Hom_{\Pc}(\rho_0)$ of $\rho_0$ such that any $\rho\in O$ is $P_k$-divergent and extended geometrically finite relative to $\Pc$.
\end{theorem}

It was shown in \cite[Theorem 1.3]{ZZ1} that a representation is relatively Anosov if and only if there exists a pair of compatible, transverse and continuous limit maps that induces a dominated splitting over a flow space associated to some Gromov model.
As an application of Theorem \ref{MainTheorem}, we refine this result in two aspects. Firstly, we do not require the existence of limit maps in advance, using the concept of restricted Anosov representation. Secondly, together with another hyperbolic geometry fact Proposition \ref{RatioGrowthLinearLemma}, we show that the result holds for any Gromov model.

\begin{corollary}\label{LimitMapWaiveCor}
    Let $(\Gamma,\Pc)$ be a relatively hyperbolic group and let $\rho:\Gamma\to \SL(d,\Kb)$ be a representation. The following are equivalent.
    \begin{itemize}
        \item[(1)] $\rho$ is $P_k$-Anosov relative to $\Pc$;
        \item[(2)] $\rho$ is $P_k$-Anosov in restriction to $\Fc(X)$ for some Gromov model $X$ of $(\Gamma,\Pc)$;
        \item[(3)] $\rho$ is $P_k$-Anosov in restriction to $\Fc(X)$ for any Gromov model $X$ of $(\Gamma,\Pc)$;
    \end{itemize}
\end{corollary}\ \\

We provide a new example of divergent, extended geometrically finite representation, but not relatively Anosov, which is originally constructed by \cite{TW23} in the study of simple Anosov representations. Let $\Sigma$ be a closed hyperbolic surface and let $p: \widetilde{\Sigma}\to \Sigma$ be a Galois covering of degree $d$. Let $\Gamma$ (respectively, $\Gamma'$) be the fundamental group of $\Sigma$ (respectively, $\widetilde{\Sigma}$). There is a simple closed geodesic $c$ in $\widetilde{\Sigma}$ such that $p(c)$ has self-intersection (hence $p(c)$ has multiple of lifts in $\widetilde{\Sigma}$), and a geometrically finite representation $\rho_1: \Gamma' \to \SL(2,\Cb)$ where the parabolic elements are exactly the $\Gamma'$-conjugates of powers of the homotopy class of $c$. Let $\Pc$ be a finite collection of subgroups of $\Gamma'$ (hence of $\Gamma$) with each of them generated by the homotopy class of one of the lifts of $p(c)$.

\begin{proposition}
    With the assumptions above, $(\Gamma,\Pc)$ is a relatively hyperbolic group and the induced representation $\mathrm{Ind}^{\Gamma}_{\Gamma'}(\rho_1):\Gamma\to \SL(2d,\Cb)$ is $P_d$-divergent and extended geometrically finite relative to $\Pc$.  
\end{proposition}

To see that $(\Gamma,\Pc)$ is relatively hyperbolic and $\mathrm{Ind}^{\Gamma}_{\Gamma'}(\rho_1)$ has the desired properties, we will need the following two preliminary theorems. The first theorem is a generalization to relatively hyperbolic groups of a result for hyperbolic groups given by \cite{Bow12},\cite{Tran13},\cite{Man15}.

\begin{theorem}\label{IntroRelativeBoundaryExtension}
    Let $(\Gamma,\Hc)$ be a relatively hyperbolic group, if $\Pc$ is a finite, almost malnormal collection of finitely generated, infinite, relatively quasi-convex subgroups such that each $H\in\Hc$ is contained in some $P\in\Pc$.
    For each $P\in\Pc$, let $\Hc_P$ be a collection of representatives of $P$-conjugation orbits in \[\{\gamma H\gamma^{-1}: H\in\Hc,\gamma\in\Gamma,\text{ such that }\gamma H\gamma^{-1}\subset P\}~.\]
    Then $(\Gamma,\Pc)$ and $(P,\Hc_P)$ for each $P\in\Pc$, are relatively hyperbolic. Moreover, $\partial(P,\Hc_P)\subset \partial(\Gamma,\Hc)$ identifies with the limit set of $P$ in $\partial(\Gamma,\Hc)$, and there exists a $\Gamma$-equivariant quotient map $\tau:\partial(\Gamma,\Hc)\to \partial(\Gamma,\Pc)$ by exactly identifying $\gamma\partial(P,\Hc_P)$ to the unique parabolic point in $\partial(\Gamma,\Pc)$ fixed by $\gamma P\gamma^{-1}$ for each $P\in\Pc$ and $\gamma\in \Gamma$.
\end{theorem}

The second theorem states that the divergent extended geometrically finite property can pass between different relatively hyperbolic structures in the setting of Theorem \ref{IntroRelativeBoundaryExtension}.

\begin{theorem}\label{RelativeAnosovExtension}
    Let the assumptions be as in Theorem \ref{IntroRelativeBoundaryExtension} and let $\rho:\Gamma\to \SL(d,\Kb)$ be a $P_k$-divergent representation. Then
    \begin{itemize}
        \item[(1)] $\rho$ is extended geometrically finite relative to $\Hc$ if and only if $\rho$ is extended geometrically finite relative to $\Pc$ and $\rho|_P$ is extended geometrically finite relative to $\Hc_P$ for each $P\in \Pc$;
        \item[(2)] $\rho$ is $P_k$-Anosov relative to $\Hc$ if and only if $\rho$ is extended geometrically finite relative to $\Pc$ and $\rho|_P$ is $P_k$-Anosov relative to $\Hc_P$ for each $P\in \Pc$.
    \end{itemize}
\end{theorem}

\subsection{Further remarks and discussions}

\subsubsection{Weakly dominated representations} Since we mentioned that Anosov representations can be charaterized by the dominated property, it is natural to consider the following condition. Let $(\Gamma,\Pc)$ be a relatively hyperbolic group and let $X$ be Gromov model of $(\Gamma,\Pc)$ with a fixed base point $o\in X$.

\begin{definition}\label{DefWeaklyDominated}
    A representation $\rho:\Gamma\to \SL(d,\Kb)$ is \emph{$P_k$-weakly dominated with respect to $X$} if there exist constant $C,c > 0$, such that \[\dfrac{\sigma_{k+1}(\rho(\gamma))}{\sigma_{k}(\rho(\gamma))}\leqslant Ce^{-c\abs{\gamma}_X}~,\] for any $\gamma\in \Gamma$, where $\abs{\gamma}_X = d(o,\gamma o)$.
\end{definition}

We will show that a $P_k$-divergent extended geometrically finite representation $\rho$ is $P_k$-weakly dominated (see Corollary \ref{WeaklyDominatedCorollary}) with respect to some Gromov model of $(\Gamma,\Pc)$. However, the contrary is unclear.

\begin{question}
    If $\rho$ is $P_k$-weakly dominated with respect to $X$, is it extended geometrically finite relative to $\Pc$?
\end{question}

Recall that when $\rho$ is $P_k$-divergent, the limit set $\Lc(\rho)$ is well-defined. To answer the question, we wish to show that $\Lambda$ provides a transverse boundary extension. The naive way to define such a boundary extension is by mapping each $P_k$-limit of $(\rho(\gamma_n) )_{n\in \Nb}$ in $\Fc_{k,d-k}$, that is, $\lim\limits_{n\to +\infty} (U_k(\rho(\gamma_n)),U_{d-k}(\rho(\gamma_n)))$, to the limit of $(\gamma_n)_{n\in \Nb}$ on $\partial(\Gamma,\Pc)$, which would be well-defined and transverse when a ``uniform transversality'' condition holds. Due to the results in Section \ref{SectionResultW23}, this ``uniform transversality'' condition holds when restricted to any cocompact subflow of $\Fc(\Lambda,\zeta,X)$. However, it seems to get weaker when one enlarges the flow space.

\subsubsection{The choice of Gromov models}\label{TheChoiceofGromovModel}
Let $(\Gamma,\Pc)$ be a relatively hyperbolic group and let $\rho:\Gamma\to \SL(d,\Kb)$ be a representation. Recall that in Corollary \ref{LimitMapWaiveCor}, we state that if $\rho$ is $P_k$-Anosov relative to $\Pc$, then it is $P_k$-Anosov in restriction to $\Fc(X)$ for any Gromov model $X$ of $(\Gamma,\Pc)$. The proof due to the fact that when $\rho$ is $P_k$-Anosov relative to $\Pc$, $\rho|_P$ is weakly unipotent (see \cite[Proposition 5.1]{ZZ1}) and hence $P_k$-weakly dominated with respect to any Gromov model $X$ of $(\Gamma,\Pc)$ for each $P\in \Pc$ (see Section \ref{SectionRemarkonRAR}).

In the general case when $\rho$ is $P_k$-divergent and extended geometrically finite relative to $\Pc$, there are some (but not any) Gromov models, such that $\rho|_P$ is $P_k$-weakly dominated with respect them, for each $P\in \Pc$. Therefore, we cannot refine Theorem \ref{MainTheorem} (2) to say ``for any Gromov model''. We provide here a method of constructing a representation that is $P_k$-divergent and extended geometrically finite relative to $\Pc$, which is not weakly dominated when restricted to the parabolic subgroups in $\Pc$ with respect to some Gromov model $X$.

The work of Mitra \cite{Mitra98} shows that there is a discrete subgroup $G$ of $\PSL(2,\Cb)$ that isomorphic to the free group of two generators, whose distortion has arbitrarily fast growth, by studying a class of examples from \cite{Minsky99}. 
Roughly speaking, the distortion having arbitrarily fast growth in this case is the same as $\log\dfrac{\sigma_1(\gamma)}{\sigma_2(\gamma)}$ diverging arbitrarily slow with respect to $\abs{\gamma}$ for $\gamma\in G$, where $\abs{\cdot}$ denotes the word metric on $G$. So we may assume that there are no constants $A>0$ and $B\geqslant 0$ such that \[\log\dfrac{\sigma_1(\gamma)}{\sigma_2(\gamma)} \geqslant A \log\abs{\gamma} - B\] holds for all $\gamma\in G$, that is, the divergence is strictly slower than any power of the word length.
It is easy to see that $(G*\Zb,\{G\})$ is relatively hyperbolic by the criterion given by \cite{Bow12}, or applying a combination theorem in \cite{Dah03}.

Embed $\PSL(2,\Cb)\cong \SO^+(3,1)$ into $\SO(4,1)$ naively, and we pick a loxodromic elements $\eta$ in $\SO(4,1)$ with two fixed points away from the limit set of $G$. Then by the ping-pong lemma, $G*\Zb \cong \langle G,\eta \rangle $ up to replacing $\eta$ by its powers, and hence we view $G*\Zb$ as a subgroup of $\SO(4,1)$. $G*\Zb$ is divergent since it is discrete in $\SO(4,1)$, a semisimple Lie group of rank $1$. The embedding of $G*\Zb$ into $\SO(4,1)$ is extended geometrically finite relative to $\{G\}$ by Lemma \ref{ExtendedDivergenceLemma}, Proposition \ref{RefinedtoMinimal} and Remark \ref{OpenSettobeWholeTransverse}.

However, whether $G$ is weakly dominated depends on the choice of the Gromov model. For example, let $X_{GM}$ be the  Groves-Manning cusp space of $(G*\Zb,\{G\})$, which is a Gromov model constructed by \cite{GM08}. By \cite[Proposition 3.9]{ZZ1}, there exist constants $\alpha\geqslant 1$ and $\beta\geqslant 0$, such that \[ \alpha^{-1} \log\abs{\gamma} -\beta \leqslant \abs{\gamma}_{X_{GM}} \leqslant \alpha \log\abs{\gamma} + \beta \] for any $\id\ne \gamma \in G$. Therefore, by our assumption above, $\rho|_G$ is not weakly dominated with respect to $X_{GM}$. To make it weakly dominated, we need to pick a Gromov model $X$, such that it ``shrinks the distance in the cusp region enough that match the distortion'' (see Section \ref{SectionGMCuspSpace} and Lemma \ref{GromovModelDivtoDom}).

\subsection{Structure of the paper}
In Section \ref{SectionRHG}, we recall some preliminaries of relatively hyperbolic groups and their Gromov models. In particular, we introduce the construction of a generalized Groves--Manning cusp space.
In Section \ref{SectionFlows}, we introduce the notion of flow, especially the flow associated to relatively hyperbolic groups, then we define and discuss about some basic properties of restricted Anosov representations.
In Section \ref{SectionBoundaryExtensions}, we introduce boundary extensions of relatively hyperbolic groups and associated flow spaces. We recall the definitions and basic results related to extended geometrically finite representations, relatively Anosov representations and asymptotically embedded representations.
Then we show Theorem \ref{MainTheorem} (2)$\Rightarrow$(1) in Section \ref{Section2to1} and then deduce the second part of Theorem \ref{MainTheorem} as a corollary. We show Theorem \ref{MainTheorem} (1)$\Rightarrow$(2), Corollary \ref{LimitMapWaiveCor}, Theorem \ref{MainRefinedtoMinimal} and Theorem \ref{MainStabilityTheorem} in Section \ref{Section1to2}.
Finally, in Section \ref{SectionRHExtensions}, we show Theorem \ref{IntroRelativeBoundaryExtension} and Theorem \ref{RelativeAnosovExtension}, then deduce that the example given in \cite{TW23} is a divergent, extended geometrically finite representation, which is not relatively Anosov.

\subsection{Acknowledgements}
The author would like to thank Nicolas Tholozan and Tengren Zhang for their patient guidance and many useful advice. The author also thank Jeffrey Danciger, Fran{\c{c}}ois Gu{\'e}ritaud, Mitul Islam, Max Riestenberg, Theodore Weisman and Feng Zhu for many interesting discussions, and thank Hee Oh and Andrew Zimmer for pointing out errors in previous versions.

\section{Relatively hyperbolic groups}\label{SectionRHG}

\subsection{Definitions and notations}
If $(X,d_X)$ is a metric space and $A\subset X$ is a subset, we denote $\Nc_R(A)$ the neighborhood of $A$ of radius $R$. A geodesic metric space $(X,d_X)$ is \emph{taut} if there exists a constant $R\geqslant 0$, such that for any point $x\in X$, there exists a bi-infinite geodesic $\ell$ in $X$ with $x$ contains in $\Nc_R(\ell(\Rb))$.

Recall that $\Gamma$ is a hyperbolic group if $\Gamma$ admits a properly discontinuous and cocompact action on a Gromov hyperbolic space $X$ by isometries, in which case $X$ is called a Gromov model of $\Gamma$. The notion of relatively hyperbolic group, introduced by Gromov \cite{Gromov}, generalized the notion of hyperbolic group, by asking the $\Gamma$-action to be ``geometrically finite'' instead of cocompact.

\begin{definition}\label{RelativelyHyperbolicGroupsDefinition}
Let $\Gamma$ be a finitely generated group with $\Pc$ a finite collection of finitely generated, infinite subgroups of $\Gamma$. Let \[\Pc^\Gamma = \{\gamma P\gamma^{-1} : \gamma\in \Gamma,P\in \Pc\}\] be the collection of conjugates of the subgroups in $\Pc$. The pair $(\Gamma,\Pc)$ is said to be a \emph{relatively hyperbolic group} if there exists a taut Gromov hyperbolic space $(X,d_X)$ and a properly discontinuous $\Gamma$-action on $X$ by isometries such that
\begin{itemize}
    \item[(1)] There exists $\Bc = \{B_P : P\in \Pc^\Gamma \}$, a collection of $\Gamma$-invariant, disjoint open horoballs of $X$, with $P\in \Pc^\Gamma$ the stabilizer of $B_P$ for each $P\in \Pc^\Gamma$;
    \item[(2)] $\Gamma$ acts on $X^{th} = X \setminus \bigcup_{B\in \Bc } B $, the thick part of $X$, cocompactly.
\end{itemize}
\end{definition}

With the conditions in the definition, we say that the $\Gamma$-action on $X$ is a \emph{cusp uniform action}, and $X$ is called a \emph{Gromov model} of $(\Gamma,\Pc)$. The decomposition $X = X^{th} \cup (\bigcup_{B\in \Bc} B )$ is called a \emph{thick-thin decomposition} of $X$. The Gromov boundary of $X$ is called the \emph{Bowditch boundary} of $(\Gamma,P)$, denoted by $\partial(\Gamma,\Pc)$. The elements of $\Pc^\Gamma$ are called \emph{parabolic subgroups} of $\Gamma$. For each $P\in\Pc^\Gamma$, we say $P$ is the \emph{stabilizer} (up to finite index) of the center $p$ of the horoball $B_P$, and $p$ is the \emph{parabolic point} fixed by $P$.
\\

Let $(\Gamma,\Pc)$ now be a relatively hyperbolic group. We say that a generating set $S$ of $\Gamma$ is adapted if $S$ is finite, symmetric and for each $P\in \Pc$, $S\cap P$ generates $P$. We denote by $\abs{\cdot}_S$ the word metric on $\Gamma$ with respect to $S$.

We may always assume that $X^{th}$ is a length space with the length metric induced by the metric on $X$, up to replacing $X^{th}$ by $\Nc_R(X^{th})$ for some $R>0$. By Milnor-\v{S}varc Lemma, $(\Gamma,\abs{\cdot}_S)$ is quasi-isometric to $X^{th}$.

For a fixed base point $o\in X^{th}$, we define the norm $\abs{\cdot}_X$ on $\Gamma$ by $\abs{\gamma}_X = d_X(o,\gamma o)$, for each $\gamma\in\Gamma$. Up to additive constants, $\abs{\cdot}_X$ is independent of the choice of the base point $o$, since for a different choice of the base point $o'$, the metric defined is $(1,2d(o,o'))$-quasi-isometric to the one above.
\\

By Bowditch \cite{Bow12}, $\partial(\Gamma,\Pc)$ is defined independent of the choice of the Gromov model and $\Gamma$ acts on $\partial(\Gamma,\Pc)$ as a convergence group, moreover, $\partial(\Gamma,\Pc)$ consists of only conical limit points and bounded parabolic points, in which case, we say the $\Gamma$-action is \emph{geometrically finite}. The converse is also true by the following theorem of Yaman \cite{Yam}.

\begin{theorem}[\cite{Yam}]\label{YamanCriterion}
    Suppose a discrete group $\Gamma$ acts on a nonempty, perfect, metrizable, compact space $M$, as a geometrically finite convergence group, and there are only finitely many $\Gamma$-orbits of bounded parabolic points, with the stabilizer of each bounded parabolic point infinite, finitely generated. Let $\{p_1,p_2,...,p_n\}$ be a set of representatives of the $\Gamma$-orbits of bounded parabolic points. Let $\Pc = \{P_1,P_2,...,P_n\}$ with each $P_i$ the stabilizer of $p_i$ for $i=1,2,...,n$. Then $(\Gamma,\Pc)$ is a relatively hyperbolic group and $M$ is $\Gamma$-equivariantly homeomorphic to $\partial(\Gamma,\Pc)$.
\end{theorem}

We denote by $\partial_{con}(\Gamma,\Pc)\subset \partial(\Gamma,\Pc)$ the collection of conical limit points and $\partial_{par}(\Gamma,\Pc)\subset \partial(\Gamma,\Pc)$ the collection of parabolic points.

\subsection{Constructing Gromov models}\label{SectionGMCuspSpace}
For a relatively hyperbolic group $(\Gamma,\Pc)$, Groves--Manning \cite{GM08} (see also \cite[Section 3.4]{ZZ1}) defined a Gromov model by gluing combinatorial horoballs to the Cayley graph of $\Gamma$.

Let $f:\Rb_{\geqslant 0} \to \Rb_{\geqslant 0}$ be an increasing function such that $f(0) = 1$ and $f(t)\to +\infty$ as $t\to +\infty$.

\begin{definition}
    Let $(T,d_T)$ be a metric graph, that is, a graph $T=(T^{(0)},T^{(1)})$ endowed with a metric $d_T$ with each edge of length $1$. The \emph{$f$-combinatorial horoball} based on $T$ is a graph $\Hc_f(T)$ with the set of vertices $\Hc_f(T)^{(0)} = T \times \Nb$. The set of edges $\Hc_f(T)^{(1)}$ consists of the following types of edges
    \begin{itemize}
        \item for any $v,w\in T^{(0)}$ and $k\in \Nb$ such that $0<d_T(v,w)\leqslant f(k)$, there is an edge between $(v,k)$ and $(w,k)$;
        \item for each $v\in T^{(0)}$ and $k\in \Nb$, there is an edge between $(v,k)$ and $(v,k+1)$.
    \end{itemize}
    The full subgraph of $\Hc_f(T)$ with vertices $T\times\{0\}$ is a copy of $T$, we call which the \emph{$0$-level} of $\Hc_f(T)$.
\end{definition}

Let $S$ be an adapted generating set of $(\Gamma,\Pc)$. We define \[X_f(\Gamma,\Pc,S) = \Cay(\Gamma,S) \cup\big(\bigcup_{\gamma\in\Gamma,P\in\Pc} \Hc_f(\gamma\Cay(P,S\cap P))\big)~,\] where we identify the subgraph $\gamma\Cay(P,S\cap P)$ in $\Cay(\Gamma,S)$ and the $0$-level of $\Hc_f(\gamma\Cay(P,S\cap P))$, for each $\gamma\in\Gamma$ and $P\in\Pc$.

\begin{theorem}[\cite{GM08}]
    If $f(s+t)\geqslant 2^t f(s)$ for any $s,t\in \Rb_{\geqslant 0}$, then $X_f(\Gamma,\Pc,S)$ is a Gromov model of $(\Gamma,\Pc)$.
\end{theorem}
\begin{remark}
    When $f(t) = 2^t$, the result is given by \cite[Theorem 3.25]{GM08}, by showing a linear homological isoperimetric inequality for $X(\Gamma,\Pc,f)$. The argument actually still holds as long as the $f$-combinatorial horoball satisfies that if there is a edge between $(v_1,k)$ and $(v_2,k)$, and a edge between $(v_2,k)$ and $(v_3,k)$, then there is a edge between $(v_1,k+1)$ and $(v_3,k+1)$. 
\end{remark}

We call $X_f = X_f(\Gamma,\Pc,S)$ the \emph{generalized Groves--Manning cusp space associated to $f$}.

Let $\Xc_f = \Xc_f(\Gamma,\Pc,S)$ be the collection of Gromov models $X$ that admit constants $\lambda\geqslant 1$ and $\epsilon\geqslant 0$, such that $f(\lambda^{-1}\abs{\gamma}_X-\epsilon)\leqslant \abs{\gamma}_S$ for any $P\in \Pc$ and $\gamma\in P$. Then the generalized Groves-Manning cusp space $X_f(\Gamma,\Pc,S)\in \Xc_f(\Gamma,\Pc,S)$.

We show that the function in $2^t$ in the theorem is optimal up to quasi-isometries. Let $\log^+_2(\cdot) = \max(\log_2(\cdot), 0)$.

\begin{proposition}\label{RatioGrowthLinearLemma}
    Let $X$ be a Gromov model of $(\Gamma,\Pc)$, then there exist constants $\lambda\geqslant 1$ and $\epsilon\geqslant 0$, such that for any $P\in \Pc$ and $\gamma\in P$, $\abs{\gamma}_X\leqslant \lambda \log^+_2 \abs{\gamma}_S + \epsilon$.
\end{proposition}
\begin{proof}
    We fix a thick-thin decomposition of $X$ and let $X^{th}$ be the thick part. Since $\Gamma$ acts on $X^{th}$ cocompactly, there exists a compact set $K$ that contains the base point $o$, such that $\Gamma\cdot K = X^{th}$ and $P\cdot(K\cap \overline{B_P}) = \overline{B_P}\setminus B_P$ for each $P\in \Pc$. Let $D$ be the diameter of $K$. For $\gamma\in P$, let $x_0 \in K \cap \overline{B_P}$ be a point on $\overline{B_P}\setminus B_P$, the horosphere of $B_P$. Then $\gamma x_0$ is also on the horosphere. Let $p$ be the parabolic point fixed by $P$ and let $\ell$ and $\ell'$ be geodesic rays from $x_0$ and $\gamma x_0$ to $p$ respectively. Let $[x_0,\gamma x_0]$ and $[x_0,\gamma x_0]^{th}$ be geodesic segments between $x_0$ and $\gamma x_0$ in $X$ and in $X^{th}$ (with the induced length metric) respectively. Let $R$ be the length of $[x_0,\gamma x_0]$ and let $R_{th}$ be the length of $[x_0,\gamma x_0]^{th}$. Let $\delta \geqslant 0$ be a constant such that $X$ is $\delta$-hyperbolic, then any geodesic triangle (might with vertices on the boundary) is $2\delta$-thin (see, for example, \cite[Lemma 3.8]{Can}). Let $m \in [x_0,\gamma x_0]$ be a point such that there exists $x$ on $l$ and $x'$ on $l'$, with $d(m,x)\leqslant 2\delta$ and $d(m,x')\leqslant 2\delta$. Then by \cite[Part III, Proposition 1.6]{BH}, we have that \[d(m,[x_0,\gamma x_0]^{th})\leqslant \delta \log^+_2 R_{th} +1~.\] One of $d(x_0,m)$ and $d(\gamma x_0, m)$ is at least $R/2$. Without loss of generality, we assume $d(x_0,m)\geqslant R/2$, then $d(x, x_0)\geqslant d(x_0,m) - d(m,x) \geqslant R/2 - 2\delta$. Then we have \[d(m, [x_0,\gamma x_0]^{th})+ 2\delta \geqslant d(x, [x_0,\gamma x_0]^{th}) \geqslant d(x, X^{th}) = d(x, x_0) \geqslant R/2 - 2\delta~,\] where $d(x, X^{th}) = d(x, x_0)$ follows from the properties of horofunctions and horoballs. Therefore, \[R/2 - 4\delta \leqslant d(m,[x_0,\gamma x_0]^{th})\leqslant \delta \log^+_2 R_{th} +1~,\] and then \[ R \leqslant 2\delta \log^+_2 R_{th} + 2+ 8\delta ~.\]
    Since $(\Gamma,\abs{\cdot}_S)$ is quasi-isometric to $X^{th}$ by an orbit map that maps $\id$ to $o$, there exists $A\geqslant 1$ and $B\geqslant 0$, such that $R_{th} \leqslant A \abs{\gamma}_S +B + 2D$. On the other hand, $R \geqslant \abs{\gamma}_X -2D$. Then we deduce that there exists constants $\lambda\geqslant 1$ and $\epsilon\geqslant 0$, such that $\abs{\gamma}_X\leqslant \lambda \log^+_2 \abs{\gamma}_S + \epsilon$ for any $\gamma\in P$. Finally, since $\Pc$ is finite, the constants $\lambda,\epsilon$ can be picked uniform.
\end{proof}

\section{Flows and restricted Anosov representations}\label{SectionFlows}

\subsection{Restricted Anosov representations}
The notion of Anosov representation in restriction to closed subflows of the geodesic flow of a hyperbolic group was introduced and studied in \cite{W23}. Then \cite{TW23} came up with a more general notion, the restricted Anosov representations. We recall some basic definitions here. We always assume that $\Kb$ is the field $\Rb$ or $\Cb$, and $k,d\in \Nb$ such that $1\leqslant k \leqslant d/2$. Let $\Gamma$ be a discrete group.

\begin{definition}\cite[Definition 2.6]{TW23}
     A flow space $(Y,\phi)$ is called a \emph{$\Gamma$-flow} if there exists a properly discontinuous $\Gamma$-action on $Y$ that commutes with $\phi$. We say a $\Gamma$-flow $(Y,\phi)$ is cocompact if the $\Gamma$-action on $Y$ is cocompact. 
\end{definition}

Let $(Y,\phi)$ be a $\Gamma$-flow. For a representation $\rho:\Gamma\to\SL(d,\Kb)$, we consider the flat bundle $E_\rho(Y) = \Gamma\backslash (Y\times\Kb^d)$, where the $\Gamma$-action on the trivial bundle $Y\times\Kb^d$ is defined by \[\gamma(y,v) = (\gamma y,\rho(\gamma)v)\] for any $\gamma\in\Gamma,y\in Y$ and $v\in\Kb^d$. The flow $\phi$ lifts to $Y\times\Kb^d$ naturally by parallel transformation, that is, \[\phi(y,v) = (\phi(y),v)~,\] for any $y\in Y$ and $v\in\Kb^d$. Since this flow commutes with the $\Gamma$-action on $Y\times\Kb^d$, it projects to a well-defined flow on $E_\rho(Y)$, still denoted by $\phi$.

We say that $E_\rho(Y) = E^s_\rho(Y) \oplus E^u_\rho(Y)$ is a \emph{dominated splitting of rank $k$} of $E_\rho(Y)$, with respect to a metric $\norm{\cdot}$ on $E_\rho(Y)$ (a Riemann metric when $\Kb=\Rb$, or a Hermitian metric when $\Kb=\Cb$), if $E^s_\rho(Y)$ and $E^u_\rho(Y)$ are $\phi$-invariant subbundles of $E_\rho(Y)$ with $\rank(E^s_\rho(Y))=k$, and there exist constants $C,c > 0$, such that \[ \dfrac{\Vert \phi^t_y(v)\Vert}{\Vert \phi^t_y(w)\Vert} \leqslant C e^{-c t}\dfrac{\Vert v\Vert}{\Vert w\Vert}~,\] for any $y\in \Gamma\backslash Y$, $v\in (E^s_\rho)_y$, $w\in (E^u_\rho)_y$ nonzero vectors and $t\in\Rb_{\geqslant 0}$. $E^s_\rho(Y)$ (respectively, $E^u_\rho(Y)$) is called the stable (respectively, unstable) direction.

\begin{definition}\cite[Definition 3.1]{TW23}\label{DefRestricetedAnosov}
    A representation $\rho:\Gamma\to\SL(d,\Kb)$ is \emph{$P_k$-Anosov in restriction to $(Y,\phi)$} if there is a metric on $E_\rho(Y)$, such that $E_\rho(Y)$ admits a dominated splitting of rank $k$ with respect to the metric.
\end{definition}

\begin{remark}\label{LiftAbuseNotations}
    We also abuse the to say that $Y\times \Kb^d = E^s_\rho\oplus E^u_\rho$ is a dominated splitting of $Y\times \Kb^d$ of rank $k$ with respect to a $\rho(\Gamma)$-invariant metric if they are the lifts from a such dominated splitting of $E_\rho(Y)$ through the quotient.
\end{remark}

\begin{remark}\cite[Remark 3.4]{TW23}\label{UnitVolumeMetric}
    We denote the standard metric on $\Kb^d$ by $\norm{\cdot}_0$. In the context of Definition \ref{DefRestricetedAnosov} and Remark \ref{LiftAbuseNotations}, We say that a metric $\norm{\cdot}$ on $Y\times \Kb^d$ is \emph{of unit volume} if there exists a map $A: Y\to \SL(d,\Kb)$, such that $\norm{\cdot} = \norm{A(y)\cdot}_0$. Since the ratio of the norms of two vectors is preserved by rescaling, we may always assume that the metric we pick is of unit volume. When $d=2$ and $k=1$, if $Y\times \Kb^d = E^s_\rho\oplus E^u_\rho$ is a dominated splitting of rank $1$ with respect to a metric $\norm{\cdot}$ of unit volume, then there exist constants $C,c > 0$ such that
    \[\norm{\phi^t(v)} \leqslant C e^{-c t}\norm{v}\text{ and } \norm{\phi^t(w)} \geqslant C e^{c t}\norm{ w}~,\] for any $y\in Y$, $v\in (E^s_\rho)_y$, $w\in (E^u_\rho)_y$ and $t\in\Rb_{\geqslant 0}$. This implies that the direct sum of several such dominated splittings is again a dominated splitting.
\end{remark}

\subsection{The flow spaces}

Suppose $X$ is a taut Gromov hyperbolic space and a group $\Gamma$ acts on $X$ properly discontinuous by isometries. In particular, we are interested in the case when $\Gamma$ is a hyperbolic group, or $(\Gamma,\Pc)$ is a relatively hyperbolic group, and $X$ is a Gromov model of $\Gamma$ or $(\Gamma,\Pc)$, respectively.

A direct way to construct a $\Gamma$-flow is to consider the collection of parameterized geodesics of $X$. Define \[\mathcal{G}(X)=\{\ell: \Rb \to X\ |\ \ell \text{ is a geodesic parametrized by length }\}~,\] equipped with the flow $\psi$ defined by \[\phi^t(\ell): s \mapsto \ell(s + t)~.\] Then $(\mathcal G(X), \phi)$ is a $\Gamma$-flow. One can define the metric $d$ on $\mathcal{G}(X)$ by \[d(\ell,\ell') = \int_{\Rb} \frac{e^{-\abs{t}}}{2}d_X\big(\ell(s),\ell'(s)\big)ds~,\] for any $\ell,\ell'\in \mathcal{G}(X)$ and verify that the projection $\pi_M':\mathcal{G}(X)\to X$ defined by $\pi_M'(\ell)=\ell(0)$ is a $\Isom(X)$-equivariant quasi-isometry.\\

\begin{definition}
    A conical limit point $x\in\partial X =\partial(\Gamma,\Pc)$ is \emph{compactly attained} if there is a geodesic ray $\ell $ in $X$, such that $\ell (+\infty) = x$ and the projection of $\ell $ into $\Gamma\backslash X$ is contained in a compact set. We denote the set of compactly attained conical limit point by $\partial_{cc}X$.
\end{definition}

\begin{notation}\label{notation: compact conical limit flow}
    We set \[\Fc(X) = \overline{\{\ell\in \Gc(X): (\ell(+\infty),\ell(-\infty)) \in \partial_{cc}^2X \}}\subset \Gc(X)~.\].
\end{notation}

Now suppose $(\Gamma,\Pc)$ is relatively hyperbolic group. Since the action of $\Gamma$ on $\partial(\Gamma,\Pc)$ is minimal, the subset $\partial_{cc}X$ is dense in $\partial X$. Consequently, for every pair of distinct points in $\partial^{(2)}(\Gamma,\Pc)$, there exists at least one flow line in $\Fc(X)$ having this pair as its endpoints. In particular, if $X$ is uniquely geodesic, in the sense that any two bi-infinite geodesics with the same endpoints coincide, then $\Fc(X)=\Gc(X)$.

The reason for considering the flow space $\Fc(X)$ rather than the full geodesic space $\Gc(X)$ is that $\Fc(X)$ is designed to capture the part of the dynamics arising from the cocompact region. This is the appropriate setting for studying representations on compact flow spaces and their limits, while excluding isolated flow lines that do not occur as limits of geodesics coming from the cocompact part.

We summaries the properties of $\Fc(X)$ here.

\begin{proposition}\label{proposition: compact conical limit flow}
    \begin{itemize}
        \item[(1)] The $\Gamma$-equivariant projection $p_M:(\Fc(X),d)\to (X,d_X)$ is a quasi-isometry.
        \item[(2)] Any flow line in $\Fc(X)$ is a geodesic in $\Fc(X)$.
        \item[(3)] The flow $\phi$ on $\Fc(X)$ commutes with the $\Gamma$-action.
        \item[(4)] The reverse on $\Fc(X)$ defined by $(\ell:t\mapsto \ell(t)) \mapsto (\hat{\ell}: t\mapsto \ell(-t))$, commutes with the $\Gamma$-action on $\Fc(X)$.
\end{itemize}
\end{proposition}


When $\Gamma$ is a hyperbolic group and $X$ is a Gromov model of $\Gamma$ (e.g. its Cayley graph with respect to a finite generating set), we define the \emph{geodesic flow of $\Gamma$} to be the space $\mathcal F(\Gamma) = \mathcal F(X)$ equipped with the flow $\phi$ and the action of $\Gamma$. This $\Gamma$-flow is well-defined up to quasi-isometric isomorphisms.

\subsection{Anosov in restriction to cocompact subflows}\label{SectionResultW23}
Let $(\Gamma,\Pc)$ be a relatively hyperbolic group and let $(X,d_X)$ be a Gromov model of $(\Gamma,\Pc)$. We fix a thick part $X^{th}$ of $X$. For any $R\geqslant 0$, let $\Fc_R(X)$ be the closed subflow of $\Fc(X)$ that consisting of flow lines corresponding to the geodesics of $X$ contained in $\overline{\Nc_R(X^{th})}$, that is, $(x,y,\Rb)\subset \Fc_R(X)$ if and only if there exists a geodesic $\ell $ of $X$, such that $\ell (+\infty)=x$, $\ell (-\infty)=y$ and $\ell (\Rb)\subset \overline{\Nc_R(X^{th})}$.

\begin{definition}
    A conical limit point $x\in\partial X =\partial(\Gamma,\Pc)$ is \emph{compactly attained} if there is a geodesic ray $\ell $ in $X$, such that $\ell (+\infty) = x$ and the projection of $\ell $ into $\Gamma\backslash X$ is contained in a compact set. We denote the set of compactly attained conical limit point by $\partial_{cc}X$.
\end{definition}

A cocompact subflow of $\Fc(X)$ can be expressed as $F\times\Rb$, where $F$ is a closed, $\Gamma$-invariant subset of $\partial^{(2)} X$ contained in $\partial_{cc}^{(2)}X = \{(x,y)\in \partial_{cc}^{2}X: x\ne y\})$. It is not hard to see that $\Fc_R(X)$ is a cocompact subflow of $\Fc(X)$ and any cocompact subflow of $\Fc(X)$ is contained in $\Fc_R(X)$ for some $R\geqslant 0$ large enough. 

Notice also that by Proposition \ref{proposition: compact conical limit flow}, there is a $\Gamma$-equivariant quasi-isometry $p_M: \Fc(X) \to X$.

Now we fix a cocompact subflow $F\times\Rb \subset \Fc_R(X)$. Let $K$ be a compact subset of $\Fc(X)$, such that $\Gamma \cdot K \supset \Fc_R(X)$ and $\Gamma \cdot p_M(K) \supset \Nc_R(X^{th})$. Let $z_0\in K$ be a fixed base point. Recall that $\abs{\gamma}_X = d_X(p_M(z_0), \gamma p_M(z_0))$ defines a metric on $\Gamma$.

We consider the subset of $\Gamma$ consisting of elements that tracking the geodesics of $F\times\Rb$, that is, \[\Gamma^+_{F,K} = \{\gamma\in \Gamma : \exists z\in K\cap (F\times\Rb), t\in \Rb_{\geqslant 0},\text{ such that }\gamma^{-1}\phi^t(z)\in K\}~.\]

We say that a representation $\rho: \Gamma\to \mathrm{SL}(d,\mathbb{K})$ is \emph{$P_k$-dominated on $\Gamma^+_{F,K}$} if there are constants $C,c > 0$, such that for all $\gamma \in \Gamma^+_{F,K}$, \[\dfrac{\sigma_{k+1}(\rho(\gamma))}{\sigma_k(\rho(\gamma))}\leqslant Ce^{-c\abs{\gamma}_S}~,\] where $\sigma_k(\cdot)$ denotes the $k$-th singular value.

\begin{notation}
    For $A,B\in\Rb$, $\lambda\geqslant 1$ and $\epsilon\geqslant 0$, we write $A\sim_{(\lambda,\epsilon)}B$ if $\lambda^{-1}A - \epsilon \leqslant B \leqslant \lambda A + \epsilon$.
\end{notation}

Similar to \cite[Lemma 3.8]{W23}, we have the following estimation.

\begin{lemma}\label{QuasiIsometricEstimation}
    Let $\gamma\in\Gamma^+_{F,K}$ with $z,\gamma^{-1}\phi^t(z)\in K$, then there exist constants $\lambda\geqslant 1$ and $\epsilon\geqslant 0$, such that, \[ \abs{\gamma}_X \sim_{(\lambda,\epsilon)} \abs{t} \sim_{(\lambda,\epsilon)} \abs{\gamma}_S\]
\end{lemma}
\begin{proof}
    By Milnor-\v{S}varc lemma, $(\Gamma,\abs{\cdot}_S)$ is quasi-isometric to $\Nc_R(X^{th})$ for any fixed $R\geqslant 0$, then $\abs{t} \sim_{(\lambda_1,\epsilon_1)} \abs{\gamma}_S$ for some  constants $\lambda_1\geqslant 1$ and $\epsilon_2\geqslant 0$. There are also constants $\lambda_2\geqslant 1$ and $\epsilon_2\geqslant 0$, such that
    $\abs{\gamma}_X \sim_{(\lambda_2,\epsilon_2)} \abs{t}$ following from Proposition \ref{proposition: compact conical limit flow}, which states that $\{\phi^t(z):z\in\Rb\}$ is a geodesic. Finally, we take $\lambda = \max(\lambda_1, \lambda_2)$ and $\epsilon = \max(\epsilon_1, \epsilon_2)$.
\end{proof}

Let \[\Pi_i(F)=\{x_i\in\partial(\Gamma,\Pc):\text{there is some element }(x_1,x_2)\in F \}\] for $i=1,2$. For a pair of limit maps \[\xi^k: \Pi_1(F)\to \Gr_k(\mathbb{K}^d)\text{ and }\xi^{d-k}:  \Pi_2(F)\to \Gr_{d-k}(\mathbb{K}^d)~,\] we say $(\xi^k,\xi^{d-k})$ is \emph{transverse on $F$} if for any $(x,y)\in F$, $\xi^k(x)\oplus \xi^{d-k}(y) = \Kb^d$, and we say it is \emph{strongly dynamics preserving on $F$} if for any sequence $(\gamma_n)_{n\in \Nb}\subset \Gamma^+_{F,K}$, with $\gamma_n\to x\in \Pi_1(F)$ and $\gamma_n^{-1}\to y\in \Pi_2(F)$ as $n\to +\infty$, we have $\gamma_n V \to \xi^k(x)$ as $n\to +\infty$, uniformly on any compact subset of $\Gr_k(\Kb^d)$ transverse to $\xi^{d-k}(y)$.

Recall that for a matrix $g\in\SL(d,\Kb)$, we denote the eigenspace of $gg^t$ for the largest $k$ eigenvalues  by $U_k(g)$. Then $U_{d-k}(g^{-1}) = g^{-1}U_k(g)^{\perp}$.

Using Lemma \ref{QuasiIsometricEstimation}, it is direct to apply the proof of \cite[Theorem 3.7, Proposition 5.6, Theorem 5.8, Proposition 5.9, Proposition 5.12 and Theorem 5.15]{W23} to deduce the following theorem.

\begin{theorem}\label{TheoremCompactSubflow}
    Let $\rho$ be a representation of $\Gamma$ into $\SL(d,\Kb)$, then the following are equivalent.
    \begin{itemize}
        \item[(1)] $\rho$ is $P_k$-Anosov in restriction to $F\times \Rb$;
        \item[(2)] $\rho$ is $P_k$-dominated on $\Gamma^+_{F,K}$;
        \item[(3)] There exists a unique pair of limit maps $\xi^k: \Pi_1(F)\to \Gr_k(\mathbb{K}^d)$ and
        $\xi^{d-k}:  \Pi_2(F)\to \Gr_{d-k}(\Kb^d)$, with
        $(\xi^k,\xi^{d-k}): F\to \Gr_k(\Kb^d)\times \Gr_{d-k}(\Kb^d)$ continuous, transverse and strongly dynamics preserving on $F$.
    \end{itemize}
    Moreover, if the conditions above hold, then \[\xi^k(x) = \lim\limits_{n\to +\infty} U_k(\rho(\gamma_n))\quad \text{and} \quad \xi^{d-k}(y) = \lim\limits_{n\to +\infty} U_{d-k}(\rho(\eta_n))\] for any sequence $(\gamma_n)_{n\in \Nb},\{\eta_n^{-1}\}\subset \Gamma^+_{F,K}$ with $\gamma_n\to x\in \Pi_1(F)$ and $\eta_n\to y\in \Pi_2(F)$ as $n\to +\infty$. The dominated splitting at $z=(x,y,t)\in F\times \Rb$ is given by $\xi^k(x)\oplus\xi^{d-k}(y)$.
\end{theorem}

\section{Boundary extensions}\label{SectionBoundaryExtensions}
In this section, we introduce boundary extensions of relatively hyperbolic groups and their associated flows, then we recall the definition of extended geometrically finite representations. We always assume that $(\Gamma,\Pc)$ is a relatively hyperbolic group.

\subsection{Boundary extensions and associated flows}\label{BoundaryExtensionSection}

\begin{definition}
    We say that $\zeta: \Lambda \to \partial (\Gamma, \Pc)$ is a \emph{boundary extension} of $(\Gamma,\Pc)$ if $\Lambda$ is a compact metrizable Hausdorff space with a $\Gamma$-action on it, and $\zeta:\Lambda\to\partial(\Gamma,\Pc)$ is a continuous $\Gamma$-equivariant, surjective map.
    We say that a boundary extension $\zeta: \Lambda \to \partial (\Gamma, \Pc)$ of $(\Gamma,\Pc)$ is \emph{minimal} if
    \begin{itemize}
        \item[(1)] for any $p\in \partial_{con}(\Gamma,\Pc)$, $\zeta^{-1}(p)$ is a singleton;
        \item[(2)] for any $p\in \partial_{par}(\Gamma,\Pc)$ with stabilizer $P\in\Pc^\Gamma$ and any $y\in \zeta^{-1}(p)$, there exists a sequence $(\gamma_n)_{n\in \Nb}\subset P$ and $x\in \zeta^{-1}(\partial_{con}(\Gamma,\Pc))$, such that $\gamma_n x\to y$ as $n\to +\infty$.
    \end{itemize}
\end{definition}

\begin{example}
    If $\Gamma$ is a hyperbolic group and $\Pc$ is a collection of subgroups of $\Gamma$ such that $(\Gamma,\Pc)$ is relatively hyperbolic, then there is a continuous, $\Gamma$-equivariant quotient $\tau:\partial\Gamma \to \partial(\Gamma,\Pc)$ by identifying the limit set of $P$ in $\partial\Gamma$ to the parabolic point fixed by $P$ in $\partial(\Gamma,\Pc)$ for each $P\in\Pc^\Gamma$. In this case, $\tau$ is a minimal boundary extension. See Theorem \ref{ClassicalExtensionReltoHyp} for details and see Theorem \ref{RelativeBoundaryExtension} for a generalization to the relatively hyperbolic case.
\end{example}

Let $\zeta: \Lambda \to \partial (\Gamma, \Pc)$ be a boundary extension of $(\Gamma,\Pc)$ and let $X$ be a Gromov model of $(\Gamma,\Pc)$. We denote $(\Lambda,\zeta)^{(2)} = \Lambda^2\setminus \{(x,y)\in\Lambda^2: \zeta(x)=\zeta(y) \}$. Now we define a $\Gamma$-flow associated to these data. Consider the flow space
\[\Fc(\zeta,\Lambda,X) = \{ (x,y,\ell)\in (\Lambda,\zeta)^{(2)} \times \Fc(X) : (\ell(+\infty),\ell(-\infty) = (\zeta(x),\zeta(y))\}\]
with the flow
\begin{eqnarray*}
    \phi^t: \Fc(\Lambda,\zeta,X) & \to & \Fc(\Lambda,\zeta,X) \\
    (x,y,\ell) & \mapsto & (x,y,\phi^t(\ell))~.
\end{eqnarray*}
Notice that there is a natural surjection 
\begin{eqnarray*}
    \pi:\Fc(\Lambda,\zeta,X) & \to & \Fc(X) \\
    (x,y,\ell) & \mapsto & \ell~.
\end{eqnarray*}

Then the $\Gamma$-action on $\Fc(X)$ induces a $\Gamma$-action on $\Fc(\Lambda,\zeta,X)$ by coordinates. Recall that we defined a reverse on $\Fc(X)$ in Proposition \ref{proposition: compact conical limit flow}, which also extends to $\Fc(\Lambda,\zeta,X)$ by $(x,y,\ell) \to (y,x,\hat{\ell})$ for any $(x,y,\ell)\in \Fc(\Lambda,\zeta,X)$.\\

\begin{remark}\label{asymptoticrigid}
    When talk about a representation $\rho:\Gamma\to \SL(d,\Kb)$, we further assume that the dominated splitting along two asymptotic flow lines are the same, that is, if $E_\rho = E^s \oplus E^u$ is a dominated splitting of $E_\rho = \Fc(\Lambda,\zeta,X) \times \Kb^d$, then $E^s_z = E^s_{z'}$ and $E^u_z = E^u_{z'}$ if $z = (x,y,\ell)$ and $z'= (x,y,\ell')$ share the same ``endpoints '' $x,y\in \Lambda$.
\end{remark}

Since $\Lambda$ is compact, the following lemma is direct.

\begin{lemma}\label{CompactPreimage}
    $\pi$ is proper, that is any compact subset of $\Fc(X)$ has compact preimage under $\pi$.
\end{lemma}

The following lemmas allow us to do estimation between the length $\abs{\cdot}_X$ and the flow time $t$.

\begin{lemma}\cite[Lemma A.3]{ZZ1}\label{TautTwoPointsEstimation}
    If $Y$ is a taut Gromov hyperbolic space, then there exists a constant $R\geqslant 0$, such that for any $y_1,y_2\in Y$, there is a bi-infinite geodesic $\ell$ in $Y$, such that $y_1,y_2\in \Nc_R(\ell(\Rb))$.
\end{lemma}

\begin{lemma}\label{TrackingLemma}
    For any fixed base point $z_0\in\Fc(\Lambda,\zeta,X)$, there exists a compact set $K\subset \Fc(\Lambda,\zeta,X)$ with the following properties. 
    \begin{itemize}
        \item[(1)] $z_0\in K$;
        \item[(2)] For any $\gamma\in\Gamma$, there exists $z\in K$ and $t\in \Rb$, such that $\gamma^{-1}\phi^{t}(z)\in K$.
        \item[(3)] There exist constants $\lambda\geqslant 1$ and $\epsilon\geqslant 0$, such that for any $z\in K$, $\gamma\in\Gamma$ and $t\in \Rb$ such that $\gamma^{-1}\phi^{t}(z)\in K$, it holds $\abs{\gamma}_X\sim_{(\lambda,\epsilon)} \abs{t}$.
    \end{itemize}
\end{lemma}
\begin{proof}
    Recall that $p_M:\Fc(X)\to X$ denotes the quasi-isometry provided by Proposition \ref{proposition: compact conical limit flow}. Since $X$ is taut, by Lemma \ref{TautTwoPointsEstimation}, there exists a constant $R\geqslant 0$, such that for any $\gamma\in\Gamma$, there exists a geodesic $\ell $ in $X$ with $p_M(\pi(z_0))$ and $p_M(\pi(\gamma z_0))$ both belong to the $R$-neighborhood of $\ell(\Rb)$. Then $\pi(z_0)$ and $\pi(\gamma z_0)$ are contained in the $R'$-neighborhood of $(\ell(+\infty),\ell(-\infty),\Rb)$, where $R'$ is a constant that only depends on $R$, the quasi-isometric constants of $p_M$ and the Morse lemma. We consider the set \[K = \{x\in \Fc(\Lambda,\zeta,X) : d(\pi(x),\pi(z_0))\leqslant R'\}~.\] By Lemma \ref{CompactPreimage}, $K$ is compact since it is the preimage under $\pi$ of a compact set in $\Fc(X)$. It is easy to see that $K$ satisfies (2). Finally, (3) follows from the fact that $p_M$ is a quasi-isometry.
\end{proof}

\subsection{Extended geometrically finite representations}
One of the alternative ways to describe Anosov representation is by the ``dynamics preserving'' property of the limit maps. The notion of extended geometrically finite representation of a relatively hyperbolic group into a semisimple Lie group $G$ was introduced by Weisman \cite{Weisman}, which generalized the notion of Anosov representation to relatively hyperbolic groups by considering the ``extended dynamics preserving'' property of the boundary extension. In the context of geometrically finite representations, the boundary extension is taken as a map from a subset $\Lambda$ of the flag manifold $G/P$ for some parabolic subgroup $P$ of $G$, since we also wish to look at the $\Gamma$-action not only on $\Lambda$, but also on the whole flag manifold through a representation.

For convenience, we still consider the case when $G = \SL(d,\Kb)$ with the flag manifold \[\Fc_{k,d-k} = \{(V,W)\in\Gr_k(\Kb^d)\times\Gr_{d-k}(\Kb^d): V\subset W\}~.\] We say that two element $(V_1, W_1), (V_2,W_2)\in \Fc_{k,d-k}$ are transverse if $V_1\oplus W_2 = V_2 \oplus W_1 = \Kb^d$. We say that two subsets $A,B\subset \Fc_{k,d-k}$ are transverse if any element of $A$ is transverse to any element of $B$.

Given a representation $\rho:\Gamma\to \SL(d,\Kb)$, $\rho(\Gamma)$ naturally acts on $\Fc_{k,d-k}$. 

\begin{definition}
    We say that a boundary extension $\zeta: \Lambda \to \partial (\Gamma, \Pc)$ is \emph{in $\Fc_{k,d-k}$}, if $\Lambda$ is a closed, $\rho(\Gamma)$-invariant subset of $\Fc_{k,d-k}$ and $\Gamma$ acts on $\Lambda$ through $\rho$. We say that $\zeta: \Lambda \to \partial (\Gamma, \Pc)$, a boundary extension in $\Fc_{k,d-k}$, is \emph{transverse} if for any $p\ne q \in\partial (\Gamma, \Pc)$, $\zeta^{-1}(p)$ and $\zeta^{-1}(q)$ are transverse.
\end{definition}

\begin{definition}\label{DefExtendedGeometricallyFinite}
    A representation $\rho:\Gamma\to \SL(d,\Kb)$ is \emph{extended geometrically finite} with boundary extension $\zeta: \Lambda \to \partial (\Gamma, \Pc)$ if
    \begin{itemize}
        \item[(1)] $\zeta: \Lambda \to \partial (\Gamma, \Pc)$ is a transverse boundary extension in $\Fc_{k,d-k}$.
        \item[(2)] $\zeta: \Lambda \to \partial (\Gamma, \Pc)$ \emph{extends the convergence dynamics}, that is, there exists a family $(C_p)_{p\in \partial (\Gamma, \Pc)}$ of open subsets of $\Fc_{k,d-k}$, such that $\Lambda\setminus \zeta^{-1}(p)\subset C_p$ for each $p\in \partial (\Gamma, \Pc)$, and if $(\gamma_n)_{n\in \Nb}$ is a sequence in $\Gamma$ with $\gamma_n\to p\in\partial(\Gamma,\Pc)$ and $\gamma_n^{-1}\to q\in \partial(\Gamma,\Pc)$ as $n\to +\infty$, then for any compact set $L\subset C_q$ and any open set $U\supset \zeta^{-1}(p)$, $\rho(\gamma_n)L\subset U$ when $n$ is large enough.
    \end{itemize}
\end{definition}

We simply say that $\rho:\Gamma\to \SL(d,\Kb)$ is extended geometrically finite relative to $\Pc$ if it is extended geometrically finite with some boundary extension of $(\Gamma,\Pc)$ in $\Fc_{k,d-k}$.

\begin{remark}
    Following \cite{Weisman} Proposition 4.5, we can always make the choice of the open set $C_p$ such that it is transverse to $\zeta^{-1}(p)$ for each $p\in \partial(\Gamma,\Pc)$.
\end{remark}

\begin{proposition}\cite[Proposition 4.8]{Weisman}\label{BoundaryExtensionRefinement}
    If a representation $\rho:\Gamma\to \SL(d,\Kb)$ is extended geometrically finite with boundary extension $\zeta: \Lambda \to \partial (\Gamma, \Pc)$ in $\Fc_{k,d-k}$, then there exists another boundary extension $\zeta': \Lambda' \to \partial (\Gamma, \Pc)$ in $\Fc_{k,d-k}$, with the following properties.
    \begin{itemize}
        \item[(1)] For any $p\in\partial_{con}(\Gamma,\Pc)$, $\zeta^{-1}(p)$ is a singleton;
        \item[(2)] For any $p\in\partial_{par}(\Gamma,\Pc)$, $\zeta^{-1}(p)$ is the closure of all accumulation points of $\gamma_nx$ where $(\gamma_n)_{n\in \Nb}$ is a sequence of pairwise distinct elements in $P$, the stabilizer of $p$, and $x\in C_p$ .
    \end{itemize}
    We say that a boundary extension $\zeta': \Lambda' \to \partial (\Gamma, \Pc)$ with the above properties is refined.
\end{proposition}

\begin{proposition}\cite[Proposition 4.6]{Weisman}\label{CriterionEGF}
    Let $\rho:\Gamma\to \SL(d,\Kb)$ be a representation and let $\zeta:\Lambda\to \partial(\Gamma,\Pc)$ be a transverse boundary extension in $\Fc_{k,d-k}$. Then $\rho$ is extended geometrically finite with boundary extension $\zeta:\Lambda\to \partial(\Gamma,\Pc)$ if and only if the following conditions hold.
    \begin{itemize}
        \item[(1)] If $(\gamma_n)_{n\in \Nb}\subset\Gamma$ is a sequence converging conically to a point in $\partial(\Gamma,\Pc)$, then $\dfrac{\sigma_{k}(\rho(\gamma_n))}{\sigma_{k+1}(\rho(\gamma_n))}\to +\infty$ as $n\to +\infty$ and \[\lim\limits_{n\to +\infty} (U_k(\rho(\gamma_n)),U_{d-k}(\rho(\gamma_n))) \in \Lambda~;\]
        \item[(2)] For any $p\in\partial_{par}(\Gamma,\Pc)$ with stabilizer $P\in\Pc^\Gamma$, there is a open set $C_p\subset \Fc_{k,d-k}$ that contains $\Lambda\setminus \zeta^{-1}(p)$, such that for any sequence of pairwise distinct elements $(\gamma_n)_{n\in \Nb}\subset P$, any compact subset $L\subset \Fc_{k,d-k}$ contained in $C_p$, and any open subset $U\subset\Fc_{k,d-k}$ containing $\zeta^{-1}(p)$, $\rho(\gamma_n)L\subset U$ for $n$ large enough.
    \end{itemize}
\end{proposition}

If $\zeta: \Lambda \to \partial (\Gamma, \Pc)$ is a refined boundary extension in $\Fc_{k,d-k}$, we can identify $\partial_{con}(\Gamma,\Pc)$ as a subset of $\Lambda$ that contains all the compactly attained conical limit points. Therefore, we can easily show the following proposition.

\begin{proposition}
    If $\zeta: \Lambda \to \partial (\Gamma, \Pc)$ is a refined boundary extension in $\Fc_{k,d-k}$, then there is a natural one-to-one correspondence between the collection of cocompact subflows of $\Fc(\Lambda,\zeta,X)$ and the collection of cocompact subflows of $\Fc(X)$ through $\pi:\Fc(\Lambda,\zeta,X)\to \Fc(X)$.
\end{proposition}

Then by Theorem \ref{TheoremCompactSubflow} (1)$\Leftrightarrow$(3), we deduce that

\begin{corollary}
    If $\rho$ is extended geometrically finite with a refined boundary extension $\zeta: \Lambda \to \partial (\Gamma, \Pc)$ in $\Fc_{k,d-k}$, then $\rho$ is $P_k$-Anosov in restriction to any cocompact subflow of $\Fc(X)$, and hence $P_k$-Anosov in restriction to any cocompact subflow of $\Fc(\Lambda,\zeta,X)$.
\end{corollary}

\subsection{Asymptotically embedded and relatively Anosov representations}\label{SectionRelationAsymptoticallyEmbeddedandandRelativelyAnosov}
The notion of Asymptotically embedded representation introduced by Kapovich--Leeb \cite{KL}, and the notion of relatively Anosov representation introduced by Zhu--Zimmer \cite{ZZ1} are also meant to generalize Anosov representations for relatively hyperbolic groups. It was shown in \cite{ZZ1} that they are equivalent. Then Weisman \cite{Weisman} showed that asymptotically embedded is a special case of extended geometrically finite when the boundary extension is a homeomorphism. We provide the details in this section.

Recall that a sequence $(g_n)_{n\in \Nb}\subset \SL(d,\Kb)$ is \emph{$P_k$-divergent} if \[\lim\limits_{n\to +\infty} \dfrac{\sigma_{k}(g_n)}{\sigma_{k+1}(g_n)}=+\infty~,\] and representation $\rho:\Gamma\to \SL(d,\Kb)$ is \emph{$P_k$-divergent} if for any sequence of pairwise distinct elements $(\gamma_n)_{n\in \Nb}\subset\Gamma$, $(\rho(\gamma_n) )_{n\in \Nb}\subset \SL(d,\Kb)$ is a $P_k$-divergent sequence. In case $\rho:\Gamma\to \SL(d,\Kb)$ is a $P_k$-divergent, the \emph{limit set} of $\rho$ in $\Fc_{k,d-k}$ is given to be \[\Lc(\rho) = \{ \lim\limits_{n\to +\infty} (U_k(\rho(\gamma_n)),U_{d-k}(\rho(\gamma_n))): (\gamma_n)_{n\in \Nb}\text{ a sequence of}\] \[\text{pairwise distinct elements in }\Gamma\}\subset \Fc_{k,d-k}~.\]

\begin{definition}\cite[Definition 7.1]{KL}\label{DefAsymptoticallyEmbedded}
    A representation $\rho:\Gamma\to \SL(d,\Kb)$ is $P_k$-asymptotically embedded relative to $\Pc$ if $\rho$ is $P_k$-divergent, and there exists a $\rho$-equivariant homeomorphism, called the \emph{limit map}, \[\xi:\partial(\Gamma,\Pc)\to \Lc(\rho)\] such that if $p\ne q\in \partial(\Gamma,\Pc)$, $\xi(p)$ and $\xi(q)$ are transverse.
\end{definition}
\begin{remark}\label{RelationAsymptoticallyEmbeddedandEGF}
    In the context of the definition, the map $\xi^{-1}:\Lc(\rho)\to \partial(\Gamma,\Pc)$ is a transverse boundary extension in $\Fc_{k,d-k}$. Then we have that a representation $\rho:\Gamma\to \SL(d,\Kb)$ is $P_k$-asymptotically embedded relative to $\Pc$ with limit map $\xi:\partial(\Gamma,\Pc)\to \Lc(\rho)$ if and only if $\rho$ is extended geometrically finite with a homeomorphic, transverse boundary extension $\xi^{-1}:\Lc(\rho)\to \partial(\Gamma,\Pc)$ in $\Fc_{k,d-k}$ (see \cite[Theorem 1.7]{Weisman}). In general, we will see in Proposition \ref{RefinedtoMinimal} that for a representation $\rho:\Gamma\to \SL(d,\Kb)$ that is $P_k$-divergent and extended geometrically finite relative to $\Pc$, a canonical choice of the boundary extension is also a $\rho$-equivariant map from $\Lc(\rho)$ to $\partial(\Gamma,\Pc)$.
\end{remark}

The definition of relatively Anosov representation is similar to asymptotically embedded representation. Instead of asking the representation to be divergent and the limit map to be a homeomorphism to its image, it requires the strongly dynamics preserving property.

\begin{definition}\cite[Definition 1.1]{ZZ1}\label{DefRelativelyAnosov}
    A representation $\rho:\Gamma\to \SL(d,\Kb)$ is $P_k$-Anosov relative to $\Pc$ if there exists a continuous, $\rho$-equivariant map \[\xi:\partial(\Gamma,\Pc)\to \Fc_{k,d-k}\] that is
    \begin{itemize}
        \item[(1)] (transverse) For any $p\ne q\in \partial(\Gamma,\Pc)$, $\xi(p)$ and $\xi(q)$ are transverse.
        \item[(2)] (strongly dynamics preserving) For any sequence of pairwise distinct elements $(\gamma_n)_{n\in \Nb}\subset \Gamma$ with $\gamma_n\to p$ and $\gamma_n^{-1}\to q$ as $n\to +\infty$, any compact subset $L\subset \Fc_{k,d-k}$ transverse to $\xi(y)$, and any open subset $U\subset \Fc_{k,d-k}$ that contains $\xi(x)$ $\rho(\gamma_n)L \subset U$ for $n$ large enough.
    \end{itemize}
\end{definition}

It is not hard to see that these two definitions are equivalent. Actually, if a representation $\rho:\Gamma\to \SL(d,\Kb)$ is $P_k$-asymptotically embedded relative to $\Pc$, the limit map is strongly dynamics preserving directly following from Lemma \ref{DivergenceLemma} (2)$\Rightarrow$(1). If $\rho$ is $P_k$-Anosov relative to $\Pc$, then the limit map is a homeomorphism since it is continuous and transverse, and $\rho$ is $P_k$-divergent which follows Lemma \ref{DivergenceLemma} (1)$\Rightarrow$(2). See \cite[Proposition 4.4]{ZZ1} for a detailed proof.

\section{Extended geometric finiteness via flows}\label{Section2to1}
We show that (2) implies (1) in Theorem \ref{MainTheorem} in this section. 

We recall the following lemmas for preparation. Let $d_{\angle}$ be the angle distance on $\Gr_k(\Kb^d)$ or $\Gr_{d-k}(\Kb^d)$ with respect to the standard metric on $\Kb^d$.

\begin{lemma}\cite[Lemma A.4, Lemma A.5]{BPS}\label{EstimationSingularValues}
    Let $g,h$ be two invertible matrices. If $g$ and $gh$ have singular value gaps at index $k$. Then
    \begin{itemize}
        \item[(1)] $d_{\angle}(U_k(g),U_k(gh))\leqslant \norm{h} \norm{h^{-1}} \dfrac{\sigma_{k+1}(g)}{\sigma_k(g)};$
        \item[(2)] $d_{\angle}(g U_k(h),U_k(gh))\leqslant \norm{g} \norm{g^{-1}} \dfrac{\sigma_{k+1}(h)}{\sigma_k(h)}.$
    \end{itemize}
\end{lemma}

The following are the key lemmas related to the strongly dynamics preserving property.

\begin{lemma}\textnormal{(see} \cite[Lemma 2.2]{CZZ} \textnormal{or} \cite[Lemma 4.19]{KLP17}\textnormal{)}\label{DivergenceLemma}
    Let $V_0\in \Gr_k(\Kb^d)$, $W_0\in \Gr_{d-k}(\Kb^d)$ and $(g_n)_{n\in \Nb}$ a sequence in $\SL(d,\Kb)$. Then the following are equivalent.
    \begin{itemize}
        \item[(1)] $g_n V \to V_0$ as $n\to +\infty$ for any $V\in \Gr_k(\Kb^d)$ transverse to $W_0$ with the convergence uniform on any compact subset of $\Gr_k(\Kb^d)$ that transverse to $W_0$.
        \item[(2)] $\dfrac{\sigma_{k}(g_n)}{\sigma_{k+1}(g_n)}\to +\infty$, $U_k(g_n)\to V_0$ and $U_{d-k}(g_n^{-1})\to W_0$ as $n\to +\infty$
        \item[(3)] There are open set $O\subset Gr_k(\Kb^d)$ and $O'\subset Gr_{d-k}(\Kb^d)$ such that $g_n V \to V_0$ and $g_n^{-1}W \to W_0$ as $n\to +\infty$, for any $V\in O$ and $W\in O'$.
    \end{itemize}
\end{lemma}

\begin{lemma}\label{ExtendedDivergenceLemma}
    Let $(g_n)_{n\in \Nb}$ be a $P_k$-divergent sequence in $\SL(d,\Kb)$. Let $A\subset\Gr_k(\Kb^d)$ and $B\subset\Gr_{d-k}(\Kb^d)$ be two closed proper subsets. Then the following are equivalent.
    \begin{itemize}
        \item[(1)] For any compact set $L$ in $\Gr_k(\Kb^d)$ transverse to $B$ and any open set $U$ in $\Gr_k(\Kb^d)$ that contains $A$, $g_n L\subset U$ for $n$ large enough, and for any compact set $L'$ in $\Gr_{d-k}(\Kb^d)$ transverse to $A$ and any open set $U'$ in $\Gr_{d-k}(\Kb^d)$ that contains $B$, $g_n^{-1} L'\subset U'$ for $n$ large enough.
        \item[(2)] $U_k(g_n)$ has all accumulation points contained in $A$ and $U_{d-k}(g_n^{-1})$ has all accumulation points contained in $B$.
    \end{itemize}
\end{lemma}
\begin{proof}
    We firstly check (2) implies (1). Suppose there exists (up to a subsequence) a sequence $(V_n)_{n\in \Nb}\subset\Gr_k(\Kb^d)$ that are all contained in a compact set $L$ transverse to $B$ but with $g_n V_n\to V\not\in A$ as $n\to +\infty$. By taking a subsequence of $(g_n)_{n\in \Nb}$, we may assume $U_k(g_n)\to V_0$ and $U_{d-k}(g_n^{-1})\to W_0$ as $n\to +\infty$ for some $V_0\in A$ and $W_0\in B$. Then by Lemma \ref{DivergenceLemma}, for any open set $U$ that contains $A$, there exists a constant $N\geqslant 0$, such that $g_n L\subset U$ for $n\geqslant N$, which gives a contradiction when $V\not\in U$. The other part of (1) can be proved similarly.

    Now we check that (1) implies (2). Up to a subsequence, we assume $U_k(g_n)\to V_0$ as $n\to +\infty$. If we write $g_n = P_n A_n Q_n$, where $P_n,Q_n\in \SO(d)$ and $A_n = \diag(\sigma_1(g_n),\sigma_2(g_n),...,\sigma_d(g_n))$ with $P_n\to P$ and $Q_n\to Q$ as $n\to +\infty$. Since $g_n$ is $P_k$-divergent, we have $V_0 = P\cdot\Span(e_1,e_2,...,e_k)$. By picking a subset $L$ in an open set in $\Gr_k(\Kb^d)$ transverse to $Q^{-1}\cdot\Span(e_{k+1},...,e_d)$ and $B$, we see that $g_n L\to V_0$ as $n\to +\infty$ and hence $V_0\in A$. The proof of the other part of (2) is similar.
\end{proof}

Let $(\Gamma,\Pc)$ be a relatively hyperbolic group with $X$ a Gromov model of $(\Gamma,\Pc)$ and let $\zeta: \Lambda \to \partial (\Gamma, \Pc)$ be a minimal boundary extension of $(\Gamma,\Pc)$. Let $\rho:\Gamma\to\SL(d,\Kb)$ be a representation which is $P_k$-Anosov in restriction to $\Fc(\Lambda,\zeta,X)$.
Following Remark \ref{LiftAbuseNotations}, there is a dominated splitting $\Fc(\Lambda,\zeta,X)\times \Kb^d = E^s\oplus E^u$ of rank $k$, with respect to a $\rho(\Gamma)$-invariant metric $\norm{\cdot}$ on $\Fc(\Lambda,\zeta,X)\times \Kb^d$. We then restate and prove Theorem \ref{MainTheorem} (2) $\Rightarrow$ (1).

\begin{theorem}\label{TheoremRestrictedtoEGF}
    Let $\zeta: \Lambda \to \partial (\Gamma, \Pc)$ be a minimal boundary extension of $(\Gamma,\Pc)$ and let $X$ be a Gromov model of $(\Gamma,\Pc)$. If $\rho$ is $P_k$-Anosov in restriction to $\Fc(\Lambda,\zeta,X)$, then $\rho$ is $P_k$-divergent and there exists a $\rho$-equivariant, $\zeta$-transverse, continuous map $\xi = (\xi^k,\xi^{d-k}): \Lambda \to \Fc_{k,d-k}$ such that $\rho$ is extended geometrically finite with boundary extension $(\xi(\Lambda),\zeta\circ\xi^{-1})$.
\end{theorem}

We say a map 
\begin{align*}
    \xi: \Lambda & \to \Fc_{k,d-k} \\
    x & \mapsto \xi(x) = (\xi^k(x),\xi^{d-k}(x))
\end{align*}
is $\rho$-equivariant if $\rho(\gamma) \xi(x) =\xi(\gamma x)$ for any $\gamma\in\Gamma$ and $x\in\Lambda$, is $\zeta$-transverse if for any $p,q\in\partial(\Gamma,\Pc)$, $\xi^k(\zeta^{-1}(p))$ and $\xi^{d-k}(\zeta^{-1}(q))$ are transverse.

We firstly show the existence of a $\rho$-equivariant, $\zeta$-transverse, continuous map $\xi = (\xi^k,\xi^{d-k}): \Lambda \to \Fc_{k,d-k}$ who defines the dominated splitting as described in Proposition \ref{LimitMapWaivePro}. Then we show that the $\xi$-image of $\Lambda$ provides a boundary extension to make $\rho$ extended geometrically finite.

\begin{proposition}\label{LimitMapWaivePro}
    There exists a $\zeta$-transverse, continuous, $\rho$-equivariant map $\xi = (\xi^k,\xi^{d-k}): \Lambda \to \Fc_{k,d-k}$ with the dominated splitting $\Fc(\Lambda,\zeta,X)\times \Kb^d = E^s\oplus E^u$ given by $(E^s_z, E^u_z) = (\xi^k(z^+),\xi^{d-k}(z^-))$ for any $z = (z^+,z^-,\ell) \in\Fc(\Lambda,\zeta,X)$.    
\end{proposition}
\begin{proof}
    Since $\rho$ is $P_k$-Anosov in restriction to $\Fc(\Lambda,\zeta,X)$, then $\rho$ is $P_k$-Anosov in restriction to any cocompact subflow of $\Fc(\Lambda,\zeta,X)$. By Theorem \ref{TheoremCompactSubflow}, there exists a pair of transverse limit maps $\xi^k: \partial_{cc}X\to \Gr_k(\Kb^d)$ and $\xi^{d-k}: \partial_{cc}X\to \Gr_{d-k}(\Kb^d)$, such that $(E^s_z, E^u_z) = (\xi^k(z^+),\xi^{d-k}(z^-))$ if $z = (z^+,z^-,\ell) \in\Fc(\Lambda,\zeta,X)$ with $z^+,z^-\in \partial_{cc}X$, where we identify $\partial_{cc}X\subset \partial_{con}X$ as a subset of $\Lambda$. This means that when we fix a limit point $z^+\in \partial_{cc}(X)$, $\xi^k(z^+) = E^s_{(z^+,z^-,\ell)}$ is a constant, i.e., independent of the choice of $z^-\in \partial_{cc}X$. Notice that $\partial_{cc}X$ is a dense subset of $\Lambda$ as $\zeta: \Lambda \to \partial (\Gamma, \Pc)$ is a minimal boundary extension and $E^s_{(z^+,z^-,\ell)}$ is continuous for the variable $z^-$. We deduce that for any fixed $z^+\in \partial_{cc}(X)$, $\xi^k(z^+) = E^s_{(z^+,z^-,\ell)}$ is independent of the choice of $z^-\in \Lambda$ and $\ell$ with $(\ell(+\infty),\ell(-\infty))=(z^+,z^-)$ by Remark \ref{asymptoticrigid}. Apply the fact that $E^s_{(z^+,z^-,\ell)}$ is also continuous for the variable $z^+$, we have that for any fixed $z^+\in \Lambda$, $E^s_{(z^+,z^-,\ell)}$ is a constant denoted by $\xi^k(z^+)$. We obtain that $\xi^k: \Lambda \to \Gr_k(\Kb^d)$ is a continuous, $\rho$-equivariant map such that $E^s_z = \xi^k(z^+)$. Similarly, we have a well-defined, continuous, $\rho$-equivariant map $\xi^{d-k}: \Lambda \to \Gr_{d-k}(\Kb^d)$ such that $E^s_z = \xi^k(z^+)$. $\xi^k$ and $\xi^{d-k}$ are compatible, i.e., $\xi=(\xi^k,\xi^{d-k})$ has image in $\Fc_{k,d-k}$, since they are compatible when they are restricted $\partial_{cc}X$. Moreover, $\xi^k$ and $\xi^{d-k}$ are $\zeta$-transverse since they define a dominated splitting over $\Fc(\Lambda,\zeta,X)$.
\end{proof}

\begin{lemma}
    $\zeta\circ\xi^{-1}: \xi(\Lambda)\to \partial(\Gamma,\Pc)$ is a well-defined minimal boundary extension of $(\Gamma,\Pc)$.
\end{lemma}
\begin{proof}
    If $x,y\in \Lambda$ with $\xi(x) =\xi(y)$, then $\zeta(x) = \zeta(y)$ since otherwise, $\xi(x)$ and $\xi(y)$ would be transverse. This implies that $\zeta\circ\xi^{-1}$ is well-defined and $\xi$ is injective on $\partial_{con}(\Gamma,\Pc)$. It is minimal since $\zeta: \Lambda \to \partial (\Gamma, \Pc)$ is minimal.
\end{proof}

Let $\norm{\cdot}_0$ be the standard metric on $\Kb^d$. Let $A: \Fc(\Lambda,\zeta,X) \to \GL(d,\Kb)$ be a continuous map such that $\norm{\cdot}_0 = \norm{A_z\cdot}_z$ for any $z\in \Fc(\Lambda,\zeta,X)$.

The proof of \cite[Lemma 6.2 and Lemma 6.3]{ZZ1} apply here, since they only depend on the dominated splitting, which tell that

\begin{lemma}\label{DominatedSplittingMetricTransEstimation}
    (1). There exist constants $C,c > 0$, such that \[\dfrac{\sigma_{k+1}(A_z^{-1}A_{\phi^t(z)})}{\sigma_{k}(A_z^{-1}A_{\phi^t(z)})}\leqslant C e^{-c t}\] for any $z\in \Fc(\Lambda,\zeta,X)$ and $t\in \Rb_{\geqslant 0}$.\\
    (2). \[\lim\limits_{t\to +\infty} \sup\limits_{z\in \Fc(\Lambda,\zeta,X)} d(U_k(A_{\phi^t(z)}^{-1}A_z),A_z^{-1}E^s_z) = 0~.\]
\end{lemma}

\begin{proof}[Proof of Theorem \ref{TheoremRestrictedtoEGF}]
    We fix a base point $z_0\in \Fc(\Lambda,\zeta,X)$. By Lemma \ref{TrackingLemma}, there exists a compact set $K\subset \Fc(\Lambda,\zeta,X)$ that contains $z_0$ and constants $\lambda\geqslant 1$ and $\epsilon \geqslant 0$, with the property that for any $\gamma\in \Gamma$, there exists $z_\gamma\in K$ and $t_\gamma\in\Rb$, such that $\gamma^{-1}\phi^{t_\gamma}(z_\gamma) \in K$ and $\abs{\gamma}_X \sim_{(\lambda,\epsilon)} \abs{t}$. Here we can assume $t_\gamma \geqslant 0$, as we may replace $z_\gamma$ by its reverse $\hat{z_\gamma}$.

    Since $K$ is compact, there exists a constant $C_K\geqslant 1$ such that \[ \norm{\cdot}_z =  \norm{A_z^{-1} \cdot}_0 \sim_{(C_K,0)} \norm{\cdot}_0\] for any $z\in K$. Then for any $\gamma\in\Gamma$, \[ \norm{\rho(\gamma)^{-1}\cdot}_0 \sim_{(C_K,0)} \norm{\rho(\gamma)^{-1}\cdot}_{\gamma^{-1}\phi^{t_\gamma}(z_\gamma)} = \norm{\cdot}_{\phi^{t_\gamma}(z_\gamma)}~.\] On the other hand, we have $\norm{\cdot}_{\phi^{t_\gamma}(z_\gamma)} = \norm{A_{\phi^{t_\gamma}(z_\gamma)}\cdot}_0$, and \[ \norm{A_{\phi^{t_\gamma}(z_\gamma)}\cdot}_0 \sim_{(C_K,0)} \norm{A_{\phi^{t_\gamma}(z_\gamma)}\cdot}_{z_\gamma} = \norm{A_{z_\gamma}^{-1}A_{\phi^{t_\gamma}(z_\gamma)}\cdot}_0~.\] Therefore \[ \norm{\rho(\gamma)^{-1}\cdot}_0 \sim_{(C_K^2,0)} \norm{A_{z_\gamma}^{-1}A_{\phi^{t_\gamma}(z_\gamma)}\cdot}_0~.\] Then by Lemma\ref{DominatedSplittingMetricTransEstimation} (1), 
    \begin{equation}\label{eqweaklydominated}
    \dfrac{\sigma_{k+1}(\rho(\gamma)^{-1})}{\sigma_{k}(\rho(\gamma)^{-1})}\leqslant C_K' \dfrac{\sigma_{k+1}(A_{z_\gamma}^{-1}A_{\phi^{t_\gamma}(z_\gamma)})}{\sigma_{k}(A_{z_\gamma}^{-1}A_{\phi^{t_\gamma}(z_\gamma)})}\leqslant C_K'C e^{-c {t_\gamma}}\leqslant C_K'Ce^{c\epsilon} e^{-c\lambda^{-1} \abs{\gamma}_X}
    \end{equation}
    where $C_K'$ is some power of $C_K$. Hence $\rho$ is $P_k$-divergent.

    Now we show that $\rho$ is extended geometrically finite with boundary extension $(\xi(\Lambda),\zeta\circ\xi^{-1})$. Let $(\gamma_n)_{n\in \Nb}$ be a sequence in $\Gamma$, with $\gamma_n \to p \in \partial(\Gamma,\Pc)$ and $\gamma_n^{-1}\to q \in \partial(\Gamma,\Pc)$ as $n\to +\infty$. Then $(\rho(\gamma_n) )_{n\in \Nb}$ is a $P_k$-divergent sequence. By Lemma \ref{ExtendedDivergenceLemma}, it suffices to show that $U_k(\rho(\gamma_n))$ has all accumulation points contained in $\xi^k(\zeta^{-1}(p))$ and $U_{d-k}(\rho(\gamma_n)^{-1})$ has all accumulation points contained in $\xi^{d-k}(\zeta^{-1}(q))$. We apply Lemma \ref{TrackingLemma} here again, then there exists $z_n\in K$ and $t_n \in \Rb_{\geqslant 0}$, such that $\gamma_n^{-1}\phi^{t_n}(z_n) \in K$ for each $n\in \Nb$. Then $t_n\to +\infty$ since $(\gamma_n)_{n\in \Nb}$ is a $P_k$-divergent sequence. Up to a subsequence, we may assume $U_k(\rho(\gamma_n))$ and $U_{d-k}(\rho(\gamma_n)^{-1})$ are convergent as $n\to +\infty$. Similar to the argument above, \[\norm{A_{\phi^{t_n}(z_n)}^{-1}\rho(\gamma_n)\cdot}_0 = \norm{\rho(\gamma_n)\cdot}_{\phi^{t_n}(z_n)} = \norm{\cdot}_{\gamma^{-1}\phi^{t_n}(z_n)}\sim_{(C_K,0)} \norm{\cdot}_0~.\] Therefore $\norm{\rho(\gamma_n)^{-1}A_{\phi^{t_n}(z_n)}}\norm{A_{\phi^{t_n}(z_n)}^{-1}\rho(\gamma_n)}$ is uniformly bounded for any $n$. By Lemma \ref{EstimationSingularValues} (1), 
    \[ d_{\angle}(U_k(\rho(\gamma_n)),U_k(A_{\phi^{t_n}(z_n)}))\leqslant \norm{\rho(\gamma_n)^{-1}A_{\phi^{t_n}(z_n)}}\norm{A_{\phi^{t_n}(z_n)}^{-1}\rho(\gamma_n)} \dfrac{\sigma_{k+1}(\rho(\gamma_n))}{\sigma_k(\rho(\gamma_n))}~.\]
    Up to a subsequence, we assume that $z_n\to z = (z^+,z^-,t)\in K$ as $n\to +\infty$, then $\zeta(z^+) = p$.
    Then \[ \lim\limits_{n\to +\infty} U_k(\rho(\gamma_n)) = \lim\limits_{n\to +\infty} U_k(A_{\phi^{t_n}(z_n)})~.\] On the other hand, by Lemma \ref{EstimationSingularValues} (2), \[d_{\angle}(A_{z_n} U_k(A_{z_n}^{-1}A_{\phi^{t_n}(z_n)}),U_k(A_{\phi^{t_n}(z_n)}))\leqslant \norm{A_{z_n}} \norm{A_{z_n}^{-1}} \dfrac{\sigma_{k+1}(A_{z_n}^{-1}A_{\phi^{t_n}(z_n)})}{\sigma_k(A_{z_n}^{-1}A_{\phi^{t_n}(z_n)})}~.\] Since $A_{z_n}$ remains in a compact set and $\dfrac{\sigma_{k+1}(A_{z_n}^{-1}A_{\phi^{t_n}(z_n)})}{\sigma_k(A_{z_n}^{-1}A_{\phi^{t_n}(z_n)})}\to 0$ as $n\to +\infty$ by Lemma \ref{DominatedSplittingMetricTransEstimation} (1), then we have
    \begin{align*}
    \lim\limits_{n\to +\infty} U_k(A_{\phi^{t_n}(z_n)}) & = \lim\limits_{n\to +\infty} A_{z_n} U_k(A_{z_n}^{-1}A_{\phi^{t_n}(z_n)})\\
    & = \lim\limits_{n\to +\infty} E_{z_n}^s\\
    & = E_{z}^s = \xi^k(z^+)\in \xi^k(\zeta^{-1}(p))~,
    \end{align*}
    where the last equality follows from $A_{z_n}$ remaining in a compact set and Lemma \ref{DominatedSplittingMetricTransEstimation} (2). Similarly, we have \[\lim\limits_{n\to +\infty} U_{d-k}(\rho(\gamma_n)^{-1}) \in \xi^{d-k}(\zeta^{-1}(q))~.\]
    Therefore, $\rho$ is extended geometrically finite with boundary extension $(\xi(\Lambda),\zeta\circ\xi^{-1})$.
\end{proof}

Recall Definition \ref{DefWeaklyDominated}. We deduce directly the ``moreover part'' of Theorem \ref{MainTheorem} from Inequality \ref{eqweaklydominated} in the argument above.

\begin{corollary}\label{WeaklyDominatedCorollary}
    $\rho$ is $P_k$-weakly dominated with respect to $X$.
\end{corollary}

\section{Restricted Anosov from extended geometrically finite}\label{Section1to2}
\subsection{Proof of the Main Theorem}\label{Section1to2theFirst} Let $(\Gamma, \Pc)$ be a relatively hyperbolic group. We show that (1) implies (2) in Theorem \ref{MainTheorem} in this section. More concretely,
\begin{theorem}\label{TheoremEGFtoRestricted}
    If a representation $\rho:\Gamma\to \SL(d,\Kb)$ is $P_k$-divergent and extended geometrically finite with a refined boundary extension $\zeta: \Lambda \to \partial (\Gamma, \Pc)$ in $\Fc_{k,d-k}$, then $\zeta: \Lambda \to \partial (\Gamma, \Pc)$ is minimal, and there exists a Gromov model $X$ of $(\Gamma,\Pc)$, such that $\rho$ is $P_k$-Anosov in restriction to $\Fc(\Lambda,\zeta,X)$.
\end{theorem}

In the assumption of the theorem, $\Lambda$ is already a subset of $\Fc_{k,d-k}$, but for convenience, we still denote $x = (\xi^k(x),\xi^{d-k}(x))\in\Fc_{k,d-k}$, where $\xi^k$ ($\xi^{d-k})$ is just the map taking the $k$-dimensional ($d-k$-dimensional) subspace from a flag. The idea of the proof is as follows. We show that $\zeta: \Lambda \to \partial (\Gamma, \Pc)$ is minimal in Proposition \ref{RefinedtoMinimal}. Then, we find a proper Gromov model $X$ of $(\Gamma,\Pc)$ such that  $\rho|_P$ is $P_k$-weakly dominated with respect to $X$ for each $P\in\Pc$. Following an argument from \cite{CZZ} (for Anosov representations of geometrically finite Fuchsian groups) and \cite{ZZ1} (for relatively Anosov representations), we construct a $\rho(\Gamma)$-invariant, reverse invariant metric on $\Fc(\Lambda,\zeta,X)\times \Kb^d$ and show that the decomposition defined by $\Fc(\Lambda,\zeta,X)\times \Kb^d = E^s \oplus E^u$, $(E^s_z,E^u_z) = (\xi^k(z^+),\xi^{d-k}(z^-))$ at any point $z=(z^+,z^-,\ell)\in \Fc(\Lambda,\zeta,X)$, is a dominated splitting of rank $k$ with respect to this metric.

Before starting the proof of Theorem \ref{TheoremEGFtoRestricted}, we show the following proposition, and as two concequence, we deduce Theorem \ref{MainRefinedtoMinimal} and the first statement of Theorem \ref{TheoremEGFtoRestricted} that the boundary extension is minimal.

\begin{proposition}\label{RefinedtoMinimal}
    Let $\rho:\Gamma\to \SL(d,\Kb)$ be a representation that is $P_k$-divergent and extended geometrically finite with a refined boundary extension $\zeta: \Lambda \to \partial (\Gamma, \Pc)$ in $\Fc_{k,d-k}$. Then the limit set $\Lc(\rho)$ of $\rho$ in $\Fc_{k,d-k}$ is identified with $\Lambda$. If $(\gamma_n)_{n\in \Nb}\subset \Gamma$ is a sequence with $\gamma_n\to q\in\partial_{con}(\Gamma,\Pc)$ then $\zeta^{-1}(q)$ is the singleton \[\lim\limits_{n\to +\infty} (U_k(\rho(\gamma_n)),U_{d-k}(\rho(\gamma_n)))\] and for any $p\in\partial_{par}(\Gamma,\Pc)$, \[\zeta^{-1}(p) = \{ \lim\limits_{n\to +\infty} (U_k(\rho(\gamma_n)),U_{d-k}(\rho(\gamma_n))): (\gamma_n)_{n\in \Nb}\subset P\] \[\text{ a sequence of pairwise distinct elements}\}~.\]
    In particular, $\zeta: \Lambda \to \partial (\Gamma, \Pc)$ is minimal.
\end{proposition}
\begin{proof}
    Since $\zeta: \Lambda \to \partial (\Gamma, \Pc)$ is a refined boundary extension, by Proposition \ref{BoundaryExtensionRefinement}, for any $q\in\partial_{con}(\Gamma,\Pc)$, $\zeta^{-1}(q)$ is a singleton, and for any $p\in\partial_{par}(\Gamma,\Pc)$, $\zeta^{-1}(p)$ is the closure of all accumulation points of $\gamma_nx$ where $(\gamma_n)_{n\in \Nb}$ is a sequence of elements in the stabilizer of $p$ and $x\in C_p$.\\ \\
    \textit{Conical limit points.}
    If $q\in \partial_{con}(\Gamma,\Pc)$, let $(\gamma_n)_{n\in \Nb}\subset \Gamma$ be a sequence such that $\gamma_n\to q$ as $n\to +\infty$. Assume that $\gamma_n^{-1}\to q'$ as $n\to +\infty$ up to a subsequence. By the definition of extended geometrically finite representation, there exists an open set $C_{p'}$ such that for any $x\in C_{p'}$, $\rho(\gamma_n)x$ converges to the singleton $\zeta^{-1}(q)$ as $n\to +\infty$. Therefore, by Lemma \ref{DivergenceLemma}, \[\zeta^{-1}(q) = \lim\limits_{n\to +\infty} (U_k(\rho(\gamma_n)),U_{d-k}(\rho(\gamma_n)))~.\]\ \\
    \textit{Parabolic points.}
    Let $p$ be a parabolic point with stabilizer $P\in \Pc^\Gamma$. Let $y\in \zeta^{-1}(p)$ be such that there is a sequence $(\gamma_n)_{n\in \Nb}\subset P$ and $x\in C_p$ with $\rho(\gamma_n) x \to y$ as $n\to +\infty$. Since $\rho$ is $P_k$-divergent, $(\gamma_n)_{n\in \Nb}$ is a $P_k$-divergent sequence, we may assume that up to a subsequence, $(U_k(\rho(\gamma_n)),U_{d-k}(\rho(\gamma_n)))\to y'$ and $(U_k(\rho(\gamma_n^{-1})),U_{d-k}(\rho(\gamma_n^{-1})))\to y''$ as $n\to +\infty$. Then by Lemma \ref{DivergenceLemma}, for any $x'$ transverse to $y''$, $\gamma_n x'\to y'$ as $n\to +\infty$. The set of flags transverse to $y''$ is open and dense in $\Fc_{k,d-k}$, hence intersects $C_p$ non-trivially. Then we have $y' \in \zeta^{-1}(p)$ and similarly, $y'' \in \zeta^{-1}(p)$. Since $x\in C_p$ is transverse to $\zeta^{-1}(p)$, $\rho(\gamma_n) x \to y' = y$ as $n\to +\infty$. This implies that \[\zeta^{-1}(p) = \{ \lim\limits_{n\to +\infty} (U_k(\rho(\gamma_n)),U_{d-k}(\rho(\gamma_n))): (\gamma_n)_{n\in \Nb}\subset P\] \[\text{ a sequence of pairwise distinct elements}\}~,\] as it is already a closed subset of $\Fc_{k,d-k}$. 
    
    Finally, since $\Lambda\setminus \zeta^{-1}(p)$ is transverse to $\zeta^{-1}(p)$ and hence transverse to $y''$, then for any $x'\in \Lambda\setminus \zeta^{-1}(p)$, $\rho(\gamma_n) x' \to y' = y$ as $n\to +\infty$. Therefore, $\zeta: \Lambda \to \partial (\Gamma, \Pc)$ is minimal.\\ \\
    \textit{$\Lambda = \Lc(\rho)$.} We know from above that $\Lambda\subset \Lc(\rho)$. If $y \in \Lc(\rho)$ is expressed as $y = \lim\limits_{n\to +\infty} (U_k(\rho(\gamma_n)),U_{d-k}(\rho(\gamma_n)))$ for some sequence $(\gamma_n)_{n\in \Nb}\subset \Gamma$. We assume that $\gamma_n\to q$, $\gamma_n^{-1}\to q'$ and $(U_k(\rho(\gamma_n^{-1})),U_{d-k}(\rho(\gamma_n^{-1})))\to y'$ as $n\to +\infty$ up to a subsequence. The intersection of $C_{q'}$ and the set of flags transverse to $y'$ is nonempty open, then $y\in \zeta^{-1}(p)\subset \Lambda$ by definition. Therefore $\Lc(\rho)\subset \Lambda$.
\end{proof}

\begin{remark}\label{OpenSettobeWholeTransverse}
    The proposition tells that when $\rho$ is $P_k$-divergent and extended geometrically finite relative to $\Pc$, a unique choice of refined boundary extensions is given by the limit set. Then by Lemma \ref{ExtendedDivergenceLemma}, we can take the set $C_p$ in Definition \ref{DefExtendedGeometricallyFinite} to be the set of all the flags transverse to $\zeta^{-1}(p)$, for any $p\in\partial(\Gamma,\Pc)$.
\end{remark}

Recall that for a parabolic subgroup $P\in \Pc$, we say that $\rho|_{P}$ is \emph{$P_k$-weakly dominated} with respect to a Gromov model $X$ of $(\Gamma,\Pc)$, if there exist constant $C,c > 0$, such that \[\dfrac{\sigma_{k+1}(\rho(\gamma))}{\sigma_{k}(\rho(\gamma))}\leqslant Ce^{-c\abs{\gamma}_X}~,\] for any $\gamma\in P$.

\begin{lemma}\label{GromovModelDivtoDom}
    There exists an increasing function $f:\Rb_{\geqslant 0} \to \Rb_{\geqslant 0}$, with $f(0) = 1$ and $f(t)\to +\infty$ as $t\to +\infty$, such that for any Gromov model $X\in \Xc_f(\Gamma,\Pc,S)$, the representation $\rho|_P$ is $P_k$-weakly dominated with respect to $X$ for each $P\in\Pc$.
\end{lemma}
\begin{proof}
    If $\rho|_P$ is $P_k$-divergent for a subgroup $P\in \Pc$, then there exists an increasing function $g:\Rb_{\geqslant 0} \to \Rb_{\geqslant 0}$, such that for any $\gamma\in P$, \[\dfrac{\sigma_{k}(\rho(\gamma))}{\sigma_{k+1}(\rho(\gamma))}\geqslant e^{g(\abs{\gamma}_S)}~.\] Otherwise, there exists a sequence $(\gamma_n)_{n\in \Nb}\subset P$, such that $\dfrac{\sigma_{k}(\rho(\gamma_n))}{\sigma_{k+1}(\rho(\gamma_n))}$ has a uniform upper bound, which contradicts $P_k$-divergence. We pick a function $f$ with $f(0)=1$ that increases fast enough such that $g(f(t))\geqslant t$ for any $t\geqslant 0$. Since for any $X\in \Xc_f(\Gamma,\Pc,S)$, there exist constants $\lambda\geqslant 1$ and $\epsilon\geqslant 0$, such that for any $\gamma\in P$, \[f(\lambda^{-1}\abs{\gamma}_X-\epsilon)\leqslant \abs{\gamma}_S~,\] 
    then \[\lambda^{-1}\abs{\gamma}_X-\epsilon \leqslant g(f(\lambda^{-1}\abs{\gamma}_X-\epsilon))\leqslant g(\abs{\gamma}_S)~,\] and hence \[\dfrac{\sigma_{k}(\rho(\gamma))}{\sigma_{k+1}(\rho(\gamma))}\geqslant e^{g(\abs{\gamma}_S)} \geqslant e^{-\epsilon}e^{\lambda^{-1}\abs{\gamma}_X}~,\] which shows $\rho|_P$ is $P_k$-weakly dominated. The constants can be taken to be uniform for all $P\in \Pc$ since $\Pc$ is finite.
\end{proof}

In the rest of this section, we always assume that $\rho:\Gamma\to \SL(d,\Kb)$ is a $P_k$-divergent, extended geometrically finite representation with a refined boundary extension $\zeta: \Lambda \to \partial (\Gamma, \Pc)$ in $\Fc_{k,d-k}$. By Lemma \ref{GromovModelDivtoDom}, we fix a Gromov model $X$ of $(\Gamma,\Pc)$ such that $\rho|_P$ is $P_k$-weakly dominated with respect to $X$ for each $P\in\Pc$ and we fix a thick-thin decomposition of $X$, \[X = X^{th} \cup (\bigcup_{B\in\Bc} B)~,\] so that the distance between any two horoballs in $\Bc$ is large enough.

Let $p_M:\Fc(X)\to X$ be the quasi-isometric projection given in Proposition \ref{proposition: compact conical limit flow}. Recall that in Section \ref{BoundaryExtensionSection}, we defined the projection \[\pi:\Fc(\Lambda,\zeta,X)\to \Fc(X)~.\]

Let $R > 0$ be a constant, then $\{\Nc_R(p_M^{-1}(B)): B\in \Bc\}$ is a collection of mutually disjoint open sets. If $P\in \Pc^\Gamma$ is the stabilizer of $B\in \Bc$, let $\Fc_P = \pi^{-1}(\Nc_R(p_M^{-1}(B)))$, which can be understood as a horoball of the flow space $\Fc(\Lambda,\zeta,X)$. The reason that we take $R$-neighborhoods here is that we wish to make $\Fc_P$ open as $p_M$ is not necessarily continuous. Set $\Fc^{th} = \Fc(\Lambda,\zeta,X) \setminus (\bigcup_{P\in\Pc^\Gamma} \Fc_P)$, which gives \[\Fc(\Lambda,\zeta,X) = \Fc^{th} \cup (\bigcup_{P\in\Pc^\Gamma} \Fc_P)~,\] a thick-thin decomposition of $\Fc(\Lambda,\zeta,X)$.

Recall that $z=(x,y,\ell)\to \hat{z}=(y,x,\hat{\ell})$ denotes the reverse map on $\Fc(\Lambda,\zeta,X)$ defined in Section \ref{SectionBoundaryExtensions}.
Following the notations in \cite[Section 9]{ZZ1}, we denote \[\partial^+\Fc_P = \{z\in \Fc(\Lambda,\zeta,X): z\not\in \Fc_P\text{ and }\exists\ \epsilon > 0, \phi^{t}(z)\in \Fc_P, \forall\ t\in(0,\epsilon)\}~,\] \[\partial^-\Fc_P = \{z\in \Fc(\Lambda,\zeta,X): \hat{z}\in \partial^+\Fc_P\}~.\]
Since $P$ acts on $\overline{B}\setminus B$, the horosphere of $B$, cocompactly, we can easily see that $P$ acts on $\partial^+\Fc_P$ and $\partial^-\Fc_P$ cocompactly.
For $z\in \partial^+\Fc_P$, we denote \[T_z^+ = \min\{t\in (0,+\infty]: \phi^t(z)\not\in \Fc_P\}~,\] and $T_{\hat{z}}^- = - T_z^+$. Then \[\Fc_P = \big(\bigcup_{z\in \partial^+\Fc_P}\bigcup_{t\in (0,T_z^+)}\phi^t(z)\big)\cup\big(\bigcup_{z\in \partial^-\Fc_P}\bigcup_{t\in (T_z^-,0)}\phi^t(z)\big)~.\]

Let $E^s$ be the subbundle of $\Fc(\Lambda,\zeta,X)\times \Kb^d$ with fiber $E^s_z = \xi^k(z^+)$ at $z = (z^+,z^-,\ell)\in \Fc(\Lambda,\zeta,X)$. Similarly, let $E^u$ be the subbundle of $\Fc(\Lambda,\zeta,X)\times \Kb^d$ with fiber $E^u_z = \xi^{d-k}(z^-)$ at $z = (z^+,z^-,\ell)$. Then $\Fc(\Lambda,\zeta,X)\times \Kb^d = E^s \oplus E^u$ is a decomposition of $\Fc(\Lambda,\zeta,X)\times \Kb^d$ into $\Gamma$-invariant, reverse invariant, $\phi$-invariant subbundles.

Now we construct a $\rho(\Gamma)$-invariant, reverse invariant metric on $\Fc(\Lambda,\zeta,X)\times \Kb^d$, with respect to which, $\Fc(\Lambda,\zeta,X)\times \Kb^d = E^s \oplus E^u$ is a dominated splitting.

Since the $\Gamma$-action on $\Fc^{th}$ is cocompact, we fix a metric $\norm{\cdot}$ on $\Fc^{th}\times \Kb^d$ that is $\Gamma$-invariant and reverse invariant, such that for any $z = (z^+,z^-,\ell) \in \Fc^{th}$, $\xi^k(z^+)$ and $\xi^{d-k}(z^-)$ are orthogonal. Let $K\subset \Fc^{th}$ be a compact subset such that $\Gamma \cdot K = \Fc^{th}$.

\begin{lemma}\label{ThickPartPointsEstimation}
     There exists a constant $C_0 > 0$ such that for any $z\in K$, $t\in \Rb_{\geqslant 0}$ such that $\phi^t(z)\in \Fc^{th}$ and $\gamma\in\Gamma$ such that $\gamma^{-1}\phi^t(z)\in K$, we have \[ \dfrac{\norm{v}_{\phi^t(z)}}{\norm{w}_{\phi^t(z)}} \leqslant C_0 \dfrac{\sigma_{k+1}(\rho(\gamma))}{\sigma_{k}(\rho(\gamma))}\dfrac{\norm{v}_{z}}{\norm{w}_{z}}~,\] where $v\in E^s_z$ and $w\in E^u_z$ are nonzero vectors.
\end{lemma}
\begin{proof}
    We show that there exists a constant $C_0 > 0$, for any sequence $z_n = (z_n^+,z_n^-,\ell_n)\in K$, $t_n\in \Rb_{\geqslant 0}$ with $\phi^{t_n}(z_n)\in \Fc^{th}$ and $t_n\to +\infty$, $\gamma_n\in\Gamma$ with $\gamma_n^{-1}\phi^{t_n}(z_n)\in K$, and any vectors $v_n\in E^s_{z_n}$ and $w_n\in E^u_{z_n}$ with $\norm{v_n}_{z_n} = \norm{w_n}_{z_n} = 1$, 
    \[\dfrac{\norm{v_n}_{\phi^{t_n}(z_n)}}{\norm{w_n}_{\phi^{t_n}(z_n)}} \leqslant C_0 \dfrac{\sigma_{d-k+1}(\rho(\gamma_n^{-1}))}{\sigma_{d-k}(\rho(\gamma_n^{-1}))} \]
    Up to a subsequence, we assume that \[\gamma_n\to p\in\partial(\Gamma,\Pc),\quad \gamma_n^{-1}\to q\in\partial(\Gamma,\Pc)~,\] \[z_n \to z = (z^+,z^-,\ell)\text{ and } \gamma_n^{-1}\phi^{t_n}(z_n)\to z' = (z'^+,z'^-,\ell')\] as $n\to +\infty$. Then we have $\zeta(z^+) = p$ and $\zeta(z'^-) = q$ since $\pi(z_n)\to\pi(z) = \ell$ and $\pi(\gamma_n^{-1}\phi^{t_n}(z_n))\to\pi(z') = \ell'$ as $n\to +\infty$. 
    Notice also that the $\Gamma$-action on $\Lambda\subset \Fc_{k,d-k}$ is through $\rho$, hence \[\gamma_n^{-1}z_n^- = (\rho(\gamma_n^{-1})\xi^{k}(z_n^-) ,\rho(\gamma_n^{-1})\xi^{d-k}(z_n^-)) \to z'^- = (\xi^{k}(z'^-),\xi^{d-k}(z'^-))\]
    We may assume $\zeta(z_n^-)$ is contained in a compact subset of $\partial(\Gamma,\Pc)\setminus \{p\}$ when $n$ is large enough since $\zeta(z_n^+)\to \zeta(z^+) = p$ as $n\to +\infty$ and there is a geodesic with endpoint $\zeta(z_n^+)$ and $\zeta(z_n^-)$ who always passes through the compact set $\pi(K)$ for each $n$. Then $z_n^-$ is contained in a compact set transverse to $\zeta^{-1}(p)$ as $\zeta: \Lambda\to \partial(\Gamma,\Pc)$ is transverse.
    By Lemma \ref{TrackingLemma} (3), $\abs{\gamma_n}_X\to +\infty$ as $n\to +\infty$, hence we may assume that the elements of $(\gamma_n )_{n\in \Nb}$ are pairwise distinct. Then $(\rho(\gamma_n) )_{n\in \Nb}$ is a $P_k$-divergent sequence as $\rho$ is $P_k$-divergent. By Proposition \ref{RefinedtoMinimal}, $\lim\limits_{n\to +\infty} (U_k(\rho(\gamma_n)),U_{d-k}(\rho(\gamma_n)))$ is contained in $\zeta^{-1}(p)$. Then by Lemma \ref{DivergenceLemma}, \[\lim\limits_{n\to +\infty} \rho(\gamma_n^{-1})\xi^{d-k}(z_n^-) = \lim\limits_{n\to +\infty} U_{d-k}(\rho(\gamma_n^{-1})) = \xi^{d-k}(z'^-)~.\]
    
    Since $K$ is compact, there exists a constant $C_K > 0$ such that $\norm{\cdot}$ is $C_K$-bi-Lipschitz to $\norm{\cdot}_0$ on $K$, where $\norm{\cdot}_0$ is the standard metric on $\Kb^d$. Then we have 
    \[ \dfrac{\norm{v_n}_{\phi^{t_n}(z_n)}}{\norm{w_n}_{\phi^{t_n}(z_n)}} = \dfrac{\norm{\rho(\gamma_n^{-1})v_n}_{\gamma_n^{-1}\phi^{t_n}(z_n)}}{\norm{\rho(\gamma_n^{-1})w_n}_{\gamma_n^{-1}\phi^{t_n}(z_n)}} \sim_{(C_K^2,0)} \dfrac{\norm{\rho(\gamma_n^{-1})v_n}_0}{\norm{\rho(\gamma_n^{-1})w_n}_0}~.\]
    We now show that \[\limsup\limits_{n\to +\infty}\dfrac{\norm{\rho(\gamma_n^{-1})v_n}_0}{\norm{\rho(\gamma_n^{-1})w_n}_0} \dfrac{\sigma_{d-k}(\rho(\gamma_n^{-1}))}{\sigma_{d-k+1}(\rho(\gamma_n^{-1}))}< +\infty~.\]
    Since $w_n\in E^u_{z_n} = \xi^{d-k}(z_n^-)$, $\rho(\gamma_n^{-1})w_n \in \xi^{d-k}(\gamma_n^{-1} z_n^-)$. Then by $U_{d-k}(\rho(\gamma_n^{-1})) \to \xi^{d-k}(z'^-)$ and $\xi^{d-k}(\gamma_n^{-1} z_n^-) \to \xi^{d-k}(z'^-)$ as $n\to +\infty$, we have \[\angle(\rho(\gamma_n^{-1})w_n, U_{d-k}(\rho(\gamma_n^{-1}))) \to 0\] as $n\to +\infty$. Then \[\liminf\limits_{n\to +\infty} \dfrac{\norm{\rho(\gamma_n^{-1})w_n}_0}{\sigma_{d-k}(\rho(\gamma_n^{-1}))} \geqslant \liminf\limits_{n\to +\infty}\norm{w_n}_0 \sim_{(C_K,0)}\norm{w_n}_{z_n} = 1\]
    Let $v_n = v_n^1 + v_n^2$ such that $v_n^1 \in U_k(\rho(\gamma_n))$ and $v_n^2$ is orthogonal to $U_k(\rho(\gamma_n))$. Then $\rho(\gamma_n^{-1})v_n^1$ and $\rho(\gamma_n^{-1})v_n^2$ are orthogonal and hence \[\norm{\rho(\gamma_n^{-1})v_n}^2_0 = \norm{\rho(\gamma_n^{-1})v_n^1}^2_0 + \norm{\rho(\gamma_n^{-1})v_n^2}^2_0~.\]
    For the first part, 
    \begin{align*}
    \norm{\rho(\gamma_n^{-1})v_n^1}_0 & \leqslant \sigma_{d-k+1}(\rho(\gamma_n^{-1}))\norm{v_n^1}_0\\
    & \leqslant \sigma_{d-k+1}(\rho(\gamma_n^{-1})) \norm{v_n}_0\\ 
    & \sim_{(C_K,0)} \sigma_{d-k+1}(\rho(\gamma_n^{-1})) \norm{v_n}_{z_n}\\
    & = \sigma_{d-k+1}(\rho(\gamma_n^{-1}))~.
    \end{align*}
    Then $\limsup\limits_{n\to +\infty}\dfrac{\norm{\rho(\gamma_n^{-1})v_n^1}_0}{\norm{\rho(\gamma_n^{-1})w_n}_0} \dfrac{\sigma_{d-k}(\rho(\gamma_n^{-1}))}{\sigma_{d-k+1}(\rho(\gamma_n^{-1}))}$ is bounded since $(\rho(\gamma_n) )_{n\in \Nb}$ is a $P_k$-divergent sequence.
    For the second part, suppose that, up to a subsequence,
    \[\lim\limits_{n\to +\infty} \dfrac{\norm{\rho(\gamma_n^{-1})v_n^2}_0}{\norm{\rho(\gamma_n^{-1})v_n^1}_0} = +\infty~.\]
    This implies that \[\limsup\limits_{n\to +\infty}\angle(\rho(\gamma_n^{-1})v_n^2, \rho(\gamma_n^{-1})v_n) = 0~.\]
    However, $\rho(\gamma_n^{-1})v_n^2$ is contained in $\rho(\gamma_n^{-1}) U_k(\rho(\gamma_n))^{\perp} = U_{d-k}(\rho(\gamma_n^{-1}))$ which has limit $\xi^{d-k}(z'^-)$ as $n\to +\infty$, while $\rho(\gamma_n^{-1})v_n$ is contained in $\rho(\gamma_n^{-1})\xi^k(z_n^+) = \xi^k(\gamma_n^{-1}z_n^+)$, which has limit $\xi^k(z'^+)$ as $n\to +\infty$. We then get a contradiction since $\xi^{d-k}(z'^-)$ and $\xi^k(z'^+)$ are transverse.

    Then there is a constant $C > 0$ such that \[\norm{\rho(\gamma_n^{-1})v_n^2}_0 \leqslant C \norm{\rho(\gamma_n^{-1})v_n^1}_0 \] and hence \[\norm{\rho(\gamma_n^{-1})v_n}_0 \sim_{(\sqrt{1+C^2},0)} \norm{\rho(\gamma_n^{-1})v_n^1}_0 \] The constant $C$ is uniform for any choice of the sequences of $z_n,t_n,\gamma_n,v_n,w_n$, as otherwise we would find another choice of these sequences leading to the same contradiction.

    Therefore we have the estimation
    \begin{align*}
    \dfrac{\norm{v_n}_{\phi^{t_n}(z_n)}}{\norm{w_n}_{\phi^{t_n}(z_n)}}  & \sim_{(C_K^2\sqrt{1+C^2},0)} \dfrac{\norm{\rho(\gamma_n^{-1})v_n^1}_0}{\norm{\rho(\gamma_n^{-1})w_n}_0}\\ 
    & \leqslant \dfrac{\sigma_{d-k+1}(\rho(\gamma_n^{-1}))}{\sigma_{d-k}(\rho(\gamma_n^{-1}))}\dfrac{\norm{v_n}_{z_n}}{\norm{w_n}_{z_n}}\\
    & \sim_{(C_K^2,0)} \dfrac{\sigma_{d-k+1}(\rho(\gamma_n^{-1}))}{\sigma_{d-k}(\rho(\gamma_n^{-1}))}~.
    \end{align*}
\end{proof}

\begin{remark}
    The idea of the proof follows from \cite[Proposition 6.5]{CZZ} (see also \cite[Theorem 5.15]{W23} and \cite[Lemma 9.4]{ZZ1}), but the way we deal with the boundary points is slightly different. Although the $\Gamma$-action on $\Lambda$ is no longer a convergence group action, the argument still works since $\Lambda$ is a subset of $\Fc_{k,d-k}$ which already contains the information of the ``$P_k$-limits of $\rho$''.
\end{remark}

\begin{lemma}\label{ThinBoundaryPointsEstimation}
    There exist constants $C,c > 0$ such that for any $z\in\partial^+\Fc_P$ with $T_z^+ < +\infty$, \[ \dfrac{\norm{v}_{\phi^{T_z^+}(z)}}{\norm{w}_{\phi^{T_z^+}(z)}} \leqslant C e^{-c T_z^+}\dfrac{\norm{v}_{z}}{\norm{w}_{z}} ~,\] where $v\in E^s_z$ and $w\in E^u_z$ are nonzero vectors.
\end{lemma}
\begin{proof}
    Since $P$ acts on $\partial^+\Fc_P$ and $\partial^-\Fc_P$ cocompactly for each $P\in \Pc^\Gamma$ and $\Pc$ is finite,
    up to replacing $K$ by a larger compact set, we may assume that $P\cdot (K\cap \partial^+\Fc_P) = \partial^+\Fc_P$ and $P\cdot (K\cap \partial^-\Fc_P) = \partial^-\Fc_P$ for each $P\in\Pc$.
    Without loss of generality, it suffices to show the lemma for $z\in K \cap \partial^+\Fc_P$ with $P\in \Pc$, then we can find $\gamma\in P$ such that $\gamma^{-1}\phi^{T_z^+}(z)\in K$. There exist uniform constants $\lambda\geqslant 1$ and $\epsilon\geqslant 0$ such that $\abs{\gamma}_X\sim_{(\lambda,\epsilon)} T_z^+$ since $p_M$ is a quasi-isometry. Recall that the Gromov model $X$ is given such that $\rho|_P$ is $P_k$-weakly dominated, then there exist uniform constants $C_1, c_1 >0$, such that \[\dfrac{\sigma_{k+1}(\rho(\gamma))}{\sigma_{k}(\rho(\gamma))} \leqslant C_1 e^{-c_1|\gamma|_X}~.\] By Lemma \ref{ThickPartPointsEstimation}, 
    \begin{align*}
    \dfrac{\norm{v}_{\phi^{T_z^+}(z)}}{\norm{w}_{\phi^{T_z^+}(z)}} &\leqslant C_0 \dfrac{\sigma_{k+1}(\rho(\gamma))}{\sigma_{k}(\rho(\gamma))}\dfrac{\norm{v}_{z}}{\norm{w}_{z}} \\ &\leqslant C_0 C_1 e^{-c_1|\gamma|_X}\dfrac{\norm{v}_{z}}{\norm{w}_{z}}\\ & \leqslant C_0 C_1 e^ {c_1\epsilon} e^{-c_1 \lambda^{-1}T_z^+}\dfrac{\norm{v}_{z}}{\norm{w}_{z}}~,
    \end{align*}
    which complete the proof.
\end{proof}

We denote $E^1_z = E^s_z = \xi^k(z^+)$, $E^2_z = E^u_z \cap E^u_{\hat{z}} = \xi^{d-k}(z^+)\cap \xi^{d-k}(z^-)$ and $E^3_z = E^s_{\hat{z}} = \xi^k(z^-)$, and denote the metric on $E^i_z$ by $\norm{\cdot}_{i,z}$ for $i=1,2,3$, which are defined on $\Fc^{th}\times \Kb^d$ and $\norm{\cdot} = \norm{\cdot}_{1,z} + \norm{\cdot}_{2,z} + \norm{\cdot}_{3,z}$. 

\begin{lemma}\cite[Proposition 3.14]{ZZ1}
    Let $\langle \cdot,\cdot \rangle$ and $\langle \cdot,\cdot \rangle'$ be two inner products on $\Kb^d$. Then there is a basis $\{v_1,v_2,...,v_d\}$ of $\Kb^d$ that is orthogonal with respect to both $\langle \cdot,\cdot \rangle$ and $\langle \cdot,\cdot \rangle'$. Moreover, \[m(t) (v_i,v_j)= m(t,\langle \cdot,\cdot \rangle,\langle \cdot,\cdot \rangle')(v_i,v_j) = (\langle v_i,v_j \rangle)^{1-t}(\langle v_i,v_j \rangle')^{t}\] for any $i,j \in \{1,2,...,d\}$ defines a path in the space of inner products on $\Kb^d$ for $t\in [0,1]$ with $m(0)(\cdot,\cdot) = \langle \cdot,\cdot \rangle$ and $m(1)(\cdot,\cdot) = \langle \cdot,\cdot \rangle'$.
\end{lemma}
\begin{remark}
    The path $m(t,\langle \cdot,\cdot \rangle,\langle \cdot,\cdot \rangle')$ is independent of the choice of the common orthogonal basis $\{v_1,v_2,...,v_d\}$.
\end{remark}

Now we extend $\norm{\cdot}$ on $\Fc_P$ for each $P\in \Pc$ in the following way. Let $C, c > 0$ be the same constants from Lemma \ref{ThinBoundaryPointsEstimation}.
If $z\in \partial^+\Fc_P$, and $t\in [0, T^+_z/3]$, define \[\norm{\cdot}_{\phi^t(z)} = e^{-ct}\norm{\cdot}_{1,z} + \norm{\cdot}_{2,z} + e^{ct}\norm{\cdot}_{3,z}~.\] Since we hope to define a reverse invariant metric, we need to set that for $t\in [2T^+_z/3, T^+_z]$, \[\norm{\cdot}_{\phi^t(z)} = e^{c(T^+_z-t)}\norm{\cdot}_{1,\phi^{T^+_z}(z)} + \norm{\cdot}_{2,\phi^{T^+_z}(z)} + e^{-c(T^+_z-t)}\norm{\cdot}_{3,\phi^{T^+_z}(z)}~.\] For $t\in [T^+_z/3,2T^+_z/3]$, let $\norm{\cdot}_{\phi^{t}(z)}$ be the metric given by \[m(\dfrac{3t-T^+_z}{T^+_z},\norm{\cdot}_{\phi^{T^+_z/3}(z)},\norm{\cdot}_{\phi^{2T^+_z/3}(z)})~.\]
Finally, $\norm{\cdot}$ extends to a continuous, $\Gamma$-invariant, reverse invariant metric on $\Fc(\Lambda,\zeta,X)\times \Kb^d$.

By the construction of the metric on $\Fc_P$ and Lemma \ref{ThinBoundaryPointsEstimation}, we can compute and deduce the following lemma by direct computation.

\begin{lemma}\label{SingleThinEstimation}
    For any $z\in \partial^+\Fc_P$ and $0\leqslant s \leqslant s+t \leqslant T_z^+ $, \[ \dfrac{\norm{v}_{\phi^{s+t}(z)}}{\norm{w}_{\phi^{s+t}(z)}} \leqslant C e^{-c t}  \dfrac{\norm{v}_{\phi^{s}(z)}}{\norm{w}_{\phi^{s}(z)}}~,\] where $v\in E^s_{\phi^{s}(z)}$ and $w\in E^u_{\phi^{s}(z)}$ are nonzero vectors.
\end{lemma}

Now we show that $\Fc(\Lambda,\zeta,X)\times \Kb^d = E^s \oplus E^u$ is a dominated splitting of rank $k$ with respect to $\norm{\cdot}$ 

By Lemma \ref{ThickPartPointsEstimation}, $\rho$ being $P_k$-divergent, and the $\Gamma$-action on $\Fc^{th}$ being cocompact, there exists a constant $T_1 >0$ such that for any $z\in \Fc^{th}$ and $t\geqslant T_1$ such that $\phi^t(z)\in\Fc^{th}\in \Fc^{th}$, one has \[\dfrac{\norm{v}_{\phi^t(z)}}{\norm{w}_{\phi^t(z)}} \leqslant \dfrac{1}{2C^2}\dfrac{\norm{v}_z}{\norm{w}_z}~,\] for any nonzero vectors $v\in E^s_{z}$ and $w\in E^u_{z}$.

For any $T\in \Rb_{\geqslant 0}$, let \[\Fc^{th}_T = \{\phi^t(z)\in \Fc(\Lambda,\zeta,X): z\in\Fc^{th}, t\in [-T,T]\}~.\] Notice that $\Gamma$ also acts on $\Fc^{th}_T$ cocompactly. Then there exists a constant $C(T) > 0$ depends on $T$, such that \[\dfrac{\norm{v}_{\phi^t(z)}}{\norm{w}_{\phi^t(z)}} \leqslant C(T)\dfrac{\norm{v}_z}{\norm{w}_z} \] for any $z\in \Fc^{th}$, $0\leqslant t\leqslant T$, $v\in E^s_{z}$ and $w\in E^u_{z}$ nonzero vectors.

Let $T_2\geqslant T_1$ be such that $C(T_1)C^2 e^{cT_1} e^{-c t}\leqslant \dfrac{1}{2}$ for any $t\geqslant T_2$.

\begin{claim}[1]
    For any $z\in \Fc(\Lambda,\zeta,X)$, \[\dfrac{\norm{v}_{\phi^t(z)}}{\norm{w}_{\phi^t(z)}} \leqslant \dfrac{1}{2}\dfrac{\norm{v}_z}{\norm{w}_z}\] for any $t\geqslant T_2$, $v\in E^s_{z}$ and $w\in E^u_{z}$ nonzero vectors.
\end{claim}

\begin{proof}[Proof of Claim (1)]
    If $\phi^s(z)\not\in \Fc^{th}$ for any $s\in [0,t]$, then there exists $P\in\Pc^\Gamma$, such that $\phi^s(z)\in\Fc_P$ for all $s\in [0,t]$, then by Lemma \ref{SingleThinEstimation}, \[\dfrac{\norm{v}_{\phi^t(z)}}{\norm{w}_{\phi^t(z)}} \leqslant C e^{-c t}\dfrac{\norm{v}_z}{\norm{w}_z} \leqslant \dfrac{1}{2}\dfrac{\norm{v}_z}{\norm{w}_z}\] for any $v\in E^s_{z}$ and $w\in E^u_{z}$ nonzero vectors.

    If there exists $s\in [0,t]$ such that $\phi^s(z)\in \Fc^{th}$, let $s'$ (respectively, $s''$) be the minimum (respectively, maximum) of such $s\in [0,t]$. Then $[0,t]$ is separated as $3$ parts, $[0,s']$, $(s',s'')$ and $(s'',t]$. The first and third parts of the flow line are fully contained in $\Fc_P$ for some $P\in\Pc^\Gamma$ respectively, hence can be estimated by Lemma \ref{SingleThinEstimation}.
    \begin{itemize}
        \item[Case 1.] When $s''-s'\leqslant T_1$, 
        \begin{align*} \dfrac{\norm{v}_{\phi^t(z)}}{\norm{w}_{\phi^t(z)}}
        & \leqslant Ce^{-c(t-s'')} \dfrac{\norm{v}_{\phi^{s''}(z)}}{\norm{w}_{\phi^{s''}(z)}}\\
        & \leqslant Ce^{-c(t-s'')} C(T_1) \dfrac{\norm{v}_{\phi^{s'}(z)}}{\norm{w}_{\phi^{s'}(z)}}\\
        & \leqslant Ce^{-c(t-s'')} C(T_1) Ce^{-cs'}\dfrac{\norm{v}_z}{\norm{w}_z} \\
        & \leqslant C(T_1)C^2 e^{cT_1} e^{-c t} \dfrac{\norm{v}_z}{\norm{w}_z} \leqslant \dfrac{1}{2} \dfrac{\norm{v}_z}{\norm{w}_z}
        \end{align*}
        for any $z\in \Fc^{th}$, $0\leqslant t\leqslant T$, $v\in E^s_{z}$ and $w\in E^u_{z}$ nonzero vectors.
        \item[Case 2.] When $s''-s'\geqslant T_1$,
        \begin{align*} 
        \dfrac{\norm{v}_{\phi^t(z)}}{\norm{w}_{\phi^t(z)}} 
        & \leqslant Ce^{-c(t-s'')} \dfrac{\norm{v}_{\phi^{s''}(z)}}{\norm{w}_{\phi^{s''}(z)}} \\
        & \leqslant Ce^{-c(t-s'')} \dfrac{1}{2C^2} \dfrac{\norm{v}_{\phi^{s'}(z)}}{\norm{w}_{\phi^{s'}(z)}}\\
        & \leqslant Ce^{-c(t-s'')} \dfrac{1}{2C^2} Ce^{-cs'} \dfrac{\norm{v}_z}{\norm{w}_z}\\
        & \leqslant \dfrac{1}{2} e^{-c (t-((s''-s')))} \dfrac{\norm{v}_z}{\norm{w}_z} \leqslant \dfrac{1}{2} \dfrac{\norm{v}_z}{\norm{w}_z}
        \end{align*}
        for any $z\in F^{th}$, $0\leqslant t\leqslant T$, $v\in E^s_{z}$ and $w\in E^u_{z}$ nonzero vectors.
    \end{itemize}
\end{proof}

\begin{claim}[2]
    There exists constant $C'>0$, such that for any $z\in \Fc(\Lambda,\zeta,X)$, \[\dfrac{\norm{v}_{\phi^t(z)}}{\norm{w}_{\phi^t(z)}} \leqslant C'\dfrac{\norm{v}_z}{\norm{w}_z}\] for any $0\leqslant t\leqslant T_2$, $v\in E^s_{z}$ and $w\in E^u_{z}$ nonzero vectors.
\end{claim}

\begin{proof}[Proof of Claim (2)]
    For any $z\in \Fc(\Lambda,\zeta,X)$, $\{\phi^t(z):t\in [0,T_2]\}$ is contained in $\Fc^{th}_{T_2}$ or a $\Fc_P$ for some $P\in\Pc^\Gamma$. For the first case, \[\dfrac{\norm{v}_{\phi^t(z)}}{\norm{w}_{\phi^t(z)}} \leqslant C(T_2)\dfrac{\norm{v}_z}{\norm{w}_z}\] and for the second case, by Lemma \ref{SingleThinEstimation} \[\dfrac{\norm{v}_{\phi^t(z)}}{\norm{w}_{\phi^t(z)}} \leqslant Ce^{-cT_2}\dfrac{\norm{v}_z}{\norm{w}_z}\leqslant C\dfrac{\norm{v}_z}{\norm{w}_z}\] for any $0\leqslant t\leqslant T_2$, $v\in E^s_{z}$ and $w\in E^u_{z}$ nonzero vectors. Let $C' = \max\{C, C(T_2)\}$.
\end{proof}

Eventually, for any $t\in\Rb_{\geqslant 0}$, we write $t = n_t T_2 + r_t$ where $n_t\in\Nb$ and $0\leqslant r_t < T_2$, then for any $z\in \Fc(\Lambda,\zeta,X)$, \[\dfrac{\norm{v}_{\phi^t(z)}}{\norm{w}_{\phi^t(z)}} \leqslant \Big(\dfrac{1}{2}\Big)^{n_t}\dfrac{\norm{v}_{\phi^{r_t}(z)}}{\norm{w}_{\phi^{r_t}(z)}} \leqslant C(T_2) \Big(\dfrac{1}{2}\Big)^{(t- T_2)/T_2}\dfrac{\norm{v}_z}{\norm{w}_z} \] for any $v\in E^s_{z}$ and $w\in E^u_{z}$ nonzero vectors, which implies that $\Fc(\Lambda,\zeta,X)\times \Kb^d = E^s \oplus E^u$ with the metric $\norm{\cdot}$ is a dominated splitting of rank $k$. Therefore we complete the proof of Theorem \ref{TheoremEGFtoRestricted}.

\subsection{A remark on relatively Anosov representations}\label{SectionRemarkonRAR}
We show Corollary \ref{LimitMapWaiveCor} in this section.
Let $(\Gamma, \Pc)$ be a relatively hyperbolic group with an adapted generating set $S$, and let $\rho:\Gamma\to\SL(d,\Kb)$ be a representation that is $P_k$-Anosov relative to $\Pc$.

\begin{proposition}\cite[Proposition 4.2, Theorem 8.1]{ZZ1}\label{ZZ1P42andT81}
    If a representation $\rho:\Gamma\to\SL(d,\Kb)$ is $P_k$-Anosov relative to $\Pc$, then there exist constants $\alpha> 0$ and $\beta\geqslant 0$, such that \[\log \dfrac{\sigma_{k}(\rho(\gamma))}{\sigma_{k+1}(\rho(\gamma))}\geqslant \alpha \log \abs{\gamma}_S - \beta \] for any $P\in \Pc$ and $\gamma\in P$.
\end{proposition}

The proposition says that the singular value ratio of elements in parabolic subgroups is growing at least linearly with respect to the word length. Then, by Proposition \ref{RatioGrowthLinearLemma}, we have that for any Gromov model $X$ of $(\Gamma, \Pc)$, $\rho|_P$ is $P_k$-weakly dominated for each $P\in \Pc$. Therefore, when focus on relatively Anosov representations, we no longer need to apply Lemma \ref{GromovModelDivtoDom} to find proper Gromov models.

On the other hand, from Section \ref{SectionRelationAsymptoticallyEmbeddedandandRelativelyAnosov}, we know that a representation $\rho$ is $P_k$-Anosov relative to $\Pc$ with limit map $\xi:\partial(\Gamma,\Pc) \cong \Lc(\rho)\subset \Fc_{k,d-k}$ if and only if $\rho$ is extended geometrically finite with a homeomorphic boundary extension $\xi^{-1}:\Lc(\rho) \to \partial(\Gamma,\Pc)$. Then in this case, $\Fc(\Lc(\rho),\xi^{-1},X) = \Fc(X)$ for any Gromov model $X$ of $(\Gamma, \Pc)$.

We apply these facts for relatively Anosov representations and deduce from Theorem \ref{TheoremRestrictedtoEGF} and Theorem \ref{TheoremEGFtoRestricted} that

\begin{corollary}
    A representation $\rho:\Gamma\to\SL(d,\Kb)$ is $P_k$-Anosov relative to $\Pc$ if and only if it is $P_k$-Anosov in restriction to $\Fc(X)$ for a (any) Gromov model $X$ of $(\Gamma, \Pc)$.
\end{corollary}

In particular, analogous to Corollary \ref{WeaklyDominatedCorollary}, we have

\begin{corollary}\label{WeaklyDominatedforAnyGromovModel}
    If a representation $\rho:\Gamma\to\SL(d,\Kb)$ is $P_k$-Anosov relative to $\Pc$, then it is $P_k$-weakly dominated with respect to any Gromov model $X$ of $(\Gamma,\Pc)$.
\end{corollary}

\subsection{Stability}
We keep the assumptions and notations as in Section \ref{Section1to2theFirst}. We now prove Theorem \ref{MainStabilityTheorem} by showing the type preserving stability of any representation that is $P_k$-Anosov in restriction to $\Fc(\Lambda,\zeta,X)$. Recall that \[\Fc(\Lambda,\zeta,X) = \Fc^{th} \cup (\bigcup_{P\in\Pc^\Gamma} \Fc_P)~,\] is a thick-thin decomposition of $\Fc(\Lambda,\zeta,X)$.

\begin{theorem}\label{StabilityTheorem}
    Let $\rho_0:\Gamma\to\SL(d,\Kb)$ be a representation that is $P_k$-Anosov in restriction to $\Fc(\Lambda,\zeta,X)$, then there exists a neighborhood $O\subset\Hom_{\Pc}(\rho_0)$ of $\rho_0$, such that any $\rho\in O$ is $P_k$-Anosov in restriction to $\Fc(\Lambda,\zeta,X)$.
\end{theorem}

We now assume that $\rho_0:\Gamma\to\SL(d,\Kb)$ is a representation that is $P_k$-Anosov in restriction to $\Fc(\Lambda,\zeta,X)$ and $O$ is a neighborhood of $\rho_0$ in $\Hom_{\Pc}(\rho_0)$. For each $P\in\Pc$, let \[A_P:O\to \SL(d,\Kb)\] be a continuous map such that \[\rho(\gamma) = A_P(\rho)\rho(\gamma)A_P(\rho)^{-1}\] for any $\gamma\in P$ and $A(\rho_0) = \Id$.

The first step is to show that under small type preserving deformations, the flat bundle over $\Gamma\backslash\Fc(\Lambda,\zeta,X)$ is deformed isomorphically. More concretely, consider the following two $\Gamma$-actions on $O\times \Fc(\Lambda,\zeta,X)\times \Kb^d$. We denote by $\iota_0$ the $\Gamma$-action on $O\times \Fc(\Lambda,\zeta,X)\times \Kb^d$ by $\iota_0(\gamma)(\rho,z,v)=(\rho,\gamma z,\rho_0(\gamma)v)$ and we denote by $\iota$ the one by $\iota(\gamma)(\rho,z)=(\rho,z,\rho(\gamma)v)$. The following lemma indicate the existence of a $\iota_0(\Gamma)$-$\iota(\Gamma)$-equivariant isomorphism \[\hat{g}: O\times \Fc(\Lambda,\zeta,X)\times \Kb^d \to O\times \Fc(\Lambda,\zeta,X)\times \Kb^d~,\] fibering over identity, such that $\hat{g}(\rho_0, z,v) = (\rho_0, z,v)$ for any $z\in \Fc(\Lambda,\zeta,X)$ and $v\in \Kb^d$. Giving such an isomorphism $\hat{g}$ is equivalent to giving a map $g: O\times \Fc(\Lambda,\zeta,X) \to \SL(d,\Kb)$ as described in the lemma below.

\begin{lemma}\label{LocalDeformationIsomorphisms}
     Up to replacing $O$ by a smaller neighborhood of $\rho_0$, there exists a continuous map
     \[g: O\times \Fc(\Lambda,\zeta,X) \to \GL(d,\Kb), \] such that \[\rho(\gamma)g(\rho,z)=g(\rho,\gamma z)\rho_0(\gamma)\] for any $(\rho,z)\in \Fc(\Lambda,\zeta,X)$, and $g|_{\rho_0\times \Fc(\Lambda,\zeta,X) }\equiv \Id$.
\end{lemma}

\begin{proof}
    We firstly construct $g$ on the thick part. Let $K$ be a compact subset of $\Fc^{th}$ such that $\Gamma\cdot K=\Fc^{th}$ and let $f:\Fc(\Lambda,\zeta,X)\to \Rb_{\geqslant 0}$ be a function with compact support and $f|_K\equiv 1$. Then \[g^{th}(\rho,z) =  \dfrac{1}{\sum_{\gamma\in\Gamma} f(\gamma^{-1}z)} \sum_{\gamma\in\Gamma} f(\gamma^{-1}z)\rho(\gamma)\circ \rho_0(\gamma)^{-1}~,\] is well-defined and verifies the equivariant condition $\rho(\gamma) g^{th}(\rho, z) = g^{th}(\rho, \gamma z) \rho_0(z)$. Actually, the summation in the formula is locally finite since $f$ has compact support and $\Gamma$ acts on $\Fc(\Lambda,\zeta,X)$ properly discontinuously. Up to replacing $O$ by a smaller neighborhood, $g^{th}|_{O\times K}$ takes values in $\GL(d,\Kb)$ by the compactness of $K$, and hence $g^{th}|_{O\times \Fc^{th}}$ takes values in $\GL(d,\Kb)$.
    
    We then construct $g$ on the cusp regions. Let $P\in \Pc$ and let $h_P:\Fc(\Lambda,\zeta,X)\to \Rb_{\geqslant 0}$ be a $P$-invariant function that supports on a small neighborhood of $\Fc_P$ and $h|_{\Fc_P}\equiv 1$. Let \[g_P (\rho, z) = h_P(z) A_P(\rho)~.\]
    For each $P'\in\Pc^{\Gamma}$ with $P'=\eta P\eta^{-1}$ for some $\eta\in \Gamma$ and $P\in \Pc$, let \[g_{P'}(\rho,z) =\rho(\eta)g_P(\rho,\eta^{-1} z) \rho_0(\eta)^{-1}~,\] which is actually independent of the choice of $\eta$ since $g_P$ is $P$-invariant.

    Finally, let \[g(\rho,z) =   \dfrac{1}{\sum_{\gamma\in\Gamma} f(\gamma^{-1}z) + \sum_{P\in\Pc^\Gamma} h_P(z) } \Big( g^{th}(\rho,z) + \sum_{P\in\Pc^\Gamma} g_P(\rho,z) \Big)~,\]
    which is well-defined with the desired properties. Actually, it is clear that $g|_{\rho_0\times \Fc(\Lambda,\zeta,X) }\equiv \Id$. Up to shrinking $O$ again, the restriction of $g$ on the support of $g^{th}$ takes values in $\GL(d,\Kb)$ by compactness. $g$ also takes values in $\GL(d,\Kb)$ outside the support of $g^{th}$ by the construction of $g_P$.
\end{proof}

The rest of the proof of Theorem \ref{StabilityTheorem} directly follows the argument of \cite[Theorem 8.1]{CZZ} in their study of cusped Hitchin representations, which is originally from \cite[Corollary 5.19]{Shub} for the stability of dominated splittings. One may also refer to \cite[Theorem 12.1]{ZZ1} for the stability of relatively Anosov representations. We sketch the argument here.

Since $\rho_0$ is $P_k$-Anosov in restriction to $\Fc(\Lambda,\zeta,X)$, let \[\Fc(\Lambda,\zeta,X)\times \Kb^d = E^s_{\rho_0} \oplus E^u_{\rho_0}\] be a dominated splitting of rank $k$ with respect to a $\rho_0(\Gamma)$-invariant metric $\norm{\cdot}$. Then \[O\times \Fc(\Lambda,\zeta,X)\times \Kb^d = (O\times E^s_{\rho_0}) \oplus (O\times E^u_{\rho_0})\] is a dominated splitting of rank $k$ over the $\Gamma$-flow $O\times \Fc(\Lambda,\zeta,X)$ with the $\Gamma$-action through $\iota_0$, with respect to the trivial extension of $\norm{\cdot}$ on the trivial bundle, still denoted by $\norm{\cdot}$.

Since there is a $\iota_0(\Gamma)$-$\iota(\Gamma)$-equivariant isomorphism $\hat{g}: O\times \Fc(\Lambda,\zeta,X)\times \Kb^d \to O\times \Fc(\Lambda,\zeta,X)\times \Kb^d$ due to the lemma above, we write $E_O^s=\hat{g}(O\times E^s_{\rho_0})$ and $E_O^u=\hat{g}(O\times E^u_{\rho_0})$, then \[O\times \Fc(\Lambda,\zeta,X)\times \Kb^d = E_O^s \oplus E_O^u\] is a decomposition of $O\times \Fc(\Lambda,\zeta,X)\times \Kb^d$ into $\iota(\Gamma)$-invariant subbundles, and $\hat{g}_*\norm{\cdot}$ is a $\iota(\Gamma)$-invariant metric.

This decomposition is $\phi$-invariant over $\rho_0 \times \Fc(\Lambda,\zeta,X)$ and we wish to deform it to be $\phi$-invariant over $O \times \Fc(\Lambda,\zeta,X)$. Small deformations of $E_O^u$ can be characterized by images of bounded sections of $\Hom({E}_O^u, {E}_O^s)$. One can show that the induced flow on $\Hom({E}_O^u, {E}_O^s)$ is a contracting map on the unit ball of $\Hom({E}_O^u, {E}_O^s)$ (with respect to the norm defined by the supremum of the operator norm on each fiber). Hence we can apply the contraction mapping theorem and deduce that there is a fixed section $s$ of $\Hom({E}_O^u, {E}_O^s)$ which provides a flow invariant deformation $F_O^u = s(E_O^u)$. By a similar process, we find a deformation of $E_O^s$, denoted by $F_O^s$, which is also flow invariant. Up to shrinking $O$ again, $O\times \Fc(\Lambda,\zeta,X)\times \Kb^d = F_O^u\oplus F_O^s$ is a dominated splitting with respect to the metric $\hat{g}_*\norm{\cdot}$. When we restrict this dominated splitting to each piece $\rho\times \Fc(\Lambda,\zeta,X)$, we see that $\rho$ is $P_k$-Anosov in restriction to $\Fc(\Lambda,\zeta,X)$.

\section{Relative hyperbolicity extensions}\label{SectionRHExtensions}
\subsection{Hyperbolic relative to relatively hyperbolic subgroups}

For a finitely generated group $\Gamma$, a collection $\Pc$ of subgroups of $\Gamma$ is said to be \emph{almost malnormal} if for any $P,P'\in \Pc$ and $\gamma\in\Gamma$, $\gamma P \gamma^{-1} \cap P$ is infinite only when $P=P'$ and $\gamma\in P$.

When $\Gamma$ is a hyperbolic group, a condition for a collection $\Pc$ of subgroups of $\Gamma$ providing $(\Gamma,\Pc)$ a relatively hyperbolic group is characterized by the following theorem.  Recall that we denote $\Pc^\Gamma = \{\gamma P\gamma^{-1} : \gamma\in \Gamma,P\in \Pc\}$.

\begin{theorem}[\cite{Bow12},\cite{Tran13}, \cite{Man15}]\label{ClassicalExtensionReltoHyp}
    Let $\Gamma$ be a hyperbolic group and $\Pc$ is a finite collection of finitely generated, infinite subgroups of $\Gamma$. Then $(\Gamma,\Pc)$ is relatively hyperbolic if and only if $\Pc$ is almost malnormal and each $P\in \Pc$ is quasi-convex in $\Gamma$. Moreover, when the conditions above hold, there is a continuous, $\Gamma$-equivariant quotient map $\tau:\partial\Gamma\to \partial(\Gamma,\Pc)$ which exactly identifies the limit set of $P$ in $\partial\Gamma$ to the parabolic point fixed by $P$ in $\partial(\Gamma,\Pc)$ for each $P\in\Pc^\Gamma$.
\end{theorem}\ \\

We now study a generalization of Theorem \ref{ClassicalExtensionReltoHyp} in the relatively hyperbolic situation. From now on, we assume that $(\Gamma,\Hc)$ is a relatively hyperbolic group. Let $\Pc$ be a finite collection of finitely generated, infinite subgroups of $\Gamma$, such that for any $H\in\Hc$, there is some $P\in\Pc$ with $H\subset P$. For each $P\in\Pc$, we write $\Hc_P^P = \{H\in\Hc^\Gamma : H\subset P\}$ and $\Hc_P$ denotes a collection of representatives of $P$-conjugation orbits in $\Hc_P^P$. We discuss about the conditions for $(\Gamma,\Pc)$ and $(P,\Hc_P)$ for each $P\in\Pc$ to be relatively hyperbolic, and moreover, the relations among the Bowditch boundaries $\partial(\Gamma,\Hc)$, $\partial(\Gamma,\Pc)$ and $\partial(P,\Hc_P)$ for each $P\in \Pc$ when they are relatively hyperbolic.

It is shown in \cite[Proposition 7.1]{Hru10} that the following two definitions of relative quasi-convexity are equivalent. One should notice that the definitions depend on the relatively hyperbolic structure of $(\Gamma,\Hc)$.

\begin{definition}\cite[Definition 6.2]{Hru10}
    A subgroup $P$ of $\Gamma$ is relatively quasi-convex if $P$ acts on $\Lambda(P)$ geometrically finitely, where $\Lambda(P)$ denotes the limit set of $P$ in $\partial(\Gamma,\Hc)$ with respect to the convergence group action of $\Gamma$ on $\partial(\Gamma,\Hc)$.
\end{definition}

\begin{definition}\cite[Definition 6.5]{Hru10}
     A subgroup $P$ of $\Gamma$ is relatively quasi-convex if $P$ is either finite, or for some (hence any) Gromov model $X$ of $(\Gamma,\Hc)$, $P$ admits a cusp uniform action on a geodesic space that is $P$-equivariantly quasi-isometric to $\Hull(\Lambda(P))$, the convex hull of $\Lambda(P)$ in $X$.
\end{definition}

By the following two propositions, we deduce that when $(\Gamma,\Pc)$ is relatively hyperbolic, then $\Pc$ is almost malnormal and each $P\in\Pc$ is relatively quasi-convex.

\begin{proposition}\cite[Proposition 2.36, Lemma 5.4]{Osin06}
    If $(\Gamma,\Pc)$ is relatively hyperbolic, then $\Pc$ is an almost malnormal collection of finitely many proper, undistorted subgroups.
\end{proposition}

\begin{proposition}\cite[Theorem 1.5]{Hru10}\label{UndistortedtoRelativelyQuasiConvex}
    If $P$ is a finitely generated, undistorted subgroup of $\Gamma$, then $P$ is relatively quasi-convex.
\end{proposition}

Conversely, we will show that,

\begin{theorem}\label{RelativeBoundaryExtension}
    If $\Pc$ is an almost malnormal collection of relatively quasi-convex subgroups, then $(\Gamma,\Pc)$ and $(P,\Hc_P)$ for each $P\in\Pc$, are relatively hyperbolic. Moreover, $\partial(P,\Hc_P) \cong \Lambda(P)\subset \partial(\Gamma,\Hc)$ and there exists a continuous, $\Gamma$-equivariant quotient map $\tau:\partial(\Gamma,\Hc)\to \partial(\Gamma,\Pc)$ by identifying $\gamma\partial(P,\Hc_P)$ to the unique parabolic point in $\partial(\Gamma,\Pc)$ fixed by $\gamma P\gamma^{-1}$ for each $P\in\Pc$ and $\gamma\in \Gamma$.
\end{theorem}

Firstly, $(P,\Hc_P)$ is relatively hyperbolic for each $P\in\Pc$ is given by the following theorem.

\begin{theorem}\cite[Theorem 9.1]{Hru10}
    If $Q$ is a relatively quasi-convex subgroup of $(\Gamma,\Hc)$, then $(Q,\Hc_Q)$ is relatively hyperbolic, where $\Hc_Q$ is a set of representatives of $Q$-conjugacy classes of 
    \[\{Q\cap \gamma H \gamma^{-1} : \gamma\in\Gamma, H\in\Hc, \text{ and } Q\cap \gamma H \gamma^{-1} \text{ is infinite}.\}\]
\end{theorem}

For the rest of Theorem \ref{RelativeBoundaryExtension}, we generalize a proof of Theorem \ref{ClassicalExtensionReltoHyp} given by Manning \cite{Man15}.

\begin{proof}[Proof of Theorem \ref{RelativeBoundaryExtension}]
    We denote \[M_{par} = \{ \gamma\Lambda(P) : \gamma\in \Gamma, P\in \Pc\} = \{ \Lambda(P) : P\in \Pc^\Gamma\} \] and \[M_{con} = \partial(\Gamma,\Hc) \setminus \bigcup_{\gamma\in\Gamma, P \in \Pc} \gamma\Lambda(P)~.\] Then elements of $M_{con}$ as singletons and elements of $M_{par}$ together give a $\Gamma$-invariant partition of $\partial(\Gamma,\Hc)$. Let $M = M_{par} \cup M_{con}$. Let $\tau: \partial(\Gamma,\Hc)\to M$ be the quotient that maps each point to the set containing it in the partition. We endow $M$ with the quotient topology and the natural $\Gamma$-action induced by the $\Gamma$-action on $\partial(\Gamma,\Hc)$. Then $\tau$ is $\Gamma$-equivariant by definition. By Theorem \ref{YamanCriterion}, it suffices to show that $M$ is a perfect, metrizable, compact space, and the $\Gamma$-action on $M$ is geometrically finite with $M_{par}$ the set of bounded parabolic points and $M_{con}$ the set of conical limit points.\\ \\
    \textit{Step 1.} We show that $M$ is a perfect metrizable, compact space.
    Let $X$ be a Gromov model of $(\Gamma,\Hc)$ with $d_X$ the metric on $X$. Then $M_{par}$ is a null sequence, that is, for any $\epsilon > 0$, there are only finitely many $P\in \Pc^\Gamma$ with the diameter of $\Lambda(P)$ larger than $\epsilon$, with respect to a visual metric on $\partial(\Gamma,\Hc)$. In fact, any compact subset of $X$ only intersects finitely many subsets $\Hull(\Lambda(P))$, $P\in \Pc^\Gamma$ and there are only finitely many $\Gamma$-orbits of $\Hull(\Lambda(P))$ for $P\in\Pc^\Gamma$. Therefore $M_{par}$ is a null sequence by the properties of the visual metric. Then by \cite[Proposition 2.1 and Proposition 2.2]{Man15}, which are originally from \cite{Dav}, we have that $M$ is a metrizable, compact space.

    $M$ is perfect and $\Gamma$ acts on $M$ as a convergence group follows directly from the fact that $\partial(\Gamma,\Hc)$ is perfect, $\Gamma$ acts on $\partial(\Gamma,\Hc)$ as a convergence group and $\tau$ is a $\Gamma$-equivariant quotient map.\\ \\
    \textit{Step 2.} We find a Gromov model of $(P,\Hc_P)$ as a subspace of $X$ for each $P\in \Pc$ and discuss their structures.
    There exists a constant $C_1\geqslant 0$, such that $\Nc_{C_1}(\Hull(\Lambda(P))$ is proper and geodesic. Then $\Nc_{C_1}(\Hull(\Lambda(P))$ is a Gromov model for $(P,\Hc_P)$ for each $P\in\Pc$ (and hence for each $P\in\Pc^\Gamma$) by the relative quasi-convexity of $P$. We denote $X_P = \Nc_{C_1}(\Hull(\Lambda(P))$. Let $\{B_H : H\in \Hc^\Gamma\}$ be a collection of horoballs satisfying the conditions in Definition \ref{RelativelyHyperbolicGroupsDefinition} for the relatively hyperbolic group $(\Gamma,\Hc)$. Since each $H\in \Hc^\Gamma$ is contained in some $P\in\Pc^\Gamma$, horoballs are quasi-convex, and $H$ acts on the horosphere of $B_H$ cocompactly, we may assume $B_H\subset X_P$ when $H\subset P$, for any $H\in\Hc^\Gamma$ and $P\in\Pc^\Gamma$. In this case, $X_P$ together with the collection of horoballs $\{B_H : H\in\Hc^\Gamma, H\subset P\}$ satisfies the conditions in \ref{RelativelyHyperbolicGroupsDefinition} for the relatively hyperbolic group $(P,\Hc_P)$ and hence gives a Gromov model of $(P,\Hc_P)$. The construction works also for constants larger than $C_1$, which means $\Nc_R(X_P)$ is also a Gromov model of $(P,\Hc_P)$ for any $R\geqslant 0$. In particular, $P$ acts on $\Nc_R(X_P)\setminus X_P$ cocompactly, since the horoballs are contained in $X_P$, and $\Nc_R(X_P)\setminus X_P$ is contained in the thick part. Now we fix the thick-thin decomposition $X = X^{th} \cup \bigcup_{H\in\Hc^\Gamma} B_H$. \\ \\
    \textit{Step 3.} We show that $M_{par}$ consists of all bounded parabolic points by checking that for each $P\in\Pc^\Gamma$, $P$ acts on $M\setminus \{\Lambda(P)\}$ cocompactly. Since $\tau$ is a $\Gamma$-equivariant quotient, it suffices to show that $P$ acts on $\partial(\Gamma,\Hc)\setminus \Lambda(P)$ cocompactly. Since $\Nc_R(X_P)$ is relatively quasi-convex, there exists a constant $C_2\geqslant 0$, such that for any $y\in\partial(\Gamma,\Hc)\setminus \Lambda(P)$, there exists a point $\mathrm{proj}(y)\in \Nc_R(X_P)\setminus X_P$, the projection of $y$ in $\Nc_R(X_P)$ such that for any sequence $(y_n )_{n\in \Nb}\subset X$ with $y_n\to y$ as $n\to +\infty$, $d_X(y_n,\mathrm{proj}(y))\leqslant d(y_n, \Nc_R(X_P))+C_2$ (See for example \cite[Section 5]{Bow12}, \cite[Section 2.3]{Man15} for the definition of the projection map). The set of points in $\partial(\Gamma,\Hc)\setminus \Lambda(P)$, with projections in $\Nc_R(X_P)$ remaining in a compact set, is compact. Since the $P$-action on $\Nc_R(X_P)\setminus X_P$ is cocompact, then the $P$-action on $\partial(\Gamma,\Hc)\setminus \Lambda(P)$ is cocompact.

    Notice that for any $R\geqslant 0$, there exists constant $D\geqslant 0$, such that the diameter of $\Nc_R(X_P) \cap \Nc_R(X_{P'})$ is at most $D$ for any $P,P'\in\Pc^\Gamma$ (See for example \cite[Lemma 4.7]{DS05}, \cite[Lemma 2.6]{Man15}). We now fix a constant $R$ such that any geodesic with endpoints in $X_P\cup \Lambda(P)$ is contained in $\Nc_R(X_P)$ and we assume $D\geqslant R$.\\ \\
    \textit{Step 4.} Now we show that $M_{con}$ consists of all conical limit points. Let $x\in M_{con}$ and let $\ell $ be a geodesic in $X$ with $\ell (+\infty) = x$ and $\ell (-\infty) = y \ne x$. Then we can pick a sequence $(t_n )_{n\in \Nb}\subset\Rb$, such that 
    \begin{itemize}
        \item[1.] $t_n$ is increasing and $t_n\to +\infty$ as $n\to +\infty$;
        \item[2.] $\ell (t_n)\in \Nc_{2D}(X^{th})$ for each $n$;
        \item[3.] For each $P\in\Pc^\Gamma$, if $\ell (t_n)\in X_P$, then the geodesic ray $\ell ([t_n,+\infty))$ intersects $X_P$ for length at most $2D$.
    \end{itemize}
    More specifically, we firstly find a sequence $(t_n )_{n\in \Nb}$ satisfies (1) and $\ell (t_n')\in X^{th}$. For each $n$, if $t_n'$ satisfies condition (3), then we take $t_n = t_n'$. If there exists $P\in\Pc^\Gamma$ such that $t_n'\in X_P$ with $\ell ([t_n',+\infty))$ intersecting with $X_P$ of length larger than $2D$, then we take $t_n = \sup\{ s - D : \ell (s)\in X_P\}$. Then $t_n$ satisfies condition (2) and (3). Otherwise, there exists another $P' \in \Pc^\Gamma$ such that $\ell (t_n)\in X_{P'}$ and $\ell ([t_n,t_n+2D])\subset \Nc_R(X_{P'})$, which contradicts that $\Nc_R(X_P)\cap \Nc_R(X_{P'})$ has diameter at most $D$ for any $P' \ne P \in \Pc^\Gamma$. Since $t_n'\leqslant t_n$, there exists a subsequence of $(t_n )_{n\in \Nb}$ with the desired conditions.

    Let $(\gamma_n)_{n\in \Nb}\subset\Gamma$ be a sequence such that $d(\gamma_n^{-1} O, \ell (t_n))\leqslant C_3$, where $O$ is a fixed base point in $X^{th}$ and $C_3$ is a constant such that $\Nc_{2D}(X^{th}) \subset \Gamma\Nc_{C_3}(\{O\})$. Then one can easily check that there exist points $a\ne b\in \partial(\Gamma,\Hc)$, such that$\gamma_n x \to a$ and $\gamma_n y' \to b$ as $n\to +\infty$, for any $y'\in \partial(\Gamma,\Hc)\setminus\{x\}$. Now it suffices to show that $\tau(a)\ne \tau(b)$. Suppose otherwise, $\tau(a) = \tau(b)$, then there exists $P\in \Pc^\Gamma$, such that $a,b\in \Lambda(P)$. The sequence of geodesics $(\gamma_n \ell )_{n\in \Nb}$ converges to a geodesic $\ell'$ with $\ell'(+\infty)=a$ and $\ell'(-\infty)=b$, and there exists $s\in\Rb$ such that $d(O,\ell'(s))\leqslant C_3$. Then $\ell'(\Rb)\subset \Hull(\Lambda(P))$. This implies that for $n$ large enough, $\gamma_n\ell([t_n,t_n+3D))\subset X_P$. Therefore we have $\ell (t_n)\in \gamma_n^{-1}X_P = X_{\gamma_n^{-1}P\gamma_n}$ and $\ell ([t_n,t_n+2D))\subset \gamma_n^{-1}X_P = X_{\gamma_n^{-1}P\gamma_n}$, which contradicts to the assumptions on $t_n$.
\end{proof}

One may also compare the following theorem with the results above.

\begin{theorem}\cite[Corollary 1.14]{DS05}
    If $(\Gamma,\Pc)$ and $(P,\Hc_P)$ for each $P\in\Pc$ are relatively hyperbolic, then $(\Gamma,\bigcup_{P\in\Pc}\Hc_P)$ is relatively hyperbolic.
\end{theorem}

\subsection{Proof of Theorem \ref{RelativeAnosovExtension} and a direct example}

we assume that $(\Gamma,\Pc)$ is a relatively hyperbolic group and for each $P\in \Pc$, $(P,\Hc_P)$ is a relatively hyperbolic group. $\Hc = \cup_{P\in \Pc} \Hc_P$. Let $\tau: \partial(\Gamma,\Hc)\to \partial(\Gamma,\Pc)$ be the quotient map given by Theorem \ref{RelativeBoundaryExtension}. We restate and prove Theorem \ref{RelativeAnosovExtension}.

\begin{theorem}\label{ExtendingRepresentationsTheorem}
    Let $\rho:\Gamma\to \SL(d,\Kb)$ be a $P_k$-divergent representation. Consider the following conditions,
    \begin{itemize}
        \item[(1)] $\rho$ is extended geometrically finite relative to $\Hc$;
        \item[(1')] $\rho$ is extended geometrically finite relative to $\Pc$ and for each $P\in \Pc$, $\rho|_P$ is extended geometrically finite relative to $\Hc_P$;
        \item[(2)] $\rho$ is Anosov relative to $\Hc$;
        \item[(2')] $\rho$ is extended geometrically finite relative to $\Pc$ and for each $P\in \Pc$, $\rho|_P$ is Anosov relative to $\Hc_P$.
    \end{itemize}
    Then (1) is equivalent to (1'), and (2) is equivalent to (2').
\end{theorem}
\begin{proof}
    Since $\rho$ is $P_k$-divergent, let $\Lc\subset \Fc_{k,d-k}$ be the limit set of $\rho$, and for each $P\in\Pc^\Gamma$, let $\Lc_P\subset \Lc$ be the limit set of $\rho|_P$.
    \\ \\
    \textit{(1)$\Rightarrow$(1') and (2)$\Rightarrow$(2').} We assume that (1) holds. Since $\rho$ is $P_k$-divergent, it is extended geometrically finite with a boundary extension $\zeta:\Lc\to \partial(\Gamma,\Hc)$ that has the properties described in Proposition \ref{RefinedtoMinimal}. Then $\tau\circ\zeta:\Lc\to \partial(\Gamma,\Pc)$ is a transverse boundary extension of $(\Gamma,\Pc)$. Let $(\gamma_n)_{n\in\Nb}\subset \Gamma$ be a sequence, then for any subsequence, there exists a further subsequence, such that $\gamma_n\to p \in \partial(\Gamma,\Hc)$ and $\gamma_n^{-1}\to q \in \partial(\Gamma,\Hc)$ as $n\to +\infty$ with respect to the convergence group action of $\Gamma$ on $\partial(\Gamma,\Hc)$. Then $\gamma_n\to \tau(p)$ and $\gamma_n^{-1}\to \tau(q) $ as $n\to +\infty$ with respect to the convergence group action of $\Gamma$ on $\partial(\Gamma,\Pc)$. For any open set $U\subset \Fc_{k,d-k}$ containing $(\tau\circ\zeta)^{-1}(\tau(p)) \supset \zeta^{-1}(p)$ and any compact set $L\subset \Fc_{k,d-k}$ transverse to $(\tau\circ\zeta)^{-1}(\tau(q))$ and hence transverse to $\zeta^{-1}(q)$, $\rho(\gamma_n) L$ contains in $U$ when $n$ is large enough. Therefore, by Remark \ref{OpenSettobeWholeTransverse}, $\rho$ is extended geometrically finite with boundary extension $\tau\circ\zeta:\Lc\to \partial(\Gamma,\Pc)$.
    
    For each $P\in\Pc$, since $\partial(P,\Hc_P)$ is identified as a subset of $\partial(\Gamma,\Hc)$, it is direct to deduce that $\rho|_P$ is extended geometrically finite with boundary extension $\zeta|_{\Lc_P}:\Lc_P\to \partial(P,\Hc_P)$, which completes the proof of (1)$\Rightarrow$(1'). In particular, when $\tau\circ\zeta:\Lc\to \partial(\Gamma,\Pc)$ is a homeomorphism, than for each $P\in\Pc$, $\zeta|_{\Lc_P}:\Lc_P\to \partial(P,\Hc_P)$ is also a homeomorphism, which completes the proof of (2)$\Rightarrow$(2').
    \\ \\
    \textit{(1')$\Rightarrow$(1) and (2')$\Rightarrow$(2).} If (1') holds, let $\widetilde{\zeta}:\Lc \to \partial(\Gamma,\Pc)$, $\zeta_P:\Lc_P \to \partial(P,\Hc_P)$ be their boundary extensions for extended geometrical finiteness with the properties described in Proposition \ref{RefinedtoMinimal}. 
    We define a $\Gamma$-equivariant surjective map $\zeta:\Lc \to \partial(\Gamma,\Hc)$ by the following way.
    \begin{itemize}
        \item If $p\in \partial(P,\Hc_P)\subset \partial(\Gamma,\Hc)$ for some $P\in\Pc$, then $\zeta^{-1}(p) = \zeta_P^{-1}(p)\subset \Lc_P \subset \Lc$;
        \item If $p\in \partial(\Gamma,\Hc)$ and there exists $\gamma\in\Gamma$ and $P\in\Pc$ such that $\gamma^{-1}p\in \partial(P,\Hc_P)$, then $\zeta^{-1}(p) = \rho(\gamma)\zeta^{-1}(\gamma^{-1} p)$;
        \item If $p \in \partial(\Gamma,\Hc)$ such that $\tau(p)\in \partial_{con}(\Gamma,\Pc)$, then $\zeta^{-1}(p) = \widetilde{\zeta}^{-1}(\tau(p))$.
    \end{itemize}
    By Proposition \ref{RefinedtoMinimal}, the conditions above define $\zeta$ everywhere on $\Lc$. Actually, $\zeta$ is the unique $\rho$-equivariant map such that $\widetilde{\zeta} = \tau\circ\zeta$ and for each $P\in\Pc$, $\zeta|_P = \zeta_P$.
    
    The map $\zeta:\Lc \to \partial(\Gamma,\Hc)$ is transverse. Actually, let $p\ne q\in \partial(\Gamma,\Hc)$ be two points. If $\tau(p)\ne \tau(q)$, then $\zeta^{-1}(p)\subset \widetilde{\zeta}^{-1}(\tau(p))$ and $\zeta^{-1}(q)\subset \widetilde{\zeta}^{-1}(\tau(q))$ are transverse as $\widetilde{\zeta}$ is transverse. If $\tau(p)=\tau(q)$, there exists $\gamma\in\Gamma$ and $P\in\Pc$ such that $\gamma^{-1} p \ne \gamma^{-1} q\in \partial(P,\Hc_P)$, then $\zeta^{-1}(p) = \rho(\gamma)\zeta^{-1}_P(\gamma^{-1} p)$ and $\zeta^{-1}(q) = \rho(\gamma)\zeta^{-1}_P(\gamma^{-1} q)$ are transverse as $\zeta_P$ is transverse.

    Then we show that $\zeta:\Lc \to \partial(\Gamma,\Hc)$ is continuous. We argue by contradiction. Let $(x_n)_{n\in \Nb}$ be a sequence of flags in $\Lc$ with $x_n\to x\in \Lc$ as $n\to +\infty$. We denote $p_n = \zeta(x_n) \in \partial(\Gamma,\Hc)$ and $q = \zeta(x)\in \partial(\Gamma,\Hc)$. We assume that $p_n\to p\in \partial(\Gamma,\Hc)$ as $n\to +\infty$ and $p\ne q$. Since $\widetilde{\zeta}$ is continuous, $\tau(p_i) = \tau(\zeta(x_n))) = \widetilde{\zeta}(x_n) \to \widetilde{\zeta}(x) = \tau(q)$,  as $n\to +\infty$, and since $\tau$ is continuous, $\tau(p_i)\to \tau(p)$ as $n\to +\infty$. Then we have $\tau(p) = \tau(q)$, which means $p,q$ are contained in the limit set of the same parabolic subgroup. Up to a transformation by the $\Gamma$-action, we may assume that $p,q\in\partial(P,\Hc_P) \subset \partial(\Gamma,\Hc)$ for some $P\in\Pc$. If $p_n$ is contained in $\partial(P,\Hc_P)$ for $n$ large enough, the by the continuity of $\zeta_P$, we have $p=q$, which contradicts our assumption. Hence we assume that, up to a subsequence, $p_n\not\in \partial(P,\Hc_P)$, i.e., $\tau(p_n)\ne \tau(p)$ for all $n\in\Nb$. Since $\tau(p)$ is a bounded parabolic point of $\partial(\Gamma,\Pc)$, $P$ acts cocompactly on $\partial(\Gamma,\Pc)\setminus \{\tau(p)\}$. Then there exists a compact set $K_0\subset \partial(\Gamma,\Pc)\setminus \{\tau(p)\}$ and a sequence $(\gamma_n)_{n\in\Nb}\subset P$, such that $\gamma_n^{-1}\tau(p_n)\in K_0$ for all $n\in\Nb$.

    Since $\tau$ is continuous and $\partial(\Gamma,\Hc)$ is compact, then $\tau^{-1}(K_0)\subset \partial(\Gamma,\Hc)\setminus \partial(P,\Hc_P)$ is compact. Now we have that $\gamma_n p_n \in \tau^{-1}(K_0)$ and $p_n\to p$ as $n\to +\infty$. Since $P$ acts on $\partial(\Gamma,\Hc)$ as a convergence group and the limit set of $P$ is precisely $\partial(P,\Hc_P)$, up to a subsequence, there exists points $a,b \in \partial(P,\Hc_P)$, such that for any compact set $K\subset \partial(\Gamma,\Hc)\setminus \{a\}$, $\gamma_n c\to a$ as $n\to +\infty$ uniformly for any $c\in K$. Apply it for the case $K = \tau^{-1}(K_0)$, then we must have $\gamma_n\to a = p$.

    On the other hand, $\widetilde{\zeta}^{-1}(K_0)\subset \Lc\setminus \{\widetilde{\zeta}^{-1}(\tau(p))\} = \Lc\setminus \Lc_P$ is compact and $\rho(\gamma_n^{-1})x_n\in \widetilde{\zeta}^{-1}(K_0)$. Since $\widetilde{\zeta}$ is transverse, then $\widetilde{\zeta}^{-1}(K_0)$ is transverse to $\Lc_P$.
    We denote \[y = \{ \lim\limits_{n\to +\infty} (U_k(\rho(\gamma_n)),U_{d-k}(\rho(\gamma_n)))\text{ and }y'= \{ \lim\limits_{n\to +\infty} (U_k(\rho(\gamma_n^{-1})),U_{d-k}(\rho(\gamma_n^{-1})))~,\] which are contained in $\Lc_P$. By Lemma \ref{DivergenceLemma}, we have $\rho(\gamma_n) y'' \to y $ as $n\to +\infty$ uniformly for any $y''\in \widetilde{\zeta}^{-1}(K_0)$ as $\widetilde{\zeta}^{-1}(K_0)$ is transverse to $y'$. Then $\rho(\gamma_n)\rho(\gamma_n^{-1})x_n = x_n \to y = x$ as $n\to +\infty$. Since $\zeta_P$ extends the convergence dynamics as described in the definition of extended geometrically finite representation, we have $q = \zeta(x) = p$, which derives a contradiction. Therefore, $\zeta$ is continuous.

    Now $\zeta: \Lc \to \partial(\Gamma,\Hc)$ is a transverse boundary extension. It is easy to see that $\zeta$ extends the convergence dynamics by Remark \ref{OpenSettobeWholeTransverse}. Therefore, $\rho$ is extended geometrically finite with boundary extension $\zeta: \Lc \to \partial(\Gamma,\Hc)$, which completes the proof of (1')$\Rightarrow$(1).
    When $\zeta_P$ is a homeomorphism for each $P\in\Pc$, by our construction, $\zeta$ is also a homeomrohpism, which implies (2')$\Rightarrow$(2).
\end{proof}

Now we provides a direct way to construct representations that are divergent, extended geometrically finite and not relatively Anosov.

\begin{example}\label{DirectExampleofDirectSum}
    Let $\Sigma = \Sigma_{g,b}$ be an oriented surface of genus $g$ with $b$ punctures and empty boundary. We mark the punctures by the set of indices $I=\{1,2,...,b\}$. Let $\Gamma$ be the fundamental group of $\Sigma$, which is a free group when $b>0$, and let $c_i\in\Gamma$ be a representative of the homotopy class of a small simple closed curve around the $i$-th puncture for each $i\in I$.
    
    For any subset $I_0\subset I$, a hyperbolic structure on $\Sigma$ with the $i$-th puncture a cusp for each $i\in I_0$ and with the $i'$-th punture a funnel for each $i'\in I\setminus I_0$ provides a geometrically finite representation $\rho_{I_0}:\Gamma\to \SL(2,\Rb)$. Let $\Pc_{I_0} = \{ \langle c_i \rangle : i\in I_0 \}$. Then by \cite[Proposition 1.7]{ZZ2}, $(\Gamma,\Pc_{I_0})$ is relatively hyperbolic and $\rho_{I_0}$ is $P_1$-Anosov relative to $\Pc_{I_0}$.
    
    For $j=1,2$, we consider the follow data. Let $I_j$ be a subset of $I$ and let $\rho_j = \rho_{I_j}$ be a representation as described above. Since $\rho_j$ is $P_1$-Anosov relative to $\Pc_{I_j}$, there is a unique continuous limit map $\xi_j : \partial(\Gamma,\Pc_{I_0})\to \Lc(\rho_j)$ that is transverse and homeomorphic, where $\Lc(\rho_j)$ is the limit set of $\rho_j$ in $\Pb(\Rb^2)$. Let $\tau_j : \partial(\Gamma, \Pc_{I_j})\to \partial(\Gamma,\Pc)$ be the quotient map given by Theorem \ref{RelativeBoundaryExtension}.

    We consider the set \[\Lambda = \{x_1\oplus x_2\in \Rb^2 \oplus \Rb^2 : x_1\in \Lc(\rho_1), x_2\in \Lc(\rho_2)\text{ and }\tau_1(\xi_1^{-1}(x_1))= \tau_2(\xi_2^{-1}(x_2))\}\]
    and the boundary extension in $\Pb(\Rb^2)$ given by
    \begin{eqnarray*}
    \zeta: \Lambda &\to & \partial(\Gamma,\Pc) \\
    x_1\oplus x_2 &\mapsto& \tau_1(\xi_1^{-1}(x_1)) ~.
    \end{eqnarray*}
    We claim that $\rho = \rho_1\oplus \rho_2:\Gamma \to \SL(4,\Rb)$ is $P_2$-divergent and extended geometrically finite with boundary extension $\zeta:\Lambda\to \partial(\Gamma,\Pc)$. Actually, let $X$ be a Gromov model of $(\Gamma,\Pc)$, then $\rho_j$ is $P_1$-Anosov in restriction to $\Fc(\Lc_j,\tau_j\circ\xi_j^{-1},X)$ for $j=1,2$.
    We consider the surjective, flow equivariant, continuous map between $\Gamma$-flows
    \begin{eqnarray*}
    \iota_j: \Fc(\Lambda, \zeta, X) &\to & \Fc(\Lc_j,\tau_j\circ\xi_j^{-1},X) \\
    (x_1\oplus x_2,y_1\oplus y_2,t) &\mapsto& (x_j,y_j,t)~.
    \end{eqnarray*}
    By pulling-back the dominated splitting and the metric through $\iota_j$, we deduce that $\rho_j$ is $P_1$-Anosov in restriction to $\Fc(\Lambda, \zeta, X)$ by Theorem \ref{ExtendingRepresentationsTheorem} (2)$\Rightarrow$(2') and Theorem \ref{TheoremEGFtoRestricted}. The direct sum of these two dominated splitting for $\rho_1$ and $\rho_2$ is a dominated splitting of rank $2$ by Remark \ref{UnitVolumeMetric}. Therefore, $\rho$ is $P_2$-divergent and extended geometrically finite with boundary extension $\zeta:\Lambda\to \partial(\Gamma,\Pc)$ by Theorem \ref{TheoremRestrictedtoEGF}.
    
    For each $i\in I$, $\rho(c_i) = \diag(\rho_1(c_i),\rho_2(c_i))$ while $\rho_j(c_i)$ is parabolic (hyperbolic, respectively) if $i\in I_j$ ($i\in I\setminus I_j$, respectively). Then $\rho$ is $P_2$-Anosov relative to $\Pc$ only when $I_1 = I_2 = I$.

    The construction also works for the direct sum of more of such representations $\rho_j$.
\end{example}

\subsection{An example from simple Anosov representations}\label{SectionExampleOfTW}
The notion of simple Anosov representation of closed surface groups, introduced by \cite{TW23}, provides a domain of discontinuity of the mapping class group action on the character variety that is strictly larger than the domain of Anosov representations. The strictly larger part was proved by constructing a simple Anosov representation, which is not Anosov. We recall the construction of this example, and show that this representation is moreover divergent and extended geometrically finite, but not relatively Anosov.

Let $\Sigma$ be a closed hyperbolic surface. Let $p: \widetilde{\Sigma}\to \Sigma$ be a Galois covering of degree $d$. Let $c$ be a simple closed geodesic in $\widetilde{\Sigma}$ such that $p(c)$ has self-intersection. We denote by $\Gamma'$ the fundamental group of $\widetilde{\Sigma}$ and $\Gamma$ the fundamental group of $\Sigma$, then $\Gamma'$ is automatically a normal subgroup of $\Gamma$ of index $d$. Let $\gamma\in\Gamma'$ be the element that represents the homotopy class of $c$.

By applying the Maskit combination theorem (see \cite[Proposition 4.3]{TW23}), there exists a geometrically finite representation $\rho_1: \Gamma' \to \SL(2,\Cb)$ where the parabolic subgroups are exactly the $\Gamma'$-conjugates of $\langle \gamma \rangle$. Then $(\Gamma',\{\langle \gamma \rangle\})$ is a relatively hyperbolic group and $\rho_1$ is $P_1$-Anosov relative to $\{\langle \gamma \rangle\}$ by \cite[Proposition 1.7]{ZZ2}.

Let $\{\id = \alpha_1, \alpha_2,...,\alpha_d\}$ be a collection of representatives of the left cosets of $\Gamma'$ in $\Gamma$. Let $\{c = c_1, c_2,...,c_d\}$ be the collection of lifts of $p(c)$ in $\widetilde{\Sigma}$ such that $c_i$ has homotopy class $ \alpha_i\gamma \alpha_i^{-1}$, for each $i=1,2,...,d$.

Since $\langle \alpha_i\gamma \alpha_i^{-1} \rangle \cong \Zb$ is undistorted, $\Pc = \{\langle \alpha_i\gamma \alpha_i^{-1} \rangle : i=1,2,...,d\}$ is almost malnormal. By Proposition \ref{UndistortedtoRelativelyQuasiConvex} and Theorem \ref{RelativeBoundaryExtension}, $(\Gamma', \Pc)$ and $(\Gamma,\Pc)$ are relatively hyperbolic.

Since $\Gamma'$ is a normal subgroup of $\Gamma$, the conjugation by $\alpha_i$ is an automorphism on $\Gamma'$. In particular, it induces a homeomorphism $\partial(\Gamma',\{\langle \gamma \rangle\}) \to \partial(\Gamma',\{\langle \alpha_i\gamma \alpha_i^{-1} \rangle\}), q\mapsto \alpha_i q$. The notation has no ambiguity by the following diagram,
\[\begin{matrix}
    \partial\Gamma\cong\partial\Gamma' &\stackrel{\alpha_i}{\longrightarrow} & \partial\Gamma\cong\partial\Gamma' \\
    \downarrow \widetilde{\tau}_1 & & \downarrow \widetilde{\tau}_i\\
    \partial(\Gamma',\{\langle \gamma \rangle\}) &\stackrel{\alpha_i}{\longrightarrow} & \partial(\Gamma',\{\langle \alpha_i\gamma \alpha_i^{-1} \rangle\})\\
    \downarrow \tau_1 & & \downarrow \tau_i \\
    \partial(\Gamma,\Pc)\cong\partial(\Gamma',\Pc) &\stackrel{\alpha_i}{\longrightarrow} & \partial(\Gamma,\Pc)\cong\partial(\Gamma',\Pc)
\end{matrix}\]
where the four vertical arrows $\widetilde{\tau}_i,\tau_i$, denote the projections given by Theorem \ref{RelativeBoundaryExtension} and the top and the bottom arrows denote the natural $\Gamma$ action on $\partial\Gamma$ and $\partial(\Gamma,\Pc)$ respectively. We also denote $\tau : \partial \Gamma \to \partial(\Gamma,\Pc)$ the projection given by Theorem \ref{RelativeBoundaryExtension}.

Since $\rho_1:\Gamma'\to \SL(2,\Cb)$ is $P_1$-Anosov relative to $\{\langle \gamma \rangle\}$, let $\Lc_0$ denote the limit set of $\rho_1$ in $\Pb(\Cb^2)$, then the limit map $\xi_1: \partial (\Gamma',\{\langle \gamma \rangle\}) \to \Lc_0$ is a $\Gamma$-equivariant homeomorphism. By Theorem \ref{ExtendingRepresentationsTheorem}, $\rho_1$ is extended geometrically finite relative to $\Pc$, with the boundary extension $\zeta_1 =\tau_1\circ \xi_1^{-1} : \Lc_0 \to \partial(\Gamma',\Pc)$.

Following from the diagram above, $\rho_i (\cdot) =\rho_1(\alpha_i^{-1}\cdot \alpha_i)$ is $P_1$-Anosov relative to $\{\langle \alpha_i\gamma \alpha_i^{-1} \rangle\}$ for each $i=1,2,...,d$, with the limit map given by 
\begin{eqnarray*}
\xi_i: \partial(\Gamma',\{\langle \alpha_i\gamma \alpha_i^{-1} \rangle\}) & \to & \Lc_0\\
q & \mapsto & \xi_1(\alpha_i^{-1}q)~.
\end{eqnarray*}
Therefore, by Theorem \ref{ExtendingRepresentationsTheorem} again, $\rho_1$ is extended geometrically finite relative to $\Pc$, with the boundary extension $\zeta_i = \tau_i \circ \xi_i^{-1}: \Lc_0 \to \partial(\Gamma',\Pc)$.

Let $X$ be a Gromov model of $(\Gamma,\Pc)$ and hence automatically a Gromov model of $(\Gamma',\Pc)$.
There is a natural flow equivariant homeomorphism
\begin{eqnarray*}
\iota_i:\Fc(\Lc_0,\zeta_i,X) &\to & \Fc(\Lc_0,\zeta_1,X) \\
(z^+,z^-,t) &\mapsto& (\xi_1(\alpha_i^{-1}\xi_i^{-1}(z^+)),\xi_1(\alpha_i^{-1}\xi_i^{-1}(z^-)),t)~.
\end{eqnarray*}
One remark here is that if we only regard $\Fc(\Lc_0,\zeta_i,X)$ as the topological space $(\Lambda_,\xi_i)^{(2)}$, we will see that $(\Lambda_,\xi_i)^{(2)} = (\Lambda_,\xi_1)^{(2)}$ for each $i=1,2,...,d$ and $\iota_i$ is the identity map.

By Theorem \ref{TheoremEGFtoRestricted}, $\rho_1$ is $P_1$-Anosov in restriction to $\Fc(\Lc_0,\zeta_1,X)$, with the dominated splitting 
\[\Fc(\Lc_0,\zeta_1,X)\times \Cb^2 = E^s_1 \oplus E^u_1\]
given by 
$((E^s_1)_z,(E^u_1)_z) = (z^+,z^-)\subset \Lc_0^2 \subset \Pb(\Cb^2)^2$ for any $z=(z^+,z^-,t)\in \Fc(\Lc_0,\zeta_1,X)$, with respect to a metric $\norm{\cdot}$ of unit volume. Then for each $i=1,2,...,d$, $\rho_i$ is $P_1$-Anosov in restriction to $\Fc(\Lc_0,\zeta_i,X)$, with the dominated splitting
\[\Fc(\Lc_0,\zeta_i,X)\times \Cb^2 = E^s_i \oplus E^u_i\]
given by the pull-back of the dominated splitting over $\Fc(\Lc_0,\zeta_1,X)$ through $\iota_i$, with respect to the pull-back metric $\iota_i^*\norm{\cdot}$.

Let $\rho = \mathrm{Ind}^{\Gamma}_{\Gamma'}(\rho_1): \Gamma \to \SL(2d,\Cb)$ be the induced representation. More precisely, we may regard $\Cb^{2d} = \oplus_{i=1}^d \alpha_i \Cb^2 $ as a direct sum of $d$ copies of $\Cb^2$, with each of them denoted by $\alpha_i \Cb^2$. Recall that the induced representation is defined by the following rule. For each $\eta\in\Gamma$ and each $\alpha_i v\in \alpha_i\Cb^2$, $\rho(\eta)\alpha_i v = \alpha_j (\rho_1(\eta') v)$, where $\eta'\in\Gamma'$ is such that $\eta \alpha_i = \alpha_j \eta'$. In particular, when we restrict $\rho$ on $\Gamma'$, we have $\rho|_{\Gamma'} = \oplus_{i=1}^{d} \rho_i$.

Consider the surjective, flow equivariant, continuous map,
\begin{eqnarray*}
\widetilde{\iota}_i:\Fc(\partial\Gamma, \tau ,X) &\to & \Fc(\Lc_0,\zeta_i,X) \\
(z^+,z^-,t) &\mapsto& (\xi_1(\alpha_i^{-1}\widetilde{\tau}_i(z^+)),\alpha_i^{-1}\widetilde{\tau}_i(z^-)),t)~.
\end{eqnarray*}
Then the dominated splitting of rank $d$, $\Fc(\Lc_0,\zeta_i,X)\times \Cb^2 = E^s_i \oplus E^u_i$ with respect to the metric $\iota_i^*\norm{\cdot}$ pulls-back to a dominated splitting of $\Fc(\partial\Gamma, \tau ,X)\times \Cb^2 = \widetilde{\iota}_i^* E^s_i \oplus \widetilde{\iota}_i^* E^u_i$ with respect to the metric $\widetilde{\iota}_i^* \iota_i^*\norm{\cdot}$. 

By Remark \ref{UnitVolumeMetric}, their direct sum for all $i=1,2,...,d$ gives a dominated splitting \[\Fc(\Lc_0,\zeta_i,X)\times \Cb^{2d} =  (\oplus_{i=1}^d \widetilde{\iota}_i^* E^s_i) \oplus (\oplus_{i=1}^d \widetilde{\iota}_i^* E^u_i)\] with respect to the metric $\oplus_{i=1}^d \widetilde{\iota}_i^* \iota_i^* \norm{\cdot}$, for the representation $\rho|_{\Gamma'}$. It is easy to see that this is moreover a dominated splitting for $\rho$ and $\oplus_{i=1}^d \widetilde{\iota}_i^* \iota_i^* \norm{\cdot}$ is actually $\rho(\Gamma)$-invariant by the definition of induced representations. Let \[\Lambda = \{x = \oplus_{i=1}^d x_i : x_i\in\Lc_0 \text{ for each }i=1,2,...,d~,\] \[ \text{ and } \zeta_1(x_1) = \zeta_2(x_2) = ... = \zeta_d(x_d) \} \subset \Cb^{2d} \]
and let
\begin{eqnarray*}
\xi:\partial\Gamma &\to & \Lambda \\
q &\mapsto& \oplus_{i=1}^d \xi_i(\widetilde{\tau}_i(q))~.
\end{eqnarray*}
is a $\rho(\Gamma)$-invariant closed subset of $\Gr_d(\Cb^{2d})$. $\xi$ is a $\Gamma$-equivariant homeomorphism since for $q\ne q'\in\partial \Gamma$, if $\xi_1(\widetilde{\tau}_1(q))=\xi_1(\widetilde{\tau}_1(q'))$, then $\widetilde{\tau}_1(q)=\widetilde{\tau}_1(q')$ as the same parabolic point in $ \partial(\Gamma',\{\langle \gamma \rangle\})$, and hence $\xi_2(\widetilde{\tau}_2(q)) \ne \xi_2(\widetilde{\tau}_2(q'))$.

Finally, we conclude that $\rho$ is $P_k$-divergent and extended geometrically finite with the boundary extension $\tau\circ\xi^{-1} :\Lambda \to \partial(\Gamma,\Pc)$ by Theorem \ref{TheoremRestrictedtoEGF}. The limit map $\xi:\partial\Gamma \to \Lambda$ is a homeomorphism, which is not transverse, while the boundary extension $\tau\circ\xi^{-1} :\Lambda \to \partial(\Gamma,\Pc)$ is transverse. Actually, for the parabolic point $p_j\in \partial(\Gamma,\Pc)$ fixed by $\langle \alpha_j\gamma \alpha_j^{-1} \rangle $ and two distinct flags $\oplus_{i=1}^d x_i, \oplus_{i=1}^d y_i \in (\tau\circ\xi^{-1})^{-1}(p_j) \subset \Lambda$, $\oplus_{1\leqslant i \leqslant d, i\ne j}^d x_i$ and $\oplus_{1\leqslant i\leqslant d, i\ne j}^d y_i$ are transverse but $x_j = y_j$.

One should also notice that $\rho$ is not $P_d$-Anosov relative to $\Pc$, since $\rho(\alpha_i\gamma\alpha_i^{-1})$ is a blockwisely diagonal matrix with $1$ of them parabolic and the rest $d-1$ of them hyperbolic.
\bibliographystyle{alpha}
\bibliography{reference.bib}

\end{document}